\documentclass[reqno, intlimits, sumlimits]{amsart} 

\usepackage{amsmath, amstext, amssymb, amsthm, bbm,lmodern, enumerate, esint, todonotes}
\usepackage[pdftex]{hyperref}
\usepackage{tikz}
\usepackage{subcaption}
\usepackage{graphicx}
\usepackage[english]{babel}
\usepackage[utf8]{inputenc}     %Umlaute
\usepackage[toc,page]{appendix}
\usepackage{fancyhdr}
\DeclareFontFamily{U}{mathx}{}
\DeclareFontShape{U}{mathx}{m}{n}{<-> mathx10}{}
\DeclareSymbolFont{mathx}{U}{mathx}{m}{n}
\DeclareMathAccent{\widehat}{0}{mathx}{"70}
\DeclareMathAccent{\widecheck}{0}{mathx}{"71}

\usepackage[ right= 2cm, left=2cm, a4paper,top=2.7cm, bottom= 3cm]{geometry}

\usepackage{mathrsfs}

\newcommand{\E}{\mathbb {E}}   %Erwartungswert
\newcommand{\R}{\mathbb {R}}   %Reelle Zahlen
\newcommand{\Z}{\mathbb {Z}}   %Ganze Zahlen
\newcommand{\N}{\mathbb {N}}   %Natuerliche Zahlen
\newcommand{\ind}{\mathbbm{1}} %Indikatorfunktion
\newcommand{\AAA}{\mathcal {A}}  %\sigma-Algebra
\newcommand{\FFF}{\mathcal {F}} %Algebra
\newcommand{\EEE}{\mathcal {E}}

\newcommand{\CCC}{\mathcal {C}}
\newcommand{\PPP}{\mathcal {P}}
\newcommand{\MMM}{\mathcal {M}}
\newcommand{\BBB}{\mathcal {B}} %Borel-\sigma-Algebra

\newcommand{\dP}{\mathbf{P}}
\newcommand{\I}{\mathrm{I}}
\newcommand{\dR}{\mathbf{R}}
\newcommand{\dt}{\frac{\mathrm{d}}{\mathrm{dt}}}
\newcommand{\m}{{\lambda}}
\newcommand{\A}{K}

\newcommand{{\nuV}}{{\mathbf{V}}}
\newcommand{{\muP}}{{{\mathbf{P}}}}

\def\Poi {{\mathbf{Poi}}}  %Poisson-Verteilung
\numberwithin{equation}{section}

\theoremstyle{plain}
\newtheorem{satz}{Theorem}[section]

\newtheorem{kor}[satz]{Corollary}
\newtheorem{lem}[satz]{Lemma}
\newtheorem{prop}[satz]{Proposition}

\theoremstyle{remark}
\newtheorem{bem}[satz]{Remark}

\newtheorem{dfn}[satz]{Definition}

\usepackage{blindtext}
\usepackage{enumitem}

\begin{document}

\author{Martin Huesmann}
\address{Martin Huesmann: Institute for Mathematical Stochastics,
University of M\"unster
Orl\'eans-Ring 10,
48149 M\"unster, Germany}
\email{martin.huesmann@uni-muenster.de}

\author{Hanna Stange}
\address{Hanna Stange: Institute for Mathematical Stochastics,
University of M\"unster
Orl\'eans-Ring 10,
48149 M\"unster, Germany}
\email{hanna.stange@uni-muenster.de}
\thanks{MH and HS are supported by the Deutsche Forschungsgemeinschaft (DFG, German Research Foundation) through the SPP 2265 {\it Random Geometric Systems}. MH and HS have been funded by the Deutsche Forschungsgemeinschaft (DFG, German Research Foundation) under Germany's Excellence Strategy EXC 2044 -390685587, Mathematics M\"unster: Dynamics--Geometry--Structure . }

\title{
Non-local Wasserstein Geometry, Gradient Flows, and Functional Inequalities for Stationary Point Processes}
\date{}

\maketitle
\begin{abstract}

We construct a non-local Benamou--Brenier-type transport distance on the space of stationary point processes and analyse the induced geometry. We show that our metric is a specific variant of the transport distance recently constructed in \cite{schiavo2023wasserstein}. As a consequence, we show that the Ornstein--Uhlenbeck semigroup is the gradient flow of the specific relative entropy w.r.t.\ the newly constructed distance. Furthermore, we show the existence of stationary geodesics, establish $1$-geodesic convexity of the specific relative entropy,  and derive stationary analogues of functional inequalities such as a specific HWI inequality and a specific Talagrand inequality.
One of the key technical contributions is the existence of solutions to the non-local continuity equation between arbitrary point processes.

\end{abstract}
\tableofcontents
\section{Introduction}

The theory of optimal transport is a powerful tool to analyse the long time behaviour of dynamical systems, to derive functional inequalities, and to investigate the geometry of metric measure spaces. Often the key is to establish or leverage convexity properties of a functional $\mathcal F$ on probability measures w.r.t.\ a transport metric. Prominent and influential examples are the identification of Fokker--Planck equations as gradient flows of the free energy in \cite{JKO}, the connection of HWI, Talagrand and log Sobolev inequalities in \cite{OtVi00}, or the equivalence  of Ricci curvature bounds and convexity estimates of the entropy in \cite{Ricciequivalence, St06, LottVillani}. We refer to, e.g.\ \cite{villani2008optimal, Otapplied, AmBrSe24}, for a comprehensive introduction to the theory.

Recently, several articles extended parts of this theory to infinite particle systems in infinite volume. While \cite{ErHu15, DeSu21, DeSu22, Su22} consider Wasserstein geometries and gradient flows w.r.t.\ usual Wasserstein-type distances, in \cite{MartinJonasBasti} and \cite{HeLe25} Wasserstein distances per volume were introduced to analyse diffusive evolutions of stationary point processes and spin systems, respectively.

In this paper, based on \cite{schiavo2023wasserstein}, we construct a non-local transport distance per volume in the form of a Benamou--Brenier-type formula on the set of stationary point processes and establish
 \begin{itemize}
     \item that the gradient flow of the specific relative entropy w.r.t.\ the newly constructed distance is the Ornstein--Uhlenbeck semigroup,
     \item  that the specific relative entropy is $ 1$-convex, and
     \item the validity of functional inequalities such as a specific Talagrand inequality and a specific HWI inequality.
 \end{itemize}

\subsection{Setting and Main Results}
For a Polish set $B\subseteq \R^d$ let $\BBB(B)$ be the Borel subsets of $B$ and $\BBB_0(B)$ those that are additionally bounded. The space of locally finite and simple counting measures on $B\in\mathcal B(\R^d)$ will be denoted by $\Gamma_B$. A {point process} (sometimes also called random point field) $\eta$ on a Polish set $B \in \mathcal{B}(\mathbb{R}^d)$ is a random variable taking values in  $\Gamma_B$.  The distribution of a point process is an element of $\PPP(\Gamma_B)$ the set of probability measures over $\Gamma_B.$ We call $\dP\in\PPP(\Gamma_B)$ locally integrable if $\int_B\gamma(D)\dP(\mathrm{d}\gamma)<\infty$ for all  $D \in \mathcal{B}_0(B)$, and denote the set of locally integrable probability measures by $\PPP_1(\Gamma_B).$  For $z \in \mathbb{R}^d$, we define the shift operator $\Theta_z:\Gamma_{\R^d}\to\Gamma_{\R^d}$ by $\Theta_z \xi(D) = \xi(D + z),\ D \in \mathcal{B}(\mathbb{R}^d).$
A measure $\dP \in \mathcal{P}_1(\Gamma_{\mathbb{R}^d})$ is called {stationary} if it is invariant under shifts, i.e.\  $(\Theta_z)_\# \dP = \dP$ for all $z \in \mathbb{R}^d$, where $(\Theta_z)_\# \dP$ denotes the push forward measure of $\dP$ under $\Theta_z$. The set of all stationary measures is denoted by $\mathcal{P}_s(\Gamma_{\mathbb{R}^d})$ and a point process is called {stationary} if its distribution belongs to $\mathcal{P}_s(\Gamma_{\mathbb{R}^d})$.\\

The definition of the transport distance in this article is a specific version of the  metric $\mathcal W_0$, recently constructed in \cite{schiavo2023wasserstein}. In turn, analogous to the construction of transport distances for Markov chains \cite{maas2011gradientflowsentropyfinite} or jump processes \cite{erbar2012gradientflowsentropyjump}, $\mathcal W_0$ is inspired by the dynamical Benamou--Brenier formulation of the Wasserstein distance \cite{BeBr00}, which we briefly recall for motivation. For a more detailed presentation, we refer e.g.\ to \cite[Chapter 8]{Ambrosio2005GradientFI}.  For probability measures $p_0,p_1$ on $\R^d$, the $\mathrm{L}^2$ Wasserstein distance can be expressed by 
\begin{align*}
    \mathrm{W}_2^2(p_0,p_1)=\inf\left\{\int_0^1\int\|\bar v_t(x)\|^2\bar p_t(\mathrm{d}x)\mathrm{d}t:(\bar p,\bar v)\in\mathcal{CE}^\prime(p_0,p_1)\right\},
\end{align*}
where $\mathcal{CE}^\prime(p_0,p_1)$ is the set of sufficiently regular curves of probability measures $\bar p$ joining $p_0$ and $p_1$ and sufficiently regular $\bar v:[0,1]\times\R^d\to\R^d$ 
satisfying the continuity equation 
\begin{align*}
        \partial_t\bar p_t+\nabla\cdot(\bar v_tp_t)=0
\end{align*}
in the distributional sense. Adapting this formulation to the setting of (stationary) point processes requires to define an action functional to be minimized and to give a meaningful notion of a continuity equation. This was accomplished in \cite{schiavo2023wasserstein} as follows.
  
The role of the gradient in the continuity equation will be taken over by the discrete difference operator $ D F : \Gamma_B \times B \to \mathbb{R} $ of a functional $ F : \Gamma_B \to \mathbb{R} $, which  is defined by 
$$ D_z F(\xi) := F(\xi + \delta_z) - F(\xi).$$  A curve of  measures  $\Bar\dP=(\Bar{\dP}_t)_{t\in[0,1]}\subset\PPP_1(\Gamma_B)$ and a velocity field $\Bar{\nuV}=(\Bar{\nuV}_t)_{t\in[0,1]}\subset\MMM(\Gamma_B\times B)$, consisting of a curve  of signed measures on $\Gamma_B\times B,$ satisfies the continuity equation if 
$$ \int_0^1\dt \phi(t)\int e^{-\int f(x)\xi(\mathrm{d}x)}\Bar{\muP}_t(\mathrm{d}\xi)\mathrm{d}t+ \int_0^1 \phi(t)\int D_z e^{-\int f(x)\xi(\mathrm{d}x)}\Bar{\nuV}_t(\mathrm{d}\xi, \mathrm{d}z)\mathrm{d}t=0,$$
for all smooth and compactly supported functions
$\phi\in\CCC_c^\infty( ( 0 , 1 ))$ and continuous, non-negative, and compactly supported functions $f\in\CCC_c^+(B).$ The set of all sufficiently regular solutions joining $\dP_0$ to $\dP_1$ is denoted by $\mathcal{CE}^B(\dP_0,\dP_1)$  (see Definition \ref{dfn:ce} for a precise formulation). To cope with the lack of a chain rule in our discrete setup, we introduce the logarithmic mean $\theta:\R^+\times\R^+\to\R$, defined by 
$$\theta(x,y)=\frac{y-x}{\log(y)-\log(x)}.$$
Denote by $\lambda$ the $d$-dimensional Lebesgue measure and by $\Poi$ the distribution of a homogeneous Poisson point process with intensity 1. For a curve $(\Bar\dP,\Bar{\nuV})\in \mathcal{CE}^{\R^d}(\dP_0,\dP_1)$ s.t.\ $\Bar{\muP}_{t}=\rho_t\Poi$ and $\Bar{\nuV}_{t}=w_t(\Poi\otimes\lambda)$ for all $ t\in[0,1]$, we introduce the action functional 
$$\mathrm{A}(\dP,\nuV):=\int_0^1\int\int\frac{|w_t(\xi,z)|^2}{\theta(\rho_t(\xi),\rho_t(\xi+\delta_z))}\Poi(\mathrm{d}\xi)\lambda(\mathrm{d}z)\mathrm{d}t$$
and finally the metric $\mathcal W_0$ from \cite{schiavo2023wasserstein} is defined as
$$\mathcal W_0^2(\dP_0,\dP_1)=\inf\left\{A(\dP,\nuV) : (\Bar\dP,\Bar\nuV) \in \mathcal{CE}^{\R^d}(\dP_0,\dP_1)\right\}.$$
The metric $\mathcal W_0$ enjoys many interesting and desirable properties (see Section \ref{sec:distanceDSHS}). However, it cannot be applied to stationary point processes where one is interested in geodesics or interpolation curves $(\Bar\dP_t)_{t\in[0,1]}$ consisting of stationary measures. Since the only stationary measure $\Bar\dP_t\ll \Poi$ is the measure $\Poi$ itself, the action functional $\mathrm{A}$ would blow up. Even more severely, in \cite[Proposition 4.3]{schiavo2023wasserstein} 
the existence of solutions to the continuity equation was only shown under the assumption that $\dP_0,\dP_1\ll\Poi$
 which fails for stationary measures.

Our first contribution is to prove the existence of solutions to the continuity equation between arbitrary measures $\dP_0$ and $\dP_1$, even stationary solutions if $\dP_0,\dP_1\in\mathcal P_s(\Gamma_{\R^d})$. We denote by $\mathcal{CE}^s(\dP_0,\dP_1)\subset \mathcal{CE}^{\R^d}(\dP_0,\dP_1) $ the set of  all stationary solutions, i.e.\ solutions for which $\Bar{\dP}\subset\mathcal{P}_s(\Gamma_{\R^d})$ (cf.\ Definition \ref{dfn:stationarysolutions}).

\begin{satz}\label{in:cesolution}
    For any Polish set $B\in\BBB(\R^d)$ and any measures $\dP_0,\dP_1\in\PPP_1(\Gamma_B)$ there exists a solution $(\Bar{\dP},\Bar{\nuV})\in\mathcal{CE}^B(\dP_0,\dP_1)$. Further, if $\dP_0,\dP_1\in\PPP_s(\Gamma_{\R^d})$, there exists a stationary solution $(\Bar{\dP}^s,\Bar{\nuV}^s)\in\mathcal{CE}^s(\dP_0,\dP_1)$.
\end{satz}

To define a suitable non-local transport metric between stationary point processes, we need to find a good replacement for the action functional $\mathrm{A}$.
Since stationary measures $\dP\in\mathcal P_s(\Gamma)$ often have the property that $\dP_{|\Lambda}\ll \Poi_{|\Lambda}$ for bounded sets $\Lambda$, we follow the philosophy for the specific entropy (see below) or the specific Wasserstein distance from \cite{MartinJonasBasti} and \cite{HeLe25} and define an action per volume as follows. 

 Let $\Lambda_n:=[-n/2,n/2]^d$ and $(\Bar{\dP},\Bar{\nuV})\in\CCC\EEE^s(\dP_0,\dP_1).$ Denote by $\bar{\muP}_{t|\Lambda_n}$ the restriction of $\Bar\dP_t$ to $\Gamma_{\Lambda_n}$ and by $\Bar{\nuV}_{t|\Lambda_n}$  the restriction of $\Bar\nuV_t$ to $\Gamma_{\Lambda_n}\times\Lambda_n$. Let 
$\Bar{\muP}_{t|\Lambda_n}=\rho^n_t\Poi_{t|\Lambda_n}$ and $\Bar{\nuV}_{t|\Lambda_n}=w^n_t(\Poi_{|\Lambda_n}\otimes\lambda_{|\Lambda_n})$ for all $n\in\N,\ t\in[0,1]$, where $\lambda_{|\Lambda_n}$ is the $d$-dimensional Lebesgue measure restricted to $\Lambda_n$. We define the action functional
\begin{align}\label{eq:AsupA}
    \AAA({\Bar{\muP}},{\Bar{\nuV}}):=\sup_{n\in\N}\frac{1}{\lambda(\Lambda_n)} \int_0^1\int\int\frac{|w_t^n(\xi,z)|^2}{\theta(\rho^n_t(\xi),\rho_t^n(\xi+\delta_z))}\Poi_{|\Lambda_n}(\mathrm{d}\xi)\lambda_{|\Lambda_n}(\mathrm{d}z)\mathrm{d}t=\sup_{n\in\N}\frac{1}{\lambda(\Lambda_n)}\mathrm{A}({\Bar{\muP}_{|\Lambda_n}},{\Bar{\nuV}_{|\Lambda_n}}).
\end{align}
For $\dP_0,\dP_1\in\PPP_s(\Gamma),$ we then define the specific non-local transport distance for stationary measures as
\begin{align*}
    \mathcal{W}_s^2(\dP_0,\dP_1):=\inf\{\mathcal{A}(\Bar{\dP},\Bar{\mathbf{V}}):(\Bar{\dP},\Bar{\mathbf{V}})\in \mathcal{CE}^s(\dP_0,\dP_1) \}.
\end{align*}
Note that the right-hand side can be interpreted as a transportation problem with an additional stochastic constraint.
Our next result states that $\mathcal W_s$ indeed defines a distance.
 \begin{satz}\label{in:distance}
     $\mathcal{W}_s$  is  a geodesic extended distance on the space of stationary measures $\PPP_s(\Gamma_{\R^d}).$ 
 \end{satz}
 In this context, extended means that $\mathcal{W}_s$ can attain the value $+\infty$.  However, as a consequence of the following specific Talagrand inequality, it is finite on the domain of the specific relative entropy 
\begin{align*}
    \mathcal{E}(\dP):=\sup_{n\in\N}\frac{1}{\lambda(\Lambda_n)}\mathrm{Ent}(\dP_{|\Lambda_n}|\Poi_{|\Lambda_n}),
\end{align*}
which is a well studied functional in the context of stationary point processes, see e.g.\ \cite{serfaty2017microscopicdescriptionlogcoulomb,fekete}.
  \begin{satz}[Specific Talagrand inequality]\label{in:Talagrand}
    Let $\dP\in\PPP_s(\Gamma_{\R^d}).$ It holds that 
    \begin{align*}
        \mathcal{W}_s^2(\dP,\Poi)\leq \mathcal{E}(\dP).
    \end{align*}
\end{satz}

 The proof of the specific Talagrand inequality relies on a key result of this article which one might hope for from comparing the definition of $\mathcal W_s$ and $\mathcal W_0$ together with the definition of the action functional $\mathcal A$ in \eqref{eq:AsupA}. It shows that the metric $\mathcal W_s$ is built from the metric $\mathcal W_0$ just as the specific entropy is built from the relative entropy.

 \begin{satz}\label{in:reformulation} 
Let $\dP_0,\dP_1\in\PPP_s(\Gamma_{\R^d})$ s.t.\ $\mathcal{E}(\dP_0),\mathcal{E}(\dP_1)<\infty.$ 
It holds that
\begin{align}\label{eq:WsW0}
     \mathcal{W}_s^2(\dP_0,\dP_1)&=\inf\left\{\lim_{n\to\infty}\frac{1}{\lambda(\Lambda_n)}\mathrm{A}({\Bar{\muP}_{|\Lambda_n}},{\Bar{\nuV}_{|\Lambda_n}}):({\Bar{\muP}},{\Bar{\nuV}})\in\CCC\EEE^s(\dP_0,\dP_1)\right\}\\
            &=\lim_{n\to\infty}\frac{1}{\lambda(\Lambda_n)}\inf\left\{\mathrm{A}({\Bar{\muP}^n},{\Bar{\nuV}^n}):({\Bar{\muP}^n},{\Bar{\nuV}^n})\in\CCC\EEE^{\Lambda_n}(\dP_{0|\Lambda_n},\dP_{1|\Lambda_n})\right\}\nonumber \\
            &=\lim_{n\to\infty}\frac{1}{\lambda(\Lambda_n)}\mathcal W_0^2(\dP_{0|\Lambda_n},\dP_{1|\Lambda_n}). \nonumber
\end{align}
\end{satz}
We remark that the assumption of finite entropy is only needed in one step of the proof where we need to know that the optimal vector field for $\mathcal W_0$ is absolutely continuous (cf.\ Lemma \ref{lem:boundstat}). It is always true that the left-hand side in \eqref{eq:WsW0} is bigger than the right-hand side (cf.\ Lemma \ref{lem:boundcands}).

We want to highlight an important and  subtle observation. In the large volume limit of $\mathcal W_0$, there is no stationarity constraint. In particular, it is not at all clear, why a minimizing curve of the right-hand side should be stationary. 
If it was, the right-hand side would be a perfect definition of $\mathcal W_s$ allowing to lift the results of \cite{schiavo2023wasserstein} to the setting of stationary measures.
In fact, most of the work of this article lies in proving Theorem \ref{in:reformulation}. Having this result at our disposal we can derive a couple of consequences for the convexity properties of the specific entropy, identify its gradient flow w.r.t.\ $\mathcal W_s$ and derive functional inequalities.

We start with the description of the gradient flow. Let $(\mathrm{S}_t)_t$ be the Ornstein--Uhlenbeck semigroup, induced by independently deleting points with probability $1-e^{-t}$ and adding new Poisson points with intensity $1-e^{-t}$. Then, we have the following gradient flow characterization of this semigroup in terms of an evolution variational inequality (EVI).
\begin{satz}[EVI] \label{in:EVI}
    Let $\dP_0,\dP_1\in\PPP_s(\Gamma_{\R^d})$ s.t.\ $\mathcal{E}(\dP_0), \mathcal{E}(\dP_1)<\infty.$ Then the following evolution variational inequality holds
    \begin{align*}
        \dt \mathcal{W}_s^2(\mathrm{S}_t \dP_0,\dP_1)+\mathcal{W}_s^2(\mathrm{S}_t \dP_0,\dP_1)\leq 2(\mathcal{E}(\dP_1)-\mathcal{E}(\mathrm{S}_t\dP_0)),\ \ \ \ t\geq0. 
    \end{align*}
\end{satz}
The previous theorem is an analogue to the results of Jordan, Kinderlehrer, and Otto \cite{JKO} in the context of stationary measures.  The EVI-gradient flow formulation is one of the strongest among other gradient flow formulations, as it entails a number of consequences for the semigroup and the geometry of the underlying space, see e.g.\ \cite{Daneri_2008}.
 A direct consequence is the $1$-contractivity of $\mathcal{W}_s$ along the Ornstein--Uhlenbeck semigroup. 
 \begin{satz}[1-contractivity of the Ornstein--Uhlenbeck semigroup]\label{in:1-Contractivity}
      Let $\dP_0,\dP_1\in\PPP_s(\Gamma_{\R^d}).$ It holds that
    \begin{align*}
        \mathcal{W}_s(\mathrm{S}_t\dP_0,\mathrm{S}_t\dP_1)\leq e^{-t}\mathcal{W}_s(\dP_0,\dP_1)\quad  t\geq0.
    \end{align*}
\end{satz}
Another important consequence of the EVI is the $1$-geodesic convexity of the specific relative entropy w.r.t.\ $\mathcal{W}_s$.  
 \begin{satz}[$1$-geodesic convexity of the specific relative entropy] \label{in:ConvexitySRE}
    Let $\dP_0,\dP_1\in\PPP_s(\Gamma_{\R^d}).$ For all $\mathcal{W}_s$-geodesics $(\Bar{\dP}_t)_{t\in[0,1]}$ connecting $\dP_0$ and $\dP_1$ it holds that 
    \begin{align*}
        \mathcal{E}(\dP_t)\leq (1-t) \mathcal{E}(\dP_0)+t\mathcal{E}(\dP_1)-\frac{t(1-t)}{2}\mathcal{W}_s^2(\dP_0,\dP_1), \quad 0\leq t\leq 1.
    \end{align*}
\end{satz}
 According to the definition of synthetic Ricci curvature bounds in the sense of  Sturm \cite{Sturm} and Lott--Villani \cite{LottVillani}, the previous theorem can be interpreted in line with or extending \cite[Theorem 5.28]{schiavo2023wasserstein} that also in a non-local specific stationary geometry the Ricci curvature of  $(\Gamma_{\R^d},\Poi)$ is bounded below by 1.  Moreover, the strict convexity of the entropy along 
$\mathcal{W}_s$-geodesics can be viewed as a non-local analogue of McCann’s displacement convexity \cite{MCCANN1997153}.

In a similar spirit, we obtain an analogue of the $\mathrm{HWI}$ inequality of Otto and Villani \cite{porousmedium} for stationary measures. In the classical setting, this inequality connects the Boltzmann entropy H, the $\mathrm{L}^2$ Wasserstein distance W$_2$, and the Fisher information $\mathsf{I}$. In the context of stationary measures, these quantities are replaced by the specific relative entropy $\EEE$, the transport distance $\mathcal{W}_s,$ and the {specific modified Fisher information}
 \begin{align*}
                 \mathcal{I}(\dP):=\limsup_{n\to\infty}\frac{1}{\lambda(\Lambda_n)}\mathrm{I}(\dP_{|\Lambda_n}|\Poi_{|\Lambda_n}),
             \end{align*} where
             \begin{align*}
             \mathrm{I}(\dP|\Poi_{|\Lambda_n}):=\begin{cases}
        \int D_x\rho(\xi) D_x\log(\rho(\xi))\ \Poi_{|\Lambda_n}\otimes \m_{|\Lambda_n}(\mathrm{d}\xi,\mathrm{d}x),  &\text{if } \dP_{|\Lambda_n}\ll\Poi_{|\Lambda_n} \text{with } \dP_{|\Lambda_n}=\rho \Poi_{|\Lambda_n}\\
        \infty, & \text{otherwise}
    \end{cases}.
             \end{align*}
\begin{satz}[$\mathrm{HWI}$ inequality]\label{in:HWI}
             Let $\dP\in\PPP_s(\Gamma_{\R^d})$. Then the following \emph{HWI} inequality holds
             \begin{align*}
                 \mathcal{E}(\dP)\leq \mathcal{W}_s(\dP,\Poi)\sqrt{\mathcal{I}(\dP)}-\frac{1}{2}\mathcal{W}_s^2(\dP,\Poi).
             \end{align*}
         \end{satz} 
Finally, we note that this HWI inequality recovers a specific variant of  Wu's logarithmic Sobolev inequality \cite{Wu}, 
            \begin{align*}
                \mathcal{E}(\dP)\leq \mathcal{I}(\dP),\quad \dP\in\PPP_s(\Gamma_{\R^d}).
            \end{align*}

\subsection{Connection to the Literature}
The Benamou--Brenier formulation of the $\mathrm{L}^2$ Wasserstein distance has proven to be well suited for diverse singular setups such as Markov chains \cite{maas2011gradientflowsentropyfinite, Mi13}, jump processes \cite{erbar2012gradientflowsentropyjump,PeRoSaTs22}, quantum semigroups \cite{CaMa17}, martingale transport \cite{HuTr19}, and of course most relevant for us Poisson spaces \cite{schiavo2023wasserstein}.
The principle that EVI formulations of gradient flows encode convexity properties of the respective functionals goes back to \cite{Daneri_2008}. These properties in turn imply functional inequalities, see e.g.\ \cite[Section 4.4]{FiGl23} and references therein.

The investigation of Wasserstein geometries, gradient flows, and functional inequalities for point processes  was initiated in \cite{ErHu15}. That article was vastly generalized and extended in \cite{DeSu21, DeSu22}. We also highlight \cite{Su22} where the dynamics associated to the $\mathsf{Sine}_\beta$ process was constructed as a gradient flow w.r.t.\ a Benamou--Brenier-type transport distance. Moreover, several functional inequalities for $\mathsf{Sine}_\beta$ have been obtained. In \cite{GoHePe20} transport inequalities for mixed binomial and Poisson point processes with finite intensity measures have been investigated. In all these contributions, the transport distances are essentially distances between point processes having a density w.r.t.\ a reference point process, they are  ``classical" transport distances.

Building on \cite{HuSt13, Hu16} and motivated by the results of \cite{erbar2018onedimensionalloggasfreeenergy}, in the recent article \cite{MartinJonasBasti} a transport distance $\mathsf C_p$ per volume for stationary point processes was introduced. It was shown that the gradient flow of the specific entropy  w.r.t.\ $\mathsf C_2$ is given by non-interacting Brownian particles. A Benamou--Brenier formulation for $\mathsf C_p$ was derived in \cite{HuMu24}. In \cite{HuLe24}, it was shown among other things that $\mathsf C_2$ satisfies an infinite $\mathrm{W}_2-H^{-1}$ inequality where the $H^{-1}$ norm takes the form of a regularised Coulomb energy. Very recently \cite{HeLe25} constructed a specific Wasserstein-type metric for spin systems and showed an analogue of the famous JKO result \cite{JKO} for an infinite system of interacting Fokker--Planck equations. 

There are of course many other possibilities to construct dynamics on point processes. As  examples we mention the  evolution under heat flow in \cite{HaChJaKa23} and the powerful Dirichlet form approach, see e.g.\ \cite{AKR, AKR2, KoLy05}, which is closely related to the optimal transport framework as shown e.g.\ in \cite{DeSu21, Su22}.

\subsection{Structure of the Paper}
In {Section 2}, we introduce the most frequently used notation and recall the transport distance $\mathcal{W}_0$, which was  developed in \cite{schiavo2023wasserstein} and is based on a non-local continuity equation. We study the continuity equation in greater detail in {Section 3} by proving Theorem \ref{in:cesolution} and analysing several properties of solutions, with a particular focus on stationary measures. In {Section 4}, we define an action functional specifically aligned to stationary measures and investigate its properties. 
Building on this, in {Section 5}, we construct the resulting transport distance $\mathcal{W}_s$, show that it is an extended geodesic distance (Theorem \ref{in:distance}), and prove the existence of minimizers.  {Section 6} contains a discussion of the relationship between the distances $\mathcal{W}_0$ and $\mathcal{W}_s$, providing a reformulation of $\mathcal{W}_s$ in terms of $\mathcal{W}_0$ (Theorem \ref{in:reformulation}).  Using this reformulation, in {Section 7}, we prove the EVI (Theorem \ref{in:EVI}) and derive additional functional inequalities, including the specific Talagrand inequality (Theorem \ref{in:Talagrand}), the $1$-geodesic convexity of the specific relative entropy (Theorem \ref{in:ConvexitySRE}), and the specific $\mathrm{HWI}$ inequality (Theorem \ref{in:HWI}).

\subsection*{Acknowledgements}
We thank Lorenzo Dello Schiavo for several fruitful discussions during the development of this project as well as useful comments on a preliminary version of this manuscript.

\section{Preliminaries}
In this section, we introduce the most frequently used notation and recall the construction of the transport distance $\mathcal W_0$ developed in \cite{schiavo2023wasserstein}. We begin with an overview of the notation used for functions and measures on $\mathbb{R}^d$.

Let $\BBB(\R^d)$ be the Borel sets of $\R^d$ and $\BBB_0(\R^d)$ those that are additionally bounded. For $B\in\BBB(\R^d)$,  we denote by $\FFF(B)$ the space of measurable functions mapping $B$ to $\R.$ For continuous functions, we write $\mathcal{C}(B).$ The space of continuous and compactly supported functions is $\mathcal{C}_c(B)$ and the space of continuous and bounded functions is $\CCC_b(B).$ The space of locally finite signed measures on $B$ is $\mathcal{M}(B)$ and we endow it with the vague topology generated by the maps $m\mapsto\int f(x)m(\mathrm{d}x),$ $f\in\CCC_c(B).$  If a function or measure space is assumed to contain only non-negative functions or measures, this will always be indicated by an additional superscript $+$.

\subsection{Point Processes}
We continue with a brief introduction to the theory of point processes. For a detailed introduction, we refer e.g.\ to \cite{daley2003introduction} and \cite{last_penrose_2017}.
We denote by $\Gamma_B$ be the space of all locally finite and simple counting measures on a Polish set $B\in\BBB(\R^d).$ If $B=\R^d$, we briefly write $\Gamma$ instead of $\Gamma_{\R^d}.$ We endow $\Gamma_B$ with the vague topology generated by the maps $\xi\mapsto\int f(x)\xi(\mathrm{d}x),$ $f\in\CCC_c(B).$ 
 A point process on $B$ is a random variable taking values in $\Gamma_B$. The distribution $\dP$ of a point process is an element of $\PPP(\Gamma_B)$ the space of probability measure on $\Gamma_B.$  The \textit{intensity measure} of  $\dP\in\PPP(\Gamma_B)$ is defined by
  \begin{align*}
\I_\dP(D):=\int\int\ind_D(x)\xi(\mathrm{d}x){\muP}(\mathrm{d}\xi),\quad \text{$D\in\BBB(\Gamma_B)$}.
\end{align*}
We call $\dP$ locally integrable if $\I_\dP(D)<\infty$ for all $D\in\BBB_0(B).$ 
 The space of all locally integrable probability measures on $\Gamma_B$ is $\PPP_1(\Gamma_B)$, and we call a point process locally integrable if its distribution is locally integrable. An $\R$-valued measurable function on the configuration space $F\in\FFF(\Gamma_B)$ has sublinear growth if there exists some  constant $c>0$ and some $f\in\CCC_c(B)$ s.t.\ 
\begin{align*}
    |F(\xi)|\leq c\left(1+\int f(x)\xi(\mathrm{d}x)\right)\quad
\end{align*}
  \text{for all } $\xi\in\Gamma_B.$ The set of all continuous functions having sublinear growth is denoted by $\CCC_1(\Gamma_B).$ We endow $\PPP_1(\Gamma_B)$ with the topology  generated by the maps $\dP\mapsto\int F(\xi)\dP(\mathrm{d}\xi)$, $F\in\CCC_1(\Gamma_B)$ turning  $\PPP_1(\Gamma_B)$ into a Polish space, see \cite[Theorem 2.11]{schiavo2023wasserstein}. 

A measure $\dP\in\PPP_1(\Gamma_B)$  is completely determined by its Laplace functional $$L_\dP(f)=\int e^{-\int f(x) \xi(\mathrm{d}x)}\dP(\mathrm{d}\xi),\quad f\in\CCC_c^+(B).$$

Using the Laplace functional, we have the following characterization of convergence in $\PPP_1(\Gamma_B)$.
\begin{prop}[{\cite[Proposition 2.12]{schiavo2023wasserstein}}]\label{prop:convergencep1}
    Let $(\dP_n)_{n\in\N}\subset\PPP_1(\Gamma_B)$ and $\dP\in\PPP_1(\Gamma_B).$ The following are equivalent:
    \begin{enumerate}
       \item[$(i)$] $\dP_n\underset{n\to\infty}{\to}\dP$ in $\PPP_1(\Gamma_B)$.
       \item[$(ii)$] $L_{\dP_n}(f)\underset{n\to\infty}{\to}L_{\dP}(f)$ for all $f\in\CCC_c^+(B)$  and $\I_{\dP_n}\underset{n\to\infty}{\to}\I_{\dP}$ vaguely in $\MMM^+(B).$
    \end{enumerate}
\end{prop}

We define the shift operator $\Theta_z:\Gamma\to \Gamma$ by $\Theta_z\xi(B):=\xi(B+z),$ $B\in\BBB(\R^d).$ A measure $\dP\in\PPP_1(\Gamma)$ is  \textit{{stationary}} if  ${(\Theta_z)}_\#\dP=\dP$ for all $z\in\R^d,$ where ${(\Theta_z)}_\#\dP$ denotes the push forward of $\dP$ w.r.t.\ $\Theta_z.$  We denote the space of all stationary measures by $\PPP_s(\Gamma)$ and call a point process stationary if its distribution is stationary. Analogously to $\PPP_1(\Gamma)$, we endow the space $\PPP_s(\Gamma)$ with the topology  induced by  the maps $\dP\mapsto\int F(\xi)\dP(\mathrm{d}\xi)$, $F\in\CCC_1(\Gamma).$ 
Further, we define the projection $pr_B:\Gamma\to\Gamma_B$ by  $pr_B(\xi):=\xi(B\cap\cdot)=\xi_{|B}.$  The restriction of 
$\dP\in\PPP_1(\Gamma)$  to $\Gamma_B$ is  $\dP_{|B}:=(pr_B)_{\#} \dP.$

Important examples of point processes are Poisson point processes $\gamma^m$   with intensity measure $m\in\MMM^+(B)$. They are characterized by the following two properties: 
\begin{enumerate}
    \item [$(i)$]  The point statistic $\gamma^m(D):=\int_{B}\ind_D(\xi)\gamma^m(\mathrm{d}\xi)$ is   Poisson distributed with parameter $m(D)$ for any $D\in\BBB(B)$.
    \item [$(ii)$] For all pairwise disjoint sets  $D_1,\dots , D_n \in\BBB(B)$ are the point statistics $\gamma^m(D_1),\dots,\gamma^m(D_n)$  independent.
\end{enumerate}
We denote the distribution of $\gamma^m$ by $\Poi^m$. The Poisson point process $\gamma^m$ is called a homogeneous Poisson point process if $\gamma^m$ is a Poisson point process on $\R^d$ with $m = c\lambda$, where $c > 0$ and $\lambda$ is the $d$-dimensional Lebesgue measure. In this case, we briefly denote the distribution by $\Poi^c$. If $c=1$ we just write $\Poi.$

We conclude this subsection by introducing two relevant operations on point processes. For $p \in [0,1]$, a homogeneous $p$-thinning of a point process $\eta$ is the point process $\eta^{(p)}$, generated by independently removing each point with probability $p$ and retaining it with probability $1-p$. If $\dP$ is the distribution of $\eta$, we denote the distribution of $\eta^{(p)}$ by $\dP^{(p)}$.

The superposition of two point processes $\eta_1$ and $\eta_2$ is the process $\eta_1 + \eta_2$. If $\eta_1$ and $\eta_2$ are independent and have distributions $\dP_1$ and $\dP_2$, respectively, we denote the distribution of their superposition by $\dP_1 \oplus \dP_2$. 

\subsection{Transport Distance on Bounded Domains}\label{sec:distanceDSHS}
We proceed with a repetition of the construction of an analogue to the Benamou--Brenier formulation of the Wasserstein distance on $\dP\in\PPP_1(\Gamma_B)$ developed in \cite{schiavo2023wasserstein}. 
Similar to the classical Benamou--Brenier formulation of the Wasserstein distance, the Wasserstein distance on $\PPP_1(\Gamma_B)$ will be defined as the minimum value of a suitable action functional evaluated over all solutions to a continuity equation.
 
We start with the discussion  of the continuity equation.
 The velocity field will be a curve of signed measures on the product space $\Gamma_B\times B$. More precisely, let $\MMM_{b,0}(\Gamma_B\times B)$ be the set of all signed measures $\mathbf{V}\in\MMM(\Gamma_B\times B)$  s.t.\ $|\mathbf{V}(\Gamma_B\times D)|<\infty$ for all $D\in\BBB_0(B).$ Further, let $\FFF_{b,0}(\Gamma_B\times B)$ be the set of all $\R$-valued functions on $\Gamma_B\times B$ that are bounded and vanish outside of a set of the form $\Gamma_B\times D$ for some $D\in\BBB_0(B).$
Moreover,  $\mathcal{C}_{b,0}(\Gamma_B\times B)$ is the set of functions that satisfy the previous condition and are additionally continuous.
We equip $\MMM_{b,0}(\Gamma_B\times B)$ with the locally convex topology induced by the seminorms 
    $\mathbf{V}\mapsto\int G(\xi,z)\mathbf{V}(\mathrm{d}\xi,\mathrm{d}z), \ G\in \mathcal{C}_{b,0}(\Gamma_B\times B).$
A well-known measure in $\MMM_{b,0}(\Gamma_B\times B)$ is the \textit{reduced Campbell measure}. For $\dP\in\PPP_1(\Gamma_B)$ it is defined by
\begin{align*}
    C_\dP(\A\times D):=\int\int \ind_D(z)\ind_\A(\xi-\delta_z)\xi(\mathrm{d}z)\dP(\mathrm{d}\xi), \quad \A\in\BBB(\Gamma_B), \ D\in\BBB(B).
\end{align*}
 Note that $C_\dP$ is a non-negative measure and that $C_\dP(\Gamma_B \times D) = I_\dP(D) < \infty$ for all $D\in\BBB_0(B)$, which ensures that $C_{\dP}\in\MMM_{b,0}(\Gamma_B\times B)$.

The \textit{discrete difference operator} also known as the \textit{add-one-cost operator}
$DF:\Gamma_B\times B\to\R$ of  $F\in\FFF(\Gamma_B)$ is defined by 
\begin{equation*}
        D_zF(\xi):=F(\xi+\delta_z)-F(\xi).
    \end{equation*}
We have introduced all necessary elements to formulate the continuity equation.
 \begin{dfn}\label{dfn:ce}
    Let $B\in\BBB(\R^d)$ be a Polish set, ${\Bar{\muP}}=(\Bar{\muP}_t)_{t\in[0,T]}\subset\PPP_1(\Gamma_B),$ and ${\Bar{\nuV}}=(\Bar{\nuV}_t)_{t\in[0,T]}\subset\MMM_{b,0}(\Gamma_B\times B).$ We say that the pair $({\Bar{\muP}},{\Bar{\nuV}})=(\Bar{\muP}_t,\Bar{\nuV}_t)_{t\in[0,T]}$ solves the \textit{continuity equation} on $B$ in  time $T>0$ if for all $f\in\CCC_c^+(B)$ and $\phi\in \CCC_c^\infty((0,T))$ it holds that 
     \begin{equation}\label{eq:continuityequation}
        \int_0^T\left(\dt \phi(t)\right)\int e^{-\int f(x)\xi(\mathrm{d}x)}\Bar{\muP}_t(\mathrm{d}\xi)\mathrm{d}t+ \int_0^T \phi(t)\int D_z e^{-\int f(x)\xi(\mathrm{d}x)}\Bar{\nuV}_t(\mathrm{d}\xi, \mathrm{d}z)\mathrm{d}t=0.
    \end{equation}
    Let $\dP_0,\dP_T\in\PPP_1(\Gamma_B).$
  We denote by $\CCC\EEE_T^B(\dP_0,\dP_T)$ the set of all pairs $({\Bar{\muP}},{\Bar{\nuV}})\subseteq \PPP_1(\Gamma_B)\times \MMM_{b,0}(\Gamma_B\times B)$ satisfying the following conditions:
  \begin{enumerate}
      \item[${(i)}$] $\Bar{\muP}_0=\dP_0$, $\Bar{\muP}_T=\dP_T$;
      \item[${(ii)}$]${\Bar{\muP}}:[0,T]\to\PPP_1(\Gamma_B)$ is $\PPP_1(\Gamma_B)$-continuous;
      \item[${(iii)}$] $({\Bar{\muP}},{\Bar{\nuV}})$ solves the continuity equation \eqref{eq:continuityequation};
      \item [${(iv)}$] $\int_0^T|\Bar{\nuV}_t|(\Gamma_B\times D)\mathrm{d}t<\infty$ for all $D\in\BBB_0(B).$
  \end{enumerate}
\end{dfn}
If $B=\R^d,$ we briefly write $\CCC\EEE_T(\dP_0,\dP_T)$ instead of $\CCC\EEE_T^{\R^d}(\dP_0,\dP_T)$, and a solution in this set is called a global solution. We will usually drop the brackets around $\dt \phi(t)$ when working with the continuity equation.
\begin{bem}
The velocity field  $\Bar{\nuV}_t(\mathrm{d}\xi,\mathrm{d}x)$ at time $t$ can be interpreted as the probability that a point appears or disappears at position $x$ in configuration $\xi$ at time $t$, depending on the sign of the measure. Using this interpretation, property $(iv)$ ensures that the expected absolute number of appearing and disappearing points  is integrable.
\end{bem}
Next, we collect some results about the continuity equation.
\begin{satz}\label{thm:propertiesce}
     Let $B\in\BBB(\R^d)$ be Polish and $ \dP_0,\dP_T\in\PPP_1(\Gamma_B)$.
     \begin{enumerate}
        \item[$(i)$] \emph{(\cite[Proposition 4.3]{schiavo2023wasserstein})} If $\dP_0, \dP_T\ll\Poi^m$ for some $m\in\MMM^+(B)$, there exists some solution $(\Bar\dP,\Bar{\nuV})\in\CCC\EEE_T^B(\dP_0,\dP_T)$.
         \item  [$(ii)$]  \emph{(\cite[Lemma 4.5]{schiavo2023wasserstein})} Let  $\kappa:[0,T]\to [0,\Tilde{T}]$ be an increasing and absolutely continuous function s.t. its inverse is also absolutely continuous. Then $({\Bar{\muP}},{\Bar{\nuV}})\in\CCC\EEE^B_{{T}}(\dP_0,\dP_T)$ solves the continuity equation in time $T$ iff  ${({\Bar{\muP}_{\kappa(t)}},\dt\kappa(t)\cdot\Bar{\nuV}_{\kappa(t)})_{t\in[0,\Tilde{T}]}\in\CCC\EEE^B_{\Tilde{T}}(\dP_0,\dP_T)}$ solves the continuity equation in time $\Tilde{T}.$ 
         
         \item [$(iii)$] \emph{(\cite[Proposition 4.8]{schiavo2023wasserstein})} Let $({\Bar{\muP}},{\Bar{\nuV}})\in\CCC\EEE^B_T(\dP_0,\dP_T).$ For any measurable and bounded function $F\in\FFF_b(\Gamma_B)$ s.t.\ $DF\in\FFF_{b,0}(\Gamma_B\times B)$ it holds that 
    \begin{align*}
        \int_0^T\dt \phi(t)\int F(\xi)\Bar{\muP}_t(\mathrm{d}\xi)\mathrm{d}t+ \int_0^T \phi(t)\int D_z F(\xi)\Bar{\nuV}_t(\mathrm{d}\xi, \mathrm{d}z)\mathrm{d}t=0.
    \end{align*}
    \item[$(iv)$] \emph{(\cite[Lemma 4.2]{schiavo2023wasserstein})} Let $({\Bar{\muP}}^n,{\Bar{\nuV}}^n)_{n\in\N}\subset\CCC\EEE^B_T(\dP_0,\dP_T),$ $(\Bar{\muP}_t)_{t\in[0,T]}\subset\PPP_1(\Gamma_B)$ a continuous curve, and $(\Bar{\nuV}_t)_{t\in[0,T]}\subset\MMM_{b,0}(\Gamma_B\times B)$. If ${\Bar{\muP}}_{t}^n \underset{n\to\infty}{\to}{\Bar{\muP}}_t$ in $\PPP_1(\Gamma_B)$ and ${\Bar{\nuV}}_{t}^n\underset{n\to\infty}{\to}{\Bar{\nuV}}_t$ in $\MMM_{b,0}(\Gamma_B\times B)$ for a.e.\ $t\in[0,T],$ then $({\Bar{\muP}},{\Bar{\nuV}})\in\CCC\EEE^B_T(\dP_0,\dP_T)$.
     \end{enumerate}
     
\end{satz}
\begin{bem}\label{bem:ce}
\begin{itemize}
    \item [$(i)$] Theorem \ref{thm:propertiesce} $(i)$ does not ensure the existence of solutions to the continuity equation joining two stationary measures, as  the only stationary measure that is absolutely continuous w.r.t.\ $\mathbf{Poi}^c$ is $\mathbf{Poi}^c$ itself.
    \item [$(ii)$] Theorem \ref{thm:propertiesce} $(ii)$ shows that solutions to the continuity equation are invariant under time reparameterization. This justifies restricting to solutions to the continuity equation in time 1. To simplify the notation, we briefly write $\CCC\EEE^B(\dP_0,\dP_1)$ instead of $\CCC\EEE^{B}_1(\dP_0,\dP_1).$
   
   \end{itemize}
\end{bem}

 As mentioned earlier, the transport distance on $\PPP_1(\Gamma_B)$ will be defined as the minimum value of an action functional evaluated over all solutions to the continuity equation.  
  Consistently with physical interpretations and the classical Benamou--Brenier formulation of the Wasserstein distance, the action functional is defined as the time integral of a Lagrange functional along a curve and an associated velocity field.
    For the definition of the Lagrange functional, we define the logarithmic mean $\theta:\R^+\times \R^+\to\R$ by
    \begin{align*}
           \theta(x,y):=\frac{y-x}{\log(y)-\log(x)}
       \end{align*}
       and the function $\alpha: \R^+\times \R^+\times\R\to \R$ by
       \begin{align*}
           \alpha(x,y,w):=\frac{|w|^2}{\theta(x,y)},
       \end{align*}
      where  $0/0=0$ by convention.
      The function $\alpha$  is lower semi-continuous, convex, and positively
homogeneous, i.e.
\begin{align}
    \alpha( rx, ry, rw) = r\alpha(w, s, t),\quad  w \in\R , \ x, y,r \geq 0,\label{eq:propalpha}\end{align}
see e.g.\ \cite[Lemma 2.2]{erbar2012gradientflowsentropyjump}.
    
        For $\dP\in\PPP_1(\Gamma_B)$ and $\mathbf{V}\in\MMM_{b,0}(\Gamma_B\times B)$ the \textit{Lagrange functional} of $({\dP},{\mathbf{V}})$ is defined by
       \begin{align}\label{eq:dfnlagrange}
           \mathcal{L}({\dP},{\mathbf{V}}):=\int\alpha\left(\frac{\mathrm{d}\dP\otimes \m_{|B}}{\text{d}\sigma},\frac{\mathrm{d}C_\dP}{\text{d}\sigma},\frac{\text{d}{\nuV}}{\text{d}\sigma}\right)\text{d}\sigma,
       \end{align}
       where $\sigma\in\MMM_{b,0}(\Gamma_B\times B)$ is non-negative s.t.\ $\dP\otimes {\m}_{|B},C_\dP,$ and $\mathbf{V}$ are absolutely continuous w.r.t.\ $\sigma,$ i.e.\ $\dP\otimes {\m}_{|B},C_\dP,\mathbf{V}\ll\sigma.$  Due to the positive homogeneity of $\alpha$, the Lagrange functional $\mathcal{L}$ is independent of the choice of $\sigma$. Further, it is jointly convex and lower semi-continuous, see \cite[Lemma 5.1, Lemma 5.2]{schiavo2023wasserstein}. If $\dP=\rho\Poi_{|B}$ and  $\mathbf{V}=w(\Poi_{|B}\otimes\m_{|B})$, the Lagrange functional can be formulated in terms of the densities $\rho$ and $w$. By choosing $\sigma=\Poi_{|B}\otimes\m_{|B}$, we get
     \begin{align}\label{eq:representationlagrange}
             \mathcal{L}({\dP},{\mathbf{V}})=\int\int\alpha(\rho(\xi),\rho(\xi+\delta_z),w(\xi,z))\Poi_{|B}(\mathrm{d}\xi)\m_{|B}(\mathrm{d}z).
         \end{align}
    The following lemma provides a  condition for the velocity field to be absolutely continuous w.r.t.\ $\Poi_{|B}\otimes \m_{|B}$.
         \begin{lem}[{\cite[Lemma 5.3]{schiavo2023wasserstein}}]\label{lem:absolutecontinuityvelocity}
             Let $\dP\in\PPP_1(\Gamma_B)$ and assume that $\dP\ll\Poi_{|B}$. If  $\mathbf{V}\in\MMM_{b,0}(\Gamma_B\times B)$ s.t.\ $\mathcal{L}(\dP,\mathbf{V})<\infty$, then  $\mathbf{V}\ll \Poi_{|B}\otimes \m_{|B}$.
         \end{lem}
The \textit{action functional} A is defined by
     \begin{align*}
         \mathrm{A}({\Bar{\muP}},{\Bar{\nuV}}):=\int_0^1\mathcal{L}(\Bar{\muP}_t,\Bar{\nuV}_t)\mathrm{d}t, \quad (\Bar{\dP},\Bar{\nuV})\in\CCC\EEE^B(\dP_0,\dP_1).
     \end{align*}
   It satisfies the following properties.
   
        \begin{lem} \label{lem:actionDSHS}
        Let $\dP_0,\dP_1\in\PPP_1(\Gamma_B).$
         \begin{enumerate}

             \item[$({i})$] \emph{(\cite[Lemma 5.7]{schiavo2023wasserstein})} The action functional $\mathrm{A}$ is jointly  convex, i.e.\  
             \begin{align*}
                 \mathrm{A}(t{\Bar{\muP}}^1+(1-t){\Bar{\muP}}^2,t{\Bar{\nuV}}^1+(1-t){\Bar{\nuV}}^2)\leq t\mathrm{A} ({\Bar{\muP}}^1,{\Bar{\nuV}}^1)+(1-t) \mathrm{A}({\Bar{\muP}}^2,{\Bar{\nuV}}^2),
             \end{align*}
             for all ${({\Bar{\muP}}^1,{\Bar{\nuV}}^1),({\Bar{\muP}}^2,{\Bar{\nuV}}^2)\in\CCC\EEE^B(\dP_0,\dP_1)},\ t\in[0,1]$.
             \item[$({ii})$] \emph{(\cite[Lemma 5.8]{schiavo2023wasserstein})} The action functional $\mathrm{A}$ is  lower semi-continuous on $\CCC\EEE^B(\dP_0,\dP_1)$.
             \item[$({iii})$] \emph{(\cite[Lemma 5.9]{schiavo2023wasserstein})} The action functional $\mathrm{A}$ has compact sublevel sets on $\CCC\EEE^B(\dP_0,\dP_1).$
         \end{enumerate}
     \end{lem}

     We introduced a continuity equation and an action functional, allowing us to define a transport distance on $\PPP_1(\Gamma_B)$ in the spirit of the Benamou--Brenier formulation of the Wasserstein distance. 
     \begin{dfn}
        For $\dP_0, \dP_1 \in \PPP_1(\Gamma_B)$ define
\begin{align*}
\mathcal{W}_0^2(\dP_0, \dP_1) &:= \inf \left\{ \mathrm{A}(\Bar{\muP}, \Bar{\nuV}) : (\Bar{\muP}, \Bar{\nuV}) \in \CCC\EEE^B(\dP_0, \dP_1) \right\} \\ &= \inf \left\{ \int_0^1 \mathcal{L}(\Bar{\muP}_t, \Bar{\nuV}_t) \mathrm{d}t : (\Bar{\muP}, \Bar{\nuV}) \in \CCC\EEE^B(\dP_0, \dP_1) \right\}.
\end{align*}
 \end{dfn}
By \cite[Theorem 5.15]{schiavo2023wasserstein}, $\mathcal{W}_0$ is indeed an extended distance.  In this context extended means that $\mathcal{W}_0$ can attain the value $+\infty$. However, it is finite on the domain of the relative entropy, which is  defined by 
\begin{align*}
    \mathrm{Ent}(\dP|\Poi_{|B})=\begin{cases}
        \int\rho(\xi)\log(\rho(\xi))\Poi_{|B}(\mathrm{d}\xi), & \text{if } \dP\ll \Poi_{|B} \text{ with } \dP=\rho \Poi_{|B}\\
        \infty, & \text{otherwise}
    \end{cases},\quad \dP\in\PPP_1(\Gamma_B).
\end{align*}

This is a direct consequence of the following Talagrand inequality.
\begin{satz}[{\cite[Theorem 5.17]{schiavo2023wasserstein}}] \label{thm:talagrandDSHS}
    Let $\dP\in\PPP_1(\Gamma_B). $ The following Talagrand inequality is satisfied $$\mathcal{W}_0^{2}(\dP, \Poi_{|B}) \leq \mathrm{Ent}{ (\dP| \Poi_{|B})}.$$
\end{satz}
\begin{bem}\label{bem:explosionofA} Note that the action functional $\mathrm{A}$ typically explodes when applied to stationary measures, implying that the distance $\mathcal{W}_0$ is infinite. For this reason, we only consider the distance $\mathcal{W}_0$ for $\dP\in\PPP_1(\Gamma_B)$, where $B\in\BBB_0(\R^d)$ is Polish. Consistent with previously defined notation, this is indicated by the subscript $0$ in $\mathcal{W}_0$. \end{bem}
Moreover, $\mathcal{W}_0$ is a geodesic distance and the
geodesics are given by the minimizing curves, introduced in the following theorem.
\begin{satz}[{\cite[Theorem 5.12]{schiavo2023wasserstein}}]\label{thm:minimizercompact}
    Let $B\in\BBB_0(\R^d)$ be a Polish set and $\dP_0,\dP_1\in\PPP_1(\Gamma_B)$ s.t.\ $\mathcal{W}_0(\dP_0,\dP_1)<\infty.$ There exists some $({\Bar{\muP}}^*,{\Bar{\nuV}}^*)\in\CCC\EEE^B(\dP_0,\dP_1)$ s.t.\ 
    \begin{align*}
        \mathcal{W}_0^2(\dP_0,\dP_1)=\int_0^1 \mathcal{L}(\Bar{\muP}_t^*, \Bar{\nuV}_t^*) \mathrm{d}t.
    \end{align*}
\end{satz}

Having described the  distance $\mathcal{W}_0$ in more detail, we now shift our focus to the study of gradient flows.
Therefore, we introduce the Ornstein--Uhlenbeck semigroup $(\mathrm{S}_t^B)_t$, which is defined by  
    \begin{align*}
        \mathrm{S}_t^B\dP:= \dP^{(e^{-t})}\oplus\Poi_{|B}^{(1-e^{-t})}, \quad\dP\in\PPP_1(\Gamma_{B}),\ t\geq0.
    \end{align*}
It is a Markov semigroup in the sense that
  \begin{enumerate}
    \item[${(i)}$] $\mathrm{S}_0^B\dP=\dP$;
     \item[${(ii)}$] $\mathrm{S}^B_t(\mathrm{S}^B_s\dP)=\mathrm{S}^B_{s+t}\dP$ for all $s,t\geq 0$;
    \item[${(iii)}$] $\lim_{t\to0}\mathrm{S}^B_t\dP=\dP$ in $\mathcal{P}_1(\Gamma),$
  \end{enumerate}
    see \cite[p. 973]{schiavo2023wasserstein}.

    The following theorem establishes an Evolution Variational Inequality (EVI), which implies that the Ornstein--Uhlenbeck semigroup is the gradient flow of the specific relative entropy w.r.t.\ $\mathcal{W}_0$ in the EVI sense. Furthermore, it discusses several consequences, including the 1-geodesic convexity of the relative entropy and various functional inequalities.

\begin{satz}\label{thm:mainthmDSHS}
Let $B\in\BBB_0(\R^d)$ be Polish and  $\dP_0,\dP_1\in\PPP_1(\Gamma_B)$ s.t.\ $ \mathrm{Ent}(\dP_0|\Poi_{|B}), \mathrm{Ent}(\dP_1|\Poi_{|B})<\infty$.
\begin{enumerate}

\item [$(i)$] \emph{(\cite[Theorem 5.26]{schiavo2023wasserstein})}
The Ornstein--Uhlenbeck semigroup  exponentially contracts $\mathcal{W}_0$ with rate $1$, i.e.
$$\mathcal{W}_0(\mathrm{S}_{t}^B \dP_0, \mathrm{S}_{t}^B \dP_1) \leq \mathrm{e}^{-t} \mathcal{W}_0(\dP_0, \dP_1), \quad t \geq 0.$$

\item[$(ii)$] \emph{(\cite[Theorem 5.27]{schiavo2023wasserstein})} The Ornstein--Uhlenbeck semigroup satisfies the \emph{Evolution Variation Inequality}
\begin{align*}
\mathrm{Ent}({\mathrm{S}_{t}^B \dP_0 | \Poi_{|B}}) + \frac{1}{2} \dt \mathcal{W}_0^{2}(\mathrm{S}_{t}^B \dP_0, \dP_1) + \frac{1}{2} \mathcal{W}_0^{2}(\mathrm{S}_{t}^B \dP_0, \dP_1) \leq \mathrm{Ent}({\dP_1 | \Poi_{|B}}), \quad
t\geq0.
\end{align*}
\item[$(iii)$] \emph{(\cite[Theorem 5.28]{schiavo2023wasserstein})} The relative entropy is $1$-geodesically convex  w.r.t.\ $\mathcal{W}_0$ on the domain of the relative entropy.
\item [$(iv)$] \emph{(\cite[Theorem 3.3 $(ii)$]{schiavo2023wasserstein})}
The \textit{modified Fisher information}
             \begin{align*}
             \mathrm{I}(\dP_0|\Poi_{|B}):=\begin{cases}
        \int D_x\rho(\xi) D_x\log(\rho(\xi))\Poi_{|B}\otimes \m_{|B}(\mathrm{d}\xi,\mathrm{d}x), \quad &\text{if } \dP_0\ll\Poi_{|B} \text{ with } \dP_0=\rho \Poi_{|B}\\
        \infty, & \text{otherwise}
    \end{cases} ,   \end{align*} controls the entropy along the Ornstein--Uhlenbeck semigroup, and the following de Bruijn's identity is satisfied
\begin{align*}
                \mathrm{Ent}(\mathrm{S}_t^B{\dP}_{0}|\Poi_{|B})-  \mathrm{Ent}({\dP}_{0}|\Poi_{|B})=- \int_0^t\mathrm{I}(\mathrm{S}_r^B{\dP}_{0}|\Poi_{|B})\mathrm{d}r.
            \end{align*}
\item[$(v)$] \emph{(\cite[Theorem 5.30]{schiavo2023wasserstein})}
The entropy, the transport distance $\mathcal{W}_0$, and the modified Fisher information  satisfy the $\mathrm{HWI}$ inequality
 \begin{align*}
    \mathrm{Ent}(\dP_0)\leq \mathcal{W}_0(\dP_0,\Poi_{|B})\sqrt{{I}(\dP_0)}-\frac{1}{2}\mathcal{W}_0^2(\dP_0,\Poi_{|B}).
\end{align*}
     \end{enumerate}
\end{satz}

\section{The Continuity Equation on $\R^d$}\label{sec:equation}
The aim of this section is to prove Theorem \ref{in:cesolution}, i.e.\ the existence of (stationary) solutions to the continuity equation between arbitrary measures $\dP_0$ and $\dP_1$.
The non-stationary case is treated in Theorem \ref{thm:solution}, the stationary case in Corollary \ref{cor:stationarysolution}. Furthermore, we establish some useful properties of solutions to the continuity equation.

We start by considering distributions of homogeneous Poisson point processes as a motivating example, and then proceed with the discussion of arbitrary measures.

\begin{lem}\label{lem:Poisolution}
   Let $\dP_0=\Poi^{c_0}$ and $ \dP_1=\Poi^{c_1}$ with $c_0,c_1>0.$   For  $t\in[0,1]$  define \begin{align*}
        \Bar{\muP}_t=\Poi^{(1-t)c_0+tc_1} \quad \text{ and }
\quad  \Bar{\nuV}_t=(c_1-c_0)\Bar{\muP}_t\otimes\lambda.    \end{align*} 
Then $({\Bar{\muP}
},{\Bar{\nuV}})\in\CCC\EEE(\dP_0, \dP_1).$ 
\end{lem}
   
\begin{proof}
To prove the claim, we need to verify that the properties $(i)$ - $(iv)$ of Definition \ref{dfn:ce} are satisfied. Property $(i)$ is true by the definition of $\Bar{\muP}_t$. We continue with the proof of property $(ii)$. By Proposition \ref{prop:convergencep1}, we note that $(\Bar{\muP}_t
)_{t\in[0,1]}$ is $\PPP_1(\Gamma)$-continuous, if and only if $(L_{\Bar{\muP}_t}(f))_{t\in[0,1]}$ is continuous for any $f\in\CCC_c^+(\R^d)$  and  $(\I_{\dP_t})_{t\in[0,1]}$ is vaguely continuous. Note that the Laplace functional of  the distribution of a homogenous Poisson point process  $\Bar{\muP}_t$ with intensity  $(1-t)c_0+tc_1$ is given by 
   \begin{equation}\label{eq:15}
       L_{\Bar{\muP}_t}(f)=\int e^{-\int f(x)\xi(\mathrm{d}x)}\Bar{\muP}_t(\mathrm{d}\xi)=e^{-((1-t)c_0+tc_1)\int (1-e^{-f(x)})\lambda(\mathrm{d}x)},\quad f\in\CCC_c^+(\R^d),
   \end{equation}
   see e.g. \cite[Theorem 3.9]{last_penrose_2017}. Therefore, we can deduce that $(L_{\Bar{\muP}_t}(f))_{t\in[0,1]}$ is continuous. The intensity measures $\I_{\Bar{\muP}_t}=((1-t)c_0+tc_1)\lambda$ are clearly  vaguely continuous in $t$.
   
   We continue with the proof of property $(iii)$ and show that the pair $({\Bar{\muP}},{\Bar{\nuV}})$ satisfies the continuity equation. Let $\phi\in\CCC_c^\infty((0,T)).$ Using  the formula for the Laplace functional \eqref{eq:15} and partial integration, we get 
   \begin{align*}
       \int_0^T\dt \phi(t)&\int e^{-\int f(x)\xi(\mathrm{d}x)}\Bar{\muP}_t(\mathrm{d}\xi)\mathrm{d}t
       = \int_0^T\dt\phi(t) e^{-((1-t)c_0+tc_1)\int (1-e^{-f(x)})\lambda(\mathrm{d}x)}\mathrm{d}t\\
       &=-  \int_0^T \phi(t) \dt e^{-((1-t)c_0+tc_1)\int (1-e^{-f(x)})\lambda(\mathrm{d}x)}\mathrm{d}t\\
       &=- \int_0^T \phi(t) (c_0-c_1)\int (1-e^{-f(z)})\lambda(\mathrm{d}z)  e^{-((1-t)c_0+tc_1)\int (1-e^{-f(x)})\lambda(\mathrm{d}x)}\mathrm{d}t\\
       &= - \int_0^T \phi(t) (c_1-c_0)\int (e^{-f(z)}-1) \int e^{-\int f(x)\xi(\mathrm{d}x)}\Bar{\muP}_t(\mathrm{d}\xi) \lambda(\mathrm{d}z)\mathrm{d}t \\
       &= - \int_0^T \phi(t) (c_1-c_0)\int \int D_z  e^{-\int f(x)\xi(\mathrm{d}x)}\Bar{\muP}_t(\mathrm{d}\xi)\lambda(\mathrm{d}z)\mathrm{d}t\\
       &= - \int_0^T \phi(t)  \int D_z  e^{-\int f(x)\xi(\mathrm{d}x)}\Bar{\nuV}_t(\mathrm{d}\xi,\mathrm{d}z)\mathrm{d}t.
   \end{align*}
   This shows that the continuity equation is satisfied for the pair $({\Bar{\muP}},{\Bar{\nuV}}).$ Finally, let $B\in \BBB_0(\R^d).$ It holds that \begin{align*}
       \int_0^1|\Bar{\nuV}_t|(\Gamma\times B)\mathrm{d}t=\int_0^1|c_1-c_0|\lambda(B)\mathrm{d}t<\infty,
   \end{align*}
   which shows that property $(iv)$ is satisfied and finishes the proof.
\end{proof}
    We continue with a discussion of the curve $(\Bar{\muP}_t
)_{t\in[0,1]}$ constructed in Theorem \ref{thm:solution}.
    A homogenous Poisson point process with intensity $c_0(1-t)+c_1t$ can also be interpreted as the independent superposition of a homogeneous Poisson point process with intensity $c_0$ that is thinned with rate $(1-t)$ and a homogeneous Poisson point process with intensity $c_1$ that is thinned with rate $t.$ In other words, it holds that \begin{align*}
   \Bar{\muP}_t= \Poi^{c_0(1-t)+c_1t}= (\Poi^{c_0})^{(1-t)}\oplus (\Poi^{c_1})^{(t)},
\end{align*}
see e.g.\ \cite[Theorem 3.3, Corollary 5.9]{last_penrose_2017}. In the following theorem, we generalize this approach and construct paths connecting general distributions of point processes in an analogous way. 
     \begin{satz}\label{thm:solution}
         Let $B\in\BBB(\R^d)$ be Polish, and let $\eta_0=\sum_{n=1}^\infty\delta_{X_n}$ and $\eta_1=\sum_{n=1}^\infty\delta_{Y_n}$ be independent point processes with distributions $\dP_0,\dP_1\in\PPP_1(\Gamma_B),$ respectively. For any $t\in[0,1],$ let $(\mathcal{T}_{X_n}^t)_{n\in\N}$ be a family of i.i.d.\ random variables s.t.\ $\mathcal{T}_{X_n}^t\sim\mathrm{Ber}(1-t)$ and  $(\mathcal{T}_{Y_n}^t)_{n\in\N}$ be another independent family of i.i.d.\ random variables s.t.\ $\mathcal{T}_{Y_n}^t\sim\mathrm{Ber}(t).$   For any $t\in[0,1]$, we define the point process 
         \begin{equation*}
             \eta_t=\eta_0^{(1-t)}+\eta_1^{(t)}=\sum_{n=1}^\infty \delta_{X_n}\ind_{\mathcal{T}_{X_n^t=1}}+ \sum_{n=1}^\infty \delta_{Y_n}\ind_{\mathcal{T}_{Y_n^t=1}}.
         \end{equation*}
         The distribution of $\eta_t$ is called $\Bar{\muP}_t.$ Moreover, let  $\Bar{\nuV}_t$ be defined by
         \begin{align*}
             \Bar{\nuV}_t(\A\times D)&= \E\left[\sum_{n=1}^\infty \ind_D(Y_n) \ind_\A(\sum_{l=1}^\infty \delta_{X_l}\ind_{\tau_{X_l}=1}+\sum_{l=1, l\not=n}^\infty\delta_{Y_l}\ind_{\tau_{Y_l}=1})\right] \\
             &\quad - \E\left[\sum_{m=1}^\infty \ind_D(X_m)\ind_\A(\sum_{l=1,l\not=m}^\infty\delta_{X_l}\ind_{\tau_{X_l}=1}+\sum_{l=1}^\infty\delta_{Y_l}\ind_{\tau_{Y_l}=1})\right],\quad K\in\BBB(\Gamma_B),\ D\in\BBB(B).
         \end{align*}
         Then $({\Bar{\muP}},{\Bar{\nuV}})\in\CCC\EEE^B(\dP_0,\dP_1)$. 
     \end{satz}
     \begin{proof}
  For the proof of the theorem, we need to show that properties $(i)$ - $(iv)$ of Definition \ref{dfn:ce} are satisfied.  Property $(i)$ is satisfied by definition. We continue with the proof of property $(ii)$. By Proposition \ref{prop:convergencep1},  $(\Bar{\dP}_t)_{t\in[0,1]}$ is $\PPP_1(\Gamma_B)$-continuous in $t$  if and only if $( L_{\Bar{\muP}_t}(f))_{t\in[0,1]}$ is continuous in $t$ for any $f\in\CCC^+_c(B)$ and $(\I_{\Bar{\muP}_t})_{t\in[0,1]}$ is vaguely continuous in $t$.
     Therefore, let $f\in \CCC_c^+(B).$ As $\eta_0$ and $\eta_1$ are locally integrable point processes and since $f$ is compactly supported, there exists random sets $M_X,N_Y\subset \N$ s.t.\ $f(X_m)=f(Y_n)=0$ for all $m\notin M_X, n\notin N_Y$ and $\E[|M_X|],\E[|N_Y|]<\infty$. Using the definition of $(\eta_t)_{t\in[0,1]},$  we get that 
         \begin{align}
             \int e^{-\int f(x)\xi(\mathrm{d}x)}\Bar{\muP}_t(\mathrm{d}\xi) &=\E\left[e^{-\sum_{n=1}^\infty f(X_n)\ind_{\{\mathcal{T}_{X_n}^t=1\}}- \sum_{n=1}^\infty f(Y_n)\ind_{\{\mathcal{T}_{Y_n}^t=1\}}}\right]\nonumber\\
             &=\E\left[e^{-\sum_{m\in M_X} f(X_m)\ind_{\{\mathcal{T}_{X_m}^t=1\}}- \sum_{n\in N_Y} f(Y_n)\ind_{\{\mathcal{T}_{Y_n}^t=1\}}}\right]\nonumber\\
             &= \E\left[\E\left[e^{-\sum_{m\in M_X} f(X_m)\ind_{\{\mathcal{T}_{X_m}^t=1\}}- \sum_{n\in N_Y} f(Y_n)\ind_{\{\mathcal{T}_{Y_n}^t=1\}}}|(X_n)_{n\in\N},(Y_n)_{n\in\N}\right]\right]\nonumber\\
             &=\E\left[\sum_{J\subseteq M_X, K\subseteq N_Y} t^{(|M_X|-|J|+|K|)}(1-t)^{(|J|+|N_Y|-|K|)}e^{{-\sum_{j\in J} f(X_j)}- \sum_{k\in K} f(Y_k)}\right].\label{eq:4}
         \end{align}
         For the computation of the conditional expectation in the last equality, we have used that  $M_X$ and $N_Y$ are a.s.\ finite and $\tau_{X_m}$ and $\tau_{Y_n}$ are Bernoulli-distributed. Let $(t_n)_{n\in\N}\subset[0,1]$ be a sequence that converges to $t\in[0,1].$ We note that 
         \begin{align}
             &\sum_{J\subseteq M_X, K\subseteq N_Y} t_n^{(|M_X|-|J|+|K|)}(1-t_n)^{(|J|+|N_Y|-|K|)}e^{{-\sum_{j\in J} f(X_j)}- \sum_{k\in K} f(Y_k)}\nonumber\\
             &\quad\leq \sum_{J\subseteq M_X} t_n^{(|M_X|-|J|)}(1-t_n)^{|J|} \sum_{ K\subseteq N_Y} t_n^{|K|}(1-t_n)^{(|N_Y|-|K|)}=1, \label{eq:500}
         \end{align}
         by using the binomial theorem. Therefore, the continuity of the Laplace functional follows by the dominated convergence theorem, i.e.\
         \begin{align*}
             \lim_{n\to\infty}L_{\Bar{\muP}_{t_n}}(f)=&\lim_{n\to\infty}\E\left[\sum_{J\subseteq M_X, K\subseteq N_Y} t_n^{(|M_X|-|J|+|K|)}(1-t_n)^{(|J|+|N_Y|-|K|)}e^{{-\sum_{j\in J} f(X_j)}- \sum_{k\in K} f(Y_k)}\right]\\
             &= \E\left[\sum_{J\subseteq M_X, K\subseteq N_Y} t^{(|M_X|-|J|+|K|)}(1-t)^{(|J|+|N_Y|-|K|)}e^{{-\sum_{j\in J} f(X_j)}- \sum_{k\in K} f(Y_k)}\right]\\
             &=L_{\Bar{\muP}_{t}}(f).
         \end{align*}
        We note that the intensity measure of $\Bar{\dP}_t$ is given by $ \I_{\Bar{\dP}_t}=(1-t)\I_{\dP_0}+t\I_{\dP_1}.$
        Thus, $(\I_{\Bar{\dP}_t})_{t\in[0,1]}$ is vaguely continuous in $t$, which finishes the proof of property $(ii)$.
         
          We continue with the proof of property $(iii)$ and show that the pair $(\Bar{\muP},\Bar{\nuV})$ satisfies the continuity equation. 
         Using the definition of $(\eta_t)_{t\in[0,1]}$  and the  functional monotone class theorem, see e.g.\ \cite[Theorem 14.C.1]{RandomMP}, we find that for all functions $G\in\FFF^+_{b,0}(\Gamma_B\times B)$ it holds that
         \begin{align}\label{eq:2}
             \int G(\xi,z)\Bar{\nuV}_t(\mathrm{d}\xi,\mathrm{d}z)&=  \E\left[\sum_{n=1}^\infty G(\sum_{l=1}^\infty \delta_{X_l}\ind_{\tau_{X_l}=1}+\sum_{l=1, l\not=n}^\infty\delta_{Y_l}\ind_{\tau_{Y_l}=1}, Y_n)\right] \nonumber\\
             &\quad - \E\left[\sum_{n=1}^\infty G(\sum_{l=1,l\not=n}^\infty\delta_{X_l}\ind_{\tau_{X_l}=1}+\sum_{l=1}^\infty\delta_{Y_l}\ind_{\tau_{Y_l}=1},X_n)\right].  
         \end{align}
          Let $f\in\CCC_c^+(B).$  We can deduce that
         \begin{align*}
             \int (e^{-f(z)}-1)e^{-\int f(x)\xi(\mathrm{d}x)}\Bar{\nuV}_t(\mathrm{d}\xi,\mathrm{d}z)
             &= \E\left[\sum_{n=1}^\infty(e^{-f(Y_n)}-1)e^{-\sum_{l=1}^\infty f(X_l)\ind_{\mathcal{T}^t_{X_l=1}}-\sum_{l=1, l\not = n}^\infty f(Y_l)\ind_{\mathcal{T}_{Y_l}^t=1}}\right] \\
             &\quad  - \E\left[\sum_{n=1}^\infty(e^{-f(X_n)}-1)e^{-\sum_{l=1, l\not=n}^\infty f(X_l)\ind_{\mathcal{T}^t_{X_l=1}}-\sum_{l=1}^\infty f(Y_l)\ind_{\mathcal{T}_{Y_l}^t=1}}\right].
             \end{align*}
            Analogously to the proof of property $(ii)$, we get that  the right-hand side of the equation above is equal to
            \begin{align*}
             &\E\left[\sum_{n\in N_Y}(e^{-f(Y_n)}-1)
             \sum_{J\subseteq M_X, K\subseteq N_Y\setminus\{n\}}t^{(|M_X|-|J|+|K|)}(1-t)^{(|J|+|N_Y|-1-|K|)}e^{-\sum_{j\in J}f(X_j)-\sum_{k\in K}f(Y_k)}\right]\\
             &-\E\left[\sum_{m\in M_X}(e^{-f(X_m)}-1)
             \sum_{J\subseteq M_X\setminus \{m\}, K\subseteq N_Y}t^{(|M_X|-|J|+|K|-1)}(1-t)^{(|J|+|N_Y|-|K|)}e^{-\sum_{j\in J}f(X_j)-\sum_{k\in K}f(Y_k)}\right].
             \end{align*} 
             Rewriting the last expression yields
             \begin{align}
            & \quad \E\left[\sum_{n\in N_Y}e^{-f(Y_n)}
             \sum_{J\subseteq M_X, K\subseteq N_Y\setminus\{n\}}t^{(|M_X|-|J|+|K|)}(1-t)^{(|J|+|N_Y|-1-|K|)}e^{-\sum_{j\in J}f(X_j)-\sum_{k\in K}f(Y_k)}\right]\nonumber\\
             &\quad  -\E\left[\sum_{n\in N_Y}
             \sum_{J\subseteq M_X, K\subseteq N_Y\setminus\{n\}}t^{(|M_X|-|J|+|K|)}(1-t)^{(|J|+|N_Y|-1-|K|)}e^{-\sum_{j\in J}f(X_j)-\sum_{k\in K}f(Y_k)}\right]\nonumber\\
             & \quad  -\E\left[\sum_{m\in M_X}e^{-f(X_m)}
             \sum_{J\subseteq M_X\setminus \{m\}, K\subseteq N_Y}t^{(|M_X|-|J|+|K|-1)}(1-t)^{(|J|+|N_Y|-|K|)}e^{-\sum_{j\in J}f(X_j)-\sum_{k\in K}f(Y_k)}\right]\nonumber\\
             &\quad +\E\left[\sum_{m\in M_X}
             \sum_{J\subseteq M_X\setminus \{m\}, K\subseteq N_Y}t^{(|M_X|-|J|+|K|)}(1-t)^{(|J|+|N_Y|-1-|K|)}e^{-\sum_{j\in J}f(X_j)-\sum_{k\in K}f(Y_k)}\right].\nonumber\end{align}
            Combining the first and the second summand and the third and the fourth summand by combinatorial considerations, we get 
            \begin{align}
            & \E\left[
             \sum_{J\subseteq M_X, K\subseteq N_Y, K\not=\emptyset} |K|t^{(|M_X|-|J|+|K|-1)}(1-t)^{(|J|+|N_Y|-|K|)}e^{-\sum_{j\in J}f(X_j)-\sum_{k\in K}f(Y_k)}\right]\nonumber\\
             &\quad  -\E\left[
             \sum_{J\subseteq M_X, K\subseteq N_Y, K\not = N_Y }(|N_Y|-|K|)t^{(|M_X|-|J|+|K|)}(1-t)^{(|J|+|N_Y|-1-|K|)}e^{-\sum_{j\in J}f(X_j)-\sum_{k\in K}f(Y_k)}\right]\nonumber\\
             &\quad-  \E\left[\sum_{J\subseteq M_X, K\subseteq N_Y, J\not =\emptyset}
             |J|t^{(|M_X|-|J|+|K|)}(1-t)^{(|J|+|N_Y|-|K|-1)}e^{-\sum_{j\in J}f(X_j)-\sum_{k\in K}f(Y_k)}\right]\nonumber\\
             &\quad + \E\left[
             \sum_{J\subseteq M_X, K\subseteq N_Y, J\not = M_X}(|M_X|-|J|)t^{(|M_X|-|J|+|K|-1)}(1-t)^{(|J|+|N_Y|-|K|)}e^{-\sum_{j\in J}f(X_j)-\sum_{k\in K}f(Y_k)}\right].\nonumber
            \end{align}
        Finally, we note that the summands for   $K=\emptyset$ in the first line, $K=N_Y$ in the second line, $J =\emptyset$ in the third line, and $J =M_X$ in the fourth line are zero. All in all, we find that 
        \begin{align}
            \nonumber\int (e^{-f(z)}-1)&e^{-\int f(x)\xi(\mathrm{d}x)}\Bar{\nuV}_t(\mathrm{d}\xi,\mathrm{d}z)\\
            &\quad\quad= \E\Bigg[\sum_{J\subseteq M_X, K\subseteq N_Y}((|M_X|-|J|+|K|)(1-t)-  (|J|+|N_Y|-|K|)t)\nonumber\\
             &\quad\quad\quad\quad\quad\quad \cdot t^{(|M_X|-|J|+|K|-1)}(1-t)^{(|J|+|N_Y|-|K|-1)}e^{{-\sum_{j\in J} f(X_j)} - \sum_{k\in K} f(Y_k)}\Bigg]. \label{eq:6}
        \end{align}
      On the other hand, by equation \eqref{eq:4} it holds that 
      \begin{align*}
          \int e^{-\int f(x)\xi(\mathrm{d}x)}\Bar{\muP}_t(\mathrm{d}\xi) =\E\left[\sum_{J\subseteq M_X, K\subseteq N_Y} t^{(|M_X|-|J|+|K|)}(1-t)^{(|J|+|N_Y|-|K|)}e^{{-\sum_{j\in J} f(X_j)}- \sum_{k\in K} f(Y_k)}\right].
      \end{align*}
      We want to differentiate the term w.r.t.\ $t$. The term inside the last expectation is bounded by $1$ by \eqref{eq:500}. Therefore, interchanging the integration and the differentiation w.r.t.\ $t$ is allowed, and we get
         \begin{align}
             &\dt \int e^{-\int f(x)\xi(\mathrm{d}x)}\Bar{\muP}_t(\mathrm{d}\xi)\nonumber\\
             &= \E[\sum_{J\subseteq M_X, K\subseteq N_Y} (|M_X|-|J|+|K|)t^{(|M_X|-|J|+|K|-1)}(1-t)^{(|J|+|N_Y|-|K|)}e^{{-\sum_{j\in J} f(X_j)}- \sum_{k\in K}^\infty f(Y_k)}\nonumber\\
             &\quad- (|J|+|N_Y|-|K|) t^{(|M_X|-|J|+|K|)}(1-t)^{(|J|+|N_Y|-|K|-1)}e^{{-\sum_{j\in J} f(X_j)}- \sum_{k\in K} f(Y_k)}]\nonumber\\
             &= \E\Big[\sum_{J\subseteq M_X, K\subseteq N_Y}((|M_X|-|J|+|K|)(1-t)-(|J|+|N_Y|-|K|)t )\nonumber\\
             &\quad\quad\quad\quad\quad\quad \cdot t^{(|M_X|-|J|+|K|-1)}(1-t)^{(|J|+|N_Y|-|K|-1)}e^{{-\sum_{j\in J} f(X_j)}- \sum_{k\in K} f(Y_k)}\Big].\label{eq:3}
         \end{align}
      Comparing both expressions, \eqref{eq:6} and \eqref{eq:3}, and using partial integration, we can deduce that $({\Bar{\muP}},{\Bar{\nuV}})$ solves the continuity equation.
      
     Finally, it remains to prove that property $(iv)$ is satisfied. Using equation \eqref{eq:2}, we obtain that \begin{align*}
             |\Bar{\nuV}_t(\A\times B)|&\leq  \E\left[\sum_{n=1}^\infty \ind_B(Y_n) \ind_\A(\sum_{l=1}^\infty \delta_{X_l}\ind_{\tau_{X_l}=1}+\sum_{l=1, l\not=n}^\infty\delta_{Y_l}\ind_{\tau_{Y_l}=1})\right] \\
             &\quad + \E\left[\sum_{n=1}^\infty \ind_B(X_n)\ind_\A(\sum_{l=1,l\not=n}^\infty\delta_{X_l}\ind_{\tau_{X_l}=1}+\sum_{l=1}^\infty\delta_{Y_l}\ind_{\tau_{Y_l}=1})\right] \\
             &\leq \E\left[\sum_{n=1}^\infty\ind_B(X_n)\right] + \E\left[\sum_{n=1}^\infty\ind_B(Y_n)\right].
         \end{align*} The last quantity is finite as $\eta_0$ and $\eta_1$ are assumed to be locally integrable, which finishes the proof.\end{proof}     
  \begin{bem}
      The solution constructed above is not the only possible one. We consider again two independent point processes $\eta_0$ and $\eta_1$ with distributions $\dP_0,\dP_1 \in \PPP_1(\Gamma_B)$, respectively. Let $\Bar{\dP}_t$ be the distribution of a point process ${\eta}_t$ constructed in the following way.  Choose an arbitrary matching between the two point processes $\eta_0$ and $\eta_1$ and select for any pair of points either the point of $\eta_0$ with probability $1-t$ or the point of $\eta_1$ with probability $t$. Then it is possible to construct a velocity field in a similar form as in Theorem \ref{thm:solution} s.t. $(\Bar{\dP},\Bar{\nuV})\in\CCC\EEE^B(\dP_0,\dP_1).$
 \end{bem}       
 \subsection{Properties of the Continuity Equation}
 In this section, we discuss various properties of the continuity equation. We start with the study of the behaviour of global solutions when restricting them to smaller domains.
Let $B\in\BBB(\R^d)$ be Polish and recall that the restriction of 
$\dP\in\PPP_1(\Gamma)$  to $\Gamma_B$ is denoted by $\dP_{|B}.$
We define $pr_{\Gamma_B}:\Gamma\times\R^d\to\Gamma_B\times \R^d$ by $pr_{\Gamma_B}(\xi,x):=(\xi_{|B},x).$ The restriction of a measure 
    ${\nuV}\in\mathcal{M}_{b,0}(\Gamma\times \R^d)$ to 
    $\Gamma_B\times B$ is defined by { $${\nuV}_{|B}:=\ind_B(x){(pr_{\Gamma_B})}_\#{\nuV}.$$}
\begin{lem} \label{lem:restriction}
      Let $({\Bar{\muP}},{\Bar{\nuV}})\in\CCC\EEE(\dP_0,\dP_1)$ and $B\in\BBB(\R^
      d)$ be Polish. Then $({\Bar{\muP}}_{|B},{\Bar{\nuV}}_{|B})\in\CCC\EEE^{B}({\dP_0}_{|B},{\dP_1}_{|B})$.
\end{lem}
\begin{proof} To prove the claim, we need to verify that properties $(i)$ - $(iv)$ of Definition \ref{dfn:ce} are satisfied for $({\Bar{\muP}}_{|B},{\Bar{\nuV}}_{|B})$.
Properties $(i)$ and $(iv)$ are  satisfied by the definition of $({\Bar{\muP}}_{|B},{\Bar{\nuV}}_{|B})$. Property $(ii)$, the $\PPP_1(\Gamma_B)$-continuity of $(\Bar{\dP}_t)_{t\in[0,1]}$, follows from the standard approach using Proposition \ref{prop:convergencep1}.
It only remains to show that $({\Bar{\muP}}_{|B},{\Bar{\nuV}}_{|B})$ solves the continuity equation on $B.$ Therefore, 
   let $\phi\in\CCC_c^\infty((0,1))$ and $f\in\CCC_c^+(B).$  It holds that 
    \begin{align*}
        \int_0^1\dt\phi(t)\int e^{-\int f(x)\xi(\mathrm{d}x)}{\Bar{\muP}_{t|B}}(\mathrm{d}\xi)\mathrm{d}t &= \int_0^1\dt\phi(t)\int e^{-\int f(x)\xi_{|B}(\mathrm{d}x)}\Bar{\muP}_t(\mathrm{d}\xi)\mathrm{d}t\nonumber\\\
        &=  \int_0^1\dt\phi(t)\int e^{-\int f(x)\ind_B(x)\xi
        (\mathrm{d}x)}\Bar{\muP}_t(\mathrm{d}\xi)\mathrm{d}t.
        \end{align*}
         Note that the continuity equation on $\R^d$ is not directly applicable for $f\ind_B$ as it is not a continuous function on $\R^d$. However, since $D_z e^{-\int f(x)\ind_B(x)\xi
        (\mathrm{d}x)}\in \FFF_{b,0}(\Gamma_B\times B),$ the continuity equation still holds for $f\ind_B$ by Theorem \ref{thm:propertiesce} $(ii)$. This implies that
        \begin{align*}
            \int_0^1\dt\phi(t)\int e^{-\int f(x)\ind_B(x)\xi
        (\mathrm{d}x)}\Bar{\muP}_t(\mathrm{d}\xi)\mathrm{d}t&=-\int_0^1 \phi(t)\int D_z e^{-\int f(x)\ind_B(x)\xi
        (\mathrm{d}x)}\Bar{\nuV}_t(\mathrm{d}\xi,\mathrm{d}z)\mathrm{d}t\\
        &= -\int_0^1 \phi(t)\int (e^{-f(z)\ind_B(z)}-1) e^{-\int f(x)\ind_B(x)\xi
        (\mathrm{d}x)}\Bar{\nuV}_t(\mathrm{d}\xi,\mathrm{d}z)\mathrm{d}t\nonumber\\
        &= -\int_0^1 \phi(t)\int (e^{-f(z)}-1) e^{-\int f(x)\ind_B(x)\xi
        (\mathrm{d}x)}\ind_B(z)\Bar{\nuV}_t(\mathrm{d}\xi,\mathrm{d}z)\mathrm{d}t\nonumber\\
        &= -\int_0^1 \phi(t)\int D_z e^{-\int f(x)\xi
        (\mathrm{d}x)}{\Bar{\nuV}_{t|B}}(\mathrm{d}\xi,\mathrm{d}z)\mathrm{d}t.
    \end{align*}
   Combining both equations finishes the proof.
\end{proof}
The following lemma deals with the superposition of solutions to the continuity equation.
\begin{lem}\label{lem:superpositionofsolution}
   Let $B,D\in\BBB(\R^d)$ be disjoint Polish sets. Let  $\dP_0,\dP_1\in\PPP_1(\Gamma_{B\cup D})$  s.t.\ $\dP_0={\dP_0}_{|B}\otimes{\dP_0}_{|D}$ and $\dP_1={\dP_1}_{|B}\otimes{\dP_1}_{|D}$. Further, let $({\Bar{\muP}}^B,{\Bar{\nuV}}^B)\in\CCC\EEE^B({\dP_0}_{|B},{\dP_1}_{|B})$ and $({\Bar{\muP}}^D,{\Bar{\nuV}}^D)\in\CCC\EEE^D({\dP_0}_{|D},{\dP_1}_{|D}).$ For all $t\in[0,1]$ define  $\Bar{\muP}_t$ by 
   \begin{align*}
        \int_{\Gamma_{B\cup D}} F(\xi)\Bar{\muP}_t(\mathrm{d}\xi)=\int_{\Gamma_B}\int_{\Gamma_D} F(\xi_1+\xi_2)\Bar{\muP}_t^B(\mathrm{d}\xi_1)\Bar{\muP}_t^D (\mathrm{d}\xi_2)
    \end{align*}
   for all $F\in\CCC_1(\Gamma_{B\cup D}).$ Further, define $\Bar{\nuV}_t$ by
    \begin{align}
        &\int_{\Gamma_{B\cup D}\times (B\cup D)}G(\xi,x)\Bar{\nuV}_t(\mathrm{d}\xi,\mathrm{d} x)\nonumber\\
        &\quad=\int_{\Gamma_D}\int_{\Gamma_{ B}\times B}G(\xi_1+\xi_2,x_1)\Bar{\nuV}_t^B(\mathrm{d}\xi_1, \mathrm{d}x_1)\Bar{\muP}_t^D(\mathrm{d}\xi_2)+ \int_{\Gamma_B}\int_{\Gamma_{ D}\times D}G(\xi_1+\xi_2,x_2)\Bar{\nuV}_t^D(\mathrm{d}\xi_2, \mathrm{d}x_2)\Bar{\muP}_t^B(\mathrm{d}\xi_1)\label{eq:300}
    \end{align}
    for all $G\in\CCC_{b,0}(\Gamma_{B\cup D}\times (B\cup D))$.
    Then $({\Bar{\muP}},{\Bar{\nuV}})\in\CCC\EEE^{B\cup D}(\dP_0,\dP_1)$.
\end{lem}
\begin{proof}
We show that properties $(i)$ - $(iv)$ of Definition \ref{dfn:ce} are satisfied for $({\Bar{\muP}},{\Bar{\nuV}})$.
Properties $(i)$ and $(iv)$ are satisfied by the definition of $({\Bar{\muP}},{\Bar{\nuV}})$. Property $(ii)$, the $\PPP_1(\Gamma_{B\cup D})$-continuity of $\Bar{\dP}$, follows from the standard approach using Proposition \ref{prop:convergencep1}. It only remains to show that $({\Bar{\muP}},{\Bar{\nuV}})$ solves the continuity equation on $B\cup D.$ Therefore, 
   let $\phi\in\CCC_c^\infty((0,1))$ and $f\in\CCC_c^+(B\cup D).$ We note  that 
   \begin{align*}
       \int_{0}^1 \dt \phi(t)\int_{\Gamma_{B\cup D}} &e^{-\int_{B\cup D}f(x)\xi(\mathrm{d}x)}\Bar{\muP}_t(\mathrm{d}\xi)\mathrm{d}t\\
   &=
       \int_{0}^1 \dt \phi(t)\int_{\Gamma_{D}}\int_{\Gamma_{B}} e^{-\int_{B\cup D}f(x)(\xi_1+\xi_2)(\mathrm{d}x)}\Bar{\muP}_t^B(\mathrm{d}\xi_1)\Bar{\muP}_t^D(\mathrm{d}\xi_2)\mathrm{d}t
       \\ &=  \int_{0}^1 \dt \phi(t)\int_{\Gamma_{B}} e^{-\int_{B} f(x)\xi_1(\mathrm{d}x)}\Bar{\muP}_t^B(\mathrm{d}\xi_1) \int_{\Gamma_{D}} e^{-\int_{D}f(x)\xi_2(\mathrm{d}x)}\Bar{\muP}_t^D(\mathrm{d}\xi_2)\mathrm{d}t.
   \end{align*}
   By using the product rule for weak derivatives, we get 
   \begin{align*}
    &  \int_{0}^1 \dt \phi(t)\int_{\Gamma_{B}} e^{-\int_{B} f(x)\xi_1(\mathrm{d}x)}\Bar{\muP}_t^B(\mathrm{d}\xi_1) \int_{\Gamma_{D}} e^{-\int_{D}f(x)\xi_2(\mathrm{d}x)}\Bar{\muP}_t^D(\mathrm{d}\xi_2)\mathrm{d}t\\
       &=  -\int_{0}^1  \phi(t)\int_{\Gamma_{B}\times B} D_{z_1} e^{-\int_{B} f(x)\xi_1(\mathrm{d}x)}\Bar{\nuV}_t^B(\mathrm{d}\xi_1, \mathrm{d}z_1) \int_{\Gamma_{D}} e^{-\int_{D}f(x)\xi_2(\mathrm{d}x)}\Bar{\muP}_t^D(\mathrm{d}\xi_2)\mathrm{d}t\\
       & \quad -  \int_{0}^1  \phi(t)\int_{\Gamma_{B}} e^{-\int_{B} f(x)\xi_1(\mathrm{d}x)}\Bar{\muP}_t^B(\mathrm{d}\xi_1) \int_{\Gamma_{D}\times D} D_{z_2} e^{-\int_{D}f(x)\xi_2(\mathrm{d}x)}\Bar{\nuV}_t^D(\mathrm{d}\xi_2, \mathrm{d}z_2)\mathrm{d}t\\
       &=- \int_{0}^1  \phi(t)\int_{\Gamma_D}\int_{\Gamma_{B}\times B} D_{z_1} e^{-\int_{B\cup D} f(x)\xi_1+ \xi_2(\mathrm{d}x)}\Bar{\nuV}_t^B(\mathrm{d}\xi_1, \mathrm{d}z_1) \Bar{\muP}_t^D(\mathrm{d}\xi_2)\mathrm{d}t\\
       &\quad -  \int_{0}^1  \phi(t)\int_{\Gamma_B}\int_{\Gamma_{D}\times D} D_{z_2} e^{-\int_{B\cup D} f(x)\xi_1+ \xi_2(\mathrm{d}x)} \Bar{\nuV}_t^D(\mathrm{d}\xi_2, \mathrm{d}z_2) \Bar{\muP}_t^B(\mathrm{d}\xi_1)\mathrm{d}t\\
       &=-\int_0^1\phi(t)\int_{\Gamma_{B\cup D}\times (B\cup D)}D_z e^{-\int_{B\cup D}f(x)\xi(\mathrm{d}x)}\Bar{\nuV}_t(\mathrm{d}\xi,\mathrm{d}z)\mathrm{d}t.
   \end{align*}
   Combining both equations finishes the proof.
\end{proof}
\subsection{Stationary Solutions to the Continuity Equation}
When considering $\dP_0,\dP_1\in\PPP_s(\Gamma)$, of particular interest are solutions $(\Bar{\dP},\Bar{\nuV})\in\CCC\EEE(\dP_0,\dP_1)$ s.t.\ the stationarity of $\dP_0$ and $\dP_1$ is preserved along the curve $\Bar{\dP}.$ In the following, we prove that this is satisfied, if the velocity fields $\Bar{\nuV}$ are stationary.
\begin{dfn}
Let $z\in\R^d$ and define $\Theta_z^{\Gamma\times \R^d}:\Gamma\times\R^d\to\Gamma\times\R^d$ by
  \begin{align}
        \Theta_z^{\Gamma\times \R^d}(\xi,x):=(\Theta_z\xi,x-z).\label{eq:502}
        \end{align}
A measure $\nuV\in\MMM_{b,0}(\Gamma\times \R^d)$ is stationary, if $ (\Theta_z^{\Gamma\times \R^d})_\#\nuV=\nuV$ for all $z\in\R^d$. The set of all such measures is denoted by $\MMM_s(\Gamma\times\R^d)$
\end{dfn}
\begin{lem}\label{lem:shiftinvarianceproperty}
    Let $\dP_0,\dP_1\in\PPP_s(\Gamma)$ and $({\Bar{\muP}},{\Bar{\nuV}})\in\CCC\EEE(\dP_0,\dP_1).$ If $\Bar{\nuV}_t\in\MMM_s(\Gamma\times \R)$ for all $t\in[0,1],$
then $\Bar{\muP}_t\in\PPP_s(\Gamma)$ for all $t\in[0,1]$. 
\end{lem}
\begin{proof}
 As the Laplace functional determines  $\Bar{\muP}_t$ uniquely, $\Bar{\muP}_t$ is stationary if and only if \begin{equation}\label{eq:1}
     \int e^{-\int f(x)\xi(\mathrm{d}x)}\Bar{\muP}_t(\mathrm{d}\xi)= \int e^{-\int f(x-z)\xi(\mathrm{d}x)}\Bar{\muP}_t(\mathrm{d}\xi), 
 \end{equation}  
 for all $z\in\R^d, \ f\in\CCC_c^+(\R^d)$.
 Therefore, we will prove the claim by showing that \eqref{eq:1} holds for all $t\in[0, 1].$

 The stationarity of $\Bar{\nuV}_t$ implies that 
 \begin{align}\label{eq:30}
     \int G(\xi,x)\Bar{\nuV}_t(\mathrm{d}\xi,\mathrm{d}x)= \int G(\Theta_z\xi,x-z)\Bar{\nuV}_t(\mathrm{d}\xi,\mathrm{d}x),
 \end{align}
 for all $G\in\FFF_{b,0}(\Gamma\times \R^d)$ and $z\in\R^d.$ 
 Let $z\in\R^d$ and $\phi\in\CCC_c^\infty((0,1)).$  Using the continuity equation and equation \eqref{eq:30}, we find that \begin{align*}
     \int_0^1\dt \phi(t)\int e^{-\int f(x)\xi(\mathrm{d}x)}\Bar{\muP}_t(\mathrm{d}\xi)\mathrm{d}t&=-\int_0^1 \phi(t)\int (e^{-f(y)}-1)e^{-\int f(x)\xi(\mathrm{d}x)}\Bar{\nuV}_t(\mathrm{d}\xi,\mathrm{d}y)\mathrm{d}t\\
     &=-\int_0^1 \phi(t)\int (e^{-f(y-z)}-1)e^{-\int f(x-z)\xi(\mathrm{d}x)}\Bar{\nuV}_t(\mathrm{d}\xi,\mathrm{d}y)\mathrm{d}t\\
     &= \int_0^1\dt \phi(t)\int e^{-\int f(x-z)\xi(\mathrm{d}x)}\Bar{\muP}_t(\mathrm{d}\xi)\mathrm{d}t.
 \end{align*}
 By  partial integration and the fundamental lemma of calculus of variation, we can deduce for a.e.\ $t\in(0,1)$ that 
 \begin{equation}\label{eq:13}
     \int e^{-\int f(x)\xi(\mathrm{d}x)}\Bar{\muP}_t(\mathrm{d}\xi)= \int e^{-\int f(x-z)\xi(\mathrm{d}x)}\Bar{\muP}_t(\mathrm{d}\xi) + C_{z,f},
 \end{equation}
  where $C_{z,f}$ is a constant that depends on $z$ and $f.$
  As $({\Bar{\muP}},{\Bar{\nuV}})\in\CCC\EEE(\dP_0,\dP_1),$ we know that $(\Bar{\muP}_t)_{t\in[0,1]}$ is a $\PPP_1(\Gamma)$-continuous curve in $\PPP_1(\Gamma)$. This implies that \begin{equation*}
     t\mapsto \int e^{-\int f(x)\xi(\mathrm{d}x)}\Bar{\muP}_t(\mathrm{d}\xi)-\int e^{-\int f(x-z)\xi(\mathrm{d}x)}\Bar{\muP}_t(\mathrm{d}\xi) \end{equation*}
     is a continuous function on $[0,1]$, see Proposition \ref{prop:convergencep1}. Therefore,  equation \eqref{eq:13} holds for all $t\in[0,1].$
   As $\Bar{\muP}_0$ is stationary, it follows that \begin{equation}
     \int e^{-\int f(x)\xi(\mathrm{d}x)}\Bar{\muP}_0(\mathrm{d}\xi)-\int e^{-\int f(x-z)\xi(\mathrm{d}x)}\Bar{\muP}_0(\mathrm{d}\xi)=0.\end{equation}
 We can deduce that $C_{z,f}=0,$ which proves the claim.
\end{proof}

The previous Lemma justifies the following definition of stationary solutions.
\begin{dfn}\label{dfn:stationarysolutions}
    Let $\dP_0,\dP_1\in\PPP_s(\Gamma).$ The set of stationary solutions $\CCC\EEE^s(\dP_0,\dP_1)$ is the set of all solutions $(\Bar{\dP},\Bar{\nuV})\in\CCC\EEE(\dP_0,\dP_1)$ s.t.\ $\Bar{\nuV}_t\in\MMM_s(\Gamma\times\R^d)$ for all $t\in[0,1].$ 
\end{dfn}

\begin{kor}\label{cor:stationarysolution}
  Let  $\dP_0, \dP_1\in\PPP_s(\Gamma).$ Then $\CCC\EEE^s(\dP_0,\dP_1)$ is non-empty.
\end{kor}
\begin{proof}
    The proof follows from the construction of the solution in Theorem \ref{thm:solution}, noting that $\Bar{\nuV}_t\in\mathcal M_s(\Gamma\times \R^d)$ for all $t\in[0,1].$
\end{proof}

    \section{The Action Functional}
   In this section, we define and discuss the action functional that will be used to construct the Benamou--Brenier-type distance for stationary measures in Section \ref{sec:distance}.
   
  Let us briefly recall
   the definition of the Lagrange functional $\mathcal{L}$ defined in \eqref{eq:dfnlagrange}. We introduced the lower semi-continuous, convex, and positively homogeneous function $\alpha(x,y,w)=|w|^2/\theta(x,y),\ w\in\R, \ x,y\geq0$, where $\theta$ is the logarithmic mean and $0/0=0$ by convention.  For $\dP\in\PPP_1(\Gamma_B)$ and $\nuV\in\MMM_{b,0}(\Gamma_B\times B)$ the Lagrange functional is defined by
   \begin{align*}
          \mathcal{L}({\dP},{\mathbf{V}}):=\int\alpha\left(\frac{\mathrm{d}\dP\otimes \m_{|B}}{\text{d}\sigma},\frac{\mathrm{d}C_\dP}{\text{d}\sigma},\frac{\text{d}{\nuV}}{\text{d}\sigma}\right)\text{d}\sigma,
       \end{align*}
      where $\sigma\in\MMM_{b,0}(\Gamma_B\times B)$ is non-negative s.t.\  $\dP\otimes {\m}_{|B},C_\dP,\mathbf{V}\ll\sigma.$
    \begin{dfn}
       Let $\dP_0,\dP_1\in\PPP_1(\Gamma)$ and $({\Bar{\muP}},{\Bar{\nuV}})\in\CCC\EEE(\dP_0,\dP_1).$ 
      The action functional of $({\Bar{\muP}},{\Bar{\nuV}})$ is defined by
        \begin{align}
            \AAA({\Bar{\muP}},{\Bar{\nuV}})&:=\sup_{n\in\N}\frac{1}{\lambda(\Lambda_n)} \int_0^1\mathcal{L}({\Bar{\muP}_{t|\Lambda_n}},{\Bar{\nuV}_{t|\Lambda_n}})\mathrm{d}t.\nonumber\\
            &= \sup_{n\in\N}\frac{1}{\lambda(\Lambda_n)}  \mathrm{A}({\Bar{\muP}_{|\Lambda_n}},{\Bar{\nuV}_{|\Lambda_n}}).\label{eq:501}
        \end{align}
     \end{dfn}
     \begin{bem}
          Theorem \ref{thm:equalitystationary}  entails
         that the supremum in the definition of the action functional $\mathcal{A}$ can be replaced by a limit when considering stationary solutions to the continuity equation.
     \end{bem}

\subsection{Properties of the Lagrange functional}
  For a more detailed study of the action functional introduced above, we begin by discussing properties of the Lagrange functional. The following Lemma shows that the Lagrange functional is shift-invariant.
     \begin{lem}\label{lem:shiftinvaraiance}
    Let $B\in\BBB(\R^d)$ be a Polish set, $\dP_1\in\PPP_1(\Gamma_{B})$, $\dP_2\in\PPP_1(\Gamma_{B+z})$, ${\nuV}_1\in\MMM_{b,0}(\Gamma_{B}\times B)$, and ${\nuV}_2\in\MMM_{b,0}(\Gamma_{B+z}\times (B+z))$. Assume that $\dP_1=({\Theta_{z}})_\#\dP_2$ and ${\nuV}_1=({\Theta_{z}^{\Gamma\times \R^d}})_\#{\nuV}_2.$ Then $$\mathcal{L}(\dP_1,{\nuV}_1)=\mathcal{L}(({\Theta_{z}})_\#\dP_2,({\Theta_{z}^{\Gamma\times \R^d}})_\#{\nuV}_2)=\mathcal{L}(\dP_2,{\nuV}_2).$$
    \end{lem}
    \begin{proof}
        The proof follows directly by choosing $\sigma_1=({\Theta_z^{\Gamma\times \R^d}})_\#\sigma_2$ in the definition of the Lagrange functional.
    \end{proof}
     
     Moreover,  the Lagrange functional is super-additive on disjoint domains in the following sense:
    \begin{lem}\label{lem:superadditivity}
        Let $B,D\in\BBB(\R^d)$ be disjoint Polish sets. Assume that ${\muP}\in\PPP_1(\Gamma_{B\cup D})$ s.t.\ ${\muP}\ll\Poi_{|B\cup D}$ with ${\muP}=\rho\Poi_{|B\cup D}$ and ${\nuV}\in\MMM_{b,0}(\Gamma_{B\cup D}\times (B\cup D))$ s.t.\ ${\nuV}\ll\Poi_{|B\cup D}\otimes \m_{|B\cup D}$ with ${\nuV}=w(\Poi_{|B\cup D}\otimes \m_{|B\cup D}).$ Then
        \begin{align*}
            \mathcal{L}({{\muP}},{{\nuV}})\geq \mathcal{L}({\muP}_{|B},{\nuV}_{|B})+\mathcal{L}({\muP}_{|D},{\nuV}_{|D}).
        \end{align*}
    \end{lem}
    \begin{proof}
    Given that ${\muP}\ll\Poi_{|B\cup D}$ and ${\nuV}\ll\Poi_{|B\cup D}\otimes \m,$ we can use the representation of the Lagrange functional in equation \eqref{eq:representationlagrange} and obtain that 
     \begin{align}
      \mathcal{L}({\muP} ,{\nuV})
             &=\int_{\Gamma_{B\cup D}}\int_{ B\cup D}\alpha(\rho(\xi),\rho(\xi+\delta_z), w(\xi,z))\m (\mathrm{d}z)\Poi_{|B\cup D}( \mathrm{d}\xi) \nonumber\\
             &=\int_{\Gamma_{B}}\int_{\Gamma_D}\int_{ B\cup D}\alpha(\rho(\xi_1+\xi_2),\rho(\xi_1+\xi_2+\delta_z), w(\xi_1+\xi_2,z))\m( \mathrm{d}z)\Poi_{|D}( \mathrm{d}\xi_2) \nonumber\Poi_{|B}(\mathrm{d}\xi_1)\\
             &= \int_{\Gamma_{B}}\int_{\Gamma_D}\int_{ B}\alpha(\rho(\xi_1+\xi_2),\rho(\xi_1+\xi_2+\delta_z), w(\xi_1+\xi_2,z))\m( \mathrm{d}z)\Poi_{|D}( \mathrm{d}\xi_2) \Poi_{|B}(\mathrm{d}\xi_1)\nonumber\\
             & \quad + \int_{\Gamma_{B}}\int_{\Gamma_D}\int_{ D}\alpha(\rho(\xi_1+\xi_2),\rho(\xi_1+\xi_2+\delta_z), w(\xi_1+\xi_2,z))\m( \mathrm{d}z)\Poi_{|D}( \mathrm{d}\xi_2) \Poi_{|B}(\mathrm{d}\xi_1).\label{eq:110}
     \end{align}
     By Lemma \ref{lem:apprestricteddensities}  we have that ${\muP}_{|B}\ll\Poi_{|B}$  s.t.\   ${\muP}_{|B}=\rho^B\Poi_{|B}$ with
      \begin{align*}
\rho^B(\xi_1)=\int_{\Gamma_D}\rho(\xi_1+\xi_2)\Poi_{|D}(\mathrm{d}\xi_2),
      \end{align*}
      and ${\nuV}_{|B}\ll\Poi_{|B}\otimes \m_{|B}$  s.t.\ ${\nuV}_{|B}=w^B(\Poi_{|B}\otimes \m_{|B})$ with
      \begin{align*}
    w^B(\xi_1,z)=\int_{\Gamma_D}w(\xi_1+\xi_2,z)\Poi_{|D}(\mathrm{d}\xi_2).
      \end{align*}
      Using the convexity of $\alpha$, see Section \ref{sec:distanceDSHS}, together with Jensen's inequality and the previous observation,  the first summand of \eqref{eq:110} can be estimated as follows
        \begin{align*}
            &\int_{\Gamma_{B}}\int_{\Gamma_D}\int_{ B}\alpha(\rho(\xi_1+\xi_2),\rho(\xi_1+\xi_2+\delta_z), w(\xi_1+\xi_2,z))\m( \mathrm{d}z)\Poi_{|D}( \mathrm{d}\xi_2) \Poi_{|B}(\mathrm{d}\xi_1)\\&\geq
           \int_{\Gamma_{B}}\int_{ B}\hspace{-0.09cm}\alpha\hspace{-0.05cm}\left(\hspace{0.11cm}\int_{\Gamma_D}\rho(\xi_1+\xi_2)\Poi_{|D}(\mathrm{d}\xi_2),\int_{\Gamma_D}\rho(\xi_1+\xi_2+\delta_z)\Poi_{|D}(\mathrm{d}\xi_2),\int_{\Gamma_D} w(\xi_1+\xi_2,z)\Poi_{|D}(\mathrm{d}\xi_2)\hspace{-0.1cm}\right)\hspace{-0.1cm}\m( \mathrm{d}z)\Poi_{|B}(\mathrm{d}\xi_1) \\
            &= \int_{\Gamma_{B}}\int_{ B}\alpha(\rho^B(\xi_1),\rho^B(\xi_1+\delta_z), w^B(\xi_1,z))\m( \mathrm{d}z)\Poi_{|B}(\mathrm{d}\xi_1)\\
            &=\mathcal{L}({\muP}_{|B},{\nuV}_{|B}).
             \end{align*}
            Analogously, we obtain 
            \begin{align*}
                \int_{\Gamma_{B}}\int_{\Gamma_D}\int_{ D}\alpha(\rho(\xi_1+\xi_2),\rho(\xi_1+\xi_2+\delta_z), w(\xi_1+\xi_2,z))\m( \mathrm{d}z)\Poi_{|D}( \mathrm{d}\xi_2) \Poi_{|B}(\mathrm{d}\xi_1)\geq \mathcal{L}({\muP}_{|D},{\nuV}_{|D}).
            \end{align*}Inserting both estimates in \eqref{eq:110} proves the claim. 
    \end{proof}
    The following lemma establishes that the Lagrange functional is even additive, when applied to product measures. 
     \begin{lem}\label{lem:additivityLagrangefunctional}
       Let $B,D\in\BBB(\R^d)$ be disjoint and Polish sets, ${\muP}\in\PPP_1(\Gamma_{B\cup D})$ s.t.\ ${\muP}={\muP}_{|B}\otimes{\muP}_{|D}$ and  let ${\nuV}\in\MMM_{b,0}(\Gamma_{B\cup D}\times(B\cup D)).$ Assume that  ${\muP}_{|B}\ll\Poi_{|B}$ with ${\muP}_{|B}=\rho^B\Poi_{|B}$  and that ${\muP}_{|D}\ll\Poi_{|D}$   with ${\muP}_{|D}=\rho^D\Poi_{|D}.$ Further, let  ${\nuV}_{|B}\ll\Poi_{|B}\otimes \m_{|B}$ with ${\nuV}_{|B}= w^B(\Poi_{|B}\otimes \m_{|B})$ and ${\nuV}_{|D}\ll\Poi_{|D}\otimes \m_{|D}$  with ${\nuV}_{|D}= w^D(\Poi_{|D}\otimes \m_{|D}).$ Define $\Tilde{\mathbf{V}}\in\MMM_{b,0}(\Gamma_{B\cup D}\times(B\cup D))$ as in \eqref{eq:300}.
    Then it holds that 
    \begin{align*}
        \mathcal{L}({\muP}, \Tilde{\mathbf{V}})=\mathcal{L}({\muP}_{|B},{\nuV}_{|B})+\mathcal{L}({\muP}_{|D},{\nuV}_{|D}).
    \end{align*}
     \end{lem}
     \begin{proof}
    As ${\muP}={\muP}_{|B}\otimes{\muP}_{|D},$ ${\muP}_{|B}\ll\Poi_{|B}$, and  ${\muP}_{|D}\ll\Poi_{|D}$,  we have by Lemma \ref{app:productdensities} that ${\muP}\ll\Poi_{|B\cup D}$ with \begin{align*}
        {\muP}= \rho^B(\cdot_{|B})\rho^D(\cdot_{|D})\Poi_{|B\cup D}.
        \end{align*}
        By the same lemma, it holds that $ \Tilde{\mathbf{V}}\ll \Poi_{|B\cup D}\otimes \m_{|B\cup D} $ with
        \begin{align*}
            \Tilde{\mathbf{V}}=(\ind_B(\cdot) w^B(\cdot_{|B},\cdot)\rho^D(\cdot_{|D})+\ind_D(\cdot) w^D(\cdot_{|D},\cdot)\rho^B(\cdot_{|B}))\Poi_{|B\cup D}\otimes \m_{|B\cup D}.
        \end{align*}
         Using the representation of the Lagrange functional from equation \eqref{eq:representationlagrange}, we can deduce that \begin{align}
             \mathcal{L}({\muP},\Tilde{\mathbf{V}})
             &= \int_{B}\int_{\Gamma_D}\int_{\Gamma_{B}}\alpha(\rho^B(\xi_1)\rho^D(\xi_2), \rho^B(\xi_1+\delta_z)\rho^D(\xi_2), w^B(\xi_1,z)\rho^D(\xi_2))\Poi_{|B}(\mathrm{d}\xi_1)\Poi_{|D}(\mathrm{d}\xi_2)\m(\mathrm{d}z)\nonumber\\
              &\quad +\int_{D}\int_{\Gamma_D}\int_{\Gamma_{B}}\alpha(\rho^B(\xi_1)\rho^D(\xi_{2}), \rho^B(\xi_1)\rho^D(\xi_2+\delta_z), w^D(\xi_{2},z)\rho^B(\xi_{1}))\Poi_{|B}(\mathrm{d}\xi_1)\Poi_{|D}(\mathrm{d}\xi_2)\m(\mathrm{d}z).\label{eq:600}
              \end{align} 
              By utilizing the positive homogeneity property of $\alpha$, see equation \eqref{eq:propalpha}, and that $\rho^D$ is a probability density w.r.t.\ $\Poi_{|D}$, we obtain that \begin{align*}
              &\int_{B}\int_{\Gamma_D}\int_{\Gamma_{B}}\alpha(\rho^B(\xi_1)\rho^D(\xi_2), \rho^B(\xi_1+\delta_z)\rho^D(\xi_2), w^B(\xi_1,z)\rho^D(\xi_2))\Poi_{|B}(\mathrm{d}\xi_1)\Poi_{|D}(\mathrm{d}\xi_2)\m(\mathrm{d}z)\\
              &=\int_{B}\int_{\Gamma_D}\int_{\Gamma_{B}}\rho^D(\xi_2)\alpha(\rho^B(\xi_1),\rho^B(\xi_1+\delta_z), w^B(\xi_1,z))\Poi_{|B}(\mathrm{d}\xi_1)\Poi_{|D}(\mathrm{d}\xi_2)\m(\mathrm{d}z)\\
              &=\int_{B}\int_{\Gamma_{B}}\alpha(\rho^B(\xi_1),\rho^B(\xi_1+\delta_z), w^B(\xi_1,z))\Poi_{|B}(\mathrm{d}\xi_1)\m(\mathrm{d}z)\\
              &=\mathcal{L}({\muP}_{|B},{\nuV}_{|B}).
              \end{align*}
                Analogously, we find that 
              \begin{align*} \int_{D}\int_{\Gamma_D}\int_{\Gamma_{B}}\alpha(\rho^B(\xi_1)\rho^D(\xi_{2}), \rho^B(\xi_1)\rho^D(\xi_2+\delta_z), w^D(\xi_{2},z)\rho^B(\xi_{1}))\Poi_{|B}(\mathrm{d}\xi_1)\Poi_{|D}(\mathrm{d}\xi_2)\m(\mathrm{d}z)=\mathcal{L}({\muP}_{|D},{\nuV}_{|D}).
         \end{align*}
        Inserting both equalities in \eqref{eq:600}  finishes the proof.
     \end{proof}
     \subsection{Properties of the Action Functional}
   After studying the Lagrange functional in detail, we now shift our focus to the action functional. We discuss several properties of the action functional to deduce the existence of minimizers (Corollary \ref{kor:minimizergeneral}) as the main result of this section. More precisely, we show the following properties.

   \begin{lem}\label{lem:action}
        Let $\dP_0,\dP_1\in\PPP_1(\Gamma).$ The action functional $\mathcal{A}$ satisfies the following properties:
         \begin{enumerate}

             \item[$({i})$] The action functional $\mathcal{A}$ is jointly convex, i.e.\ 
             \begin{align*}
                 \mathcal{A}(t{\Bar{\muP}}^1+(1-t){\Bar{\muP}}^2,t{\Bar{\nuV}}^1+(1-t){\Bar{\nuV}}^2)\leq t\mathcal{A} ({\Bar{\muP}}^1,{\Bar{\nuV}}^1)+(1-t) \mathcal{A}({\Bar{\muP}}^2,{\Bar{\nuV}}^2)
             \end{align*}
             for all $({\Bar{\muP}}^1,{\Bar{\nuV}}^1),({\Bar{\muP}}^2,{\Bar{\nuV}}^2)\in\CCC\EEE(\dP_0,\dP_1),\ t\in[0,1].$ 
             \item[$({ii})$] The action functional  $\mathcal{A}$ is  lower semi-continuous on $\CCC\EEE(\dP_0,\dP_1)$.
             \item[$({iii})$] The action functional  $\mathcal{A}$ has compact sublevel sets on $\CCC\EEE(\dP_0,\dP_1).$
         \end{enumerate}
     \end{lem}
   The previous lemma implies the existence of minimizers of the action functional.
      \begin{kor}\label{kor:minimizergeneral}
     Let $\dP_0,\dP_1\in\PPP_1(\Gamma)$ and assume that $\inf\left\{\mathcal{A}({\Bar{\muP}},{\Bar{\nuV}}): ({\Bar{\muP}},{\Bar{\nuV}})\in\CCC\EEE(\dP_0,\dP_1)\right\}<\infty.$  There exists some $({\Bar{\muP}}^*,{\Bar{\nuV}}^*)\in\CCC\EEE(\dP_0,\dP_1)$ s.t.\
       \begin{align*}
           \inf\left\{\mathcal{A}({\Bar{\muP}},{\Bar{\nuV}}): ({\Bar{\muP}},{\Bar{\nuV}})\in\CCC\EEE(\dP_0,\dP_1)\right\}=\mathcal{A}({\Bar{\muP}}^*,{\Bar{\nuV}}^*).
       \end{align*}
     \end{kor}
     \begin{proof}
         The proof follows from the direct method of calculus of variation and Lemma \ref{lem:action}.
     \end{proof}
     We continue with the proof of Lemma \ref{lem:action}.
     \begin{proof}[Proof of Lemma \ref{lem:action}]
         $(i)$ By Lemma \ref{lem:actionDSHS} $(i)$, we know that 
         \begin{align*}
          \frac{1}{\lambda(\Lambda_n)}  \mathrm{A}({\Bar{\muP}_{|\Lambda_n}},{\Bar{\nuV}}_{|\Lambda_n})=
          \frac{1}{\lambda(\Lambda_n)} \int_0^1\mathcal{L}({\Bar{\muP}_{t|\Lambda_n}},{\Bar{\nuV}_{t|\Lambda_n}})\mathrm{d}t
         \end{align*}
         is jointly convex for any $n\in\N$. The convexity of $\mathcal{A}$ follows immediately, as the supremum of convex functions is again convex.\\
         
         $(ii)$ Let $({\Bar{\muP}}^m,{\Bar{\nuV}}^m)_{m\in\N}\subset\CCC\EEE(\dP_0,\dP_1)$  s.t.\ $\Bar{\muP}_t^m\to\Bar{\muP}_t$ in $\PPP_1(\Gamma)$ as $m\to\infty $ and $\Bar{\nuV}_t^m\to\Bar{\nuV}_t$ in $\MMM_{b,0}(\Gamma\times \R^d)$ as $m\to\infty $.
         By Lemma \ref{lem:restriction} it holds that $({{{\Bar{\muP}}}^m}_{|\Lambda_n},{{{\Bar{\nuV}}}^m}_{|\Lambda_n})_{m\in\N}\subset \CCC\EEE^{\Lambda_n}({\dP_0}_{|\Lambda_n},{\dP_1}_{|\Lambda_n})$ for all $n\in\N.$ We deduce from Lemma \ref{lem:actionDSHS} $(ii)$ that 
         \begin{align*}
             \int_0^1\mathcal{L}({\Bar{\muP}_{t|\Lambda_n}},{\Bar{\nuV}_{t|\Lambda_n}})\mathrm{d}t\leq \liminf_{m\to\infty}\int_0^1\mathcal{L}({\Bar{\muP}_{t|\Lambda_n}^m},{\Bar{\nuV}_{t|\Lambda_n}^m})\mathrm{d}t 
         \end{align*}
         for all $n\in\N.$  As the supremum of lower semi-continuous functions is again lower semi-continuous, this implies the lower semi-continuity i.e.\
         \begin{align*}
             \sup_{n\in\N}\frac{1}{\lambda(\Lambda_n)}\int_0^1\mathcal{L}({\Bar{\muP}_{t|\Lambda_n}},{\Bar{\nuV}_{t|\Lambda_n}})\mathrm{d}t\leq \liminf_{m\to\infty}\sup_{n\in\N}\frac{1}{\lambda(\Lambda_n)} \mathcal{L}({\Bar{\muP}_{t|\Lambda_n}^m},{\Bar{\nuV}_{t|\Lambda_n}^m})\mathrm{d}t.
         \end{align*}\\
         
        $(iii)$  Let $C\in\R$  and consider a sequence $({\Bar{\muP}}^m,{\Bar{\nuV}}^m)_{m\in\N}\subset \CCC\EEE(\dP_0,\dP_1)$ s.t.\
        \begin{align} \label{eq:18}
            \sup_{n\in\N}\frac{1}{\lambda(\Lambda_n)}\int_0^1\mathcal{L}({\Bar{\muP}_{t|\Lambda_n}^m},{\Bar{\nuV}_{t|\Lambda_n}^m})\mathrm{d}t\leq C
        \end{align}
        for all $m\in\N.$ To prove the claim, we need to construct a sequence of natural numbers $(h(l))_{l\in\N}$ and a pair $(\Bar{\dP},\Bar{\nuV})\in\CCC\EEE(\dP_0,\dP_1)$ s.t.\ $({\Bar{\muP}}^{h(l)},{\Bar{\nuV}}^{h(l)})_{l\in\N}$ converges to $(\Bar{\dP},\Bar{\nuV})$ and $(\Bar{\dP},\Bar{\nuV})$  satisfies equation \eqref{eq:18}.  
        
       We first note that equation \eqref{eq:18} implies that
        \begin{align*}
            \int_0^1\mathcal{L}({\Bar{\muP}_{t|\Lambda_1}^m},{\Bar{\nuV}_{t|\Lambda_1}^m})\mathrm{d}t\leq C
        \end{align*} for all $m\in\N.$ By Lemma \ref{lem:actionDSHS} $(iii)$, there exists a subsequence $(m_1(l))_{l\in\N}$ and some   $({\Bar{\muP}}^{1},{\Bar{\nuV}}^{1})\in\CCC\EEE^{\Lambda_1}({\dP_0}_{|\Lambda_1},{\dP_1}_{|\Lambda_1})$ s.t.\
        \begin{align*}
            {\Bar{\muP}_{t|\Lambda_1}^{m_1(l)}}&\underset{m\to\infty}{\to} \Bar{\muP}_t^{1}\quad \quad \text{in }\mathcal{P}_1(\Gamma_{\Lambda_1}) \text{ for all } t\in[0,1],\\
            {\Bar{\nuV}_{t|\Lambda_1}^{m_1(l)}}&\underset{m\to\infty}{\to} \Bar{\nuV}_t^{1} \quad \quad \text{in }\mathcal{M}_{b,0}(\Gamma_{\Lambda_1}\times \Lambda_1 ) \text{ for all } t\in[0,1].
        \end{align*}
        Inductively, assume that we have a subsequence $(m_{n-1}(l))_{l\in\N}$ of $(m_{n-2}(l))_{l\in\N}$ constructed in the above way.  By equation \eqref{eq:18} it holds that
        \begin{align*}
            \int_0^1\mathcal{L}({\Bar{\muP}_{t|\Lambda_n}^{m_{n-1}(l)}},{\Bar{\nuV}_{t|\Lambda_n}^{m_{n-1}(l)}})\mathrm{d}t\leq C\lambda(\Lambda_n)
        \end{align*}
        for all $l\in\N.$ Lemma \ref{lem:actionDSHS} $(iii)$ again implies the existence of a subsequence $(m_n(l))_{l\in\N}$ of $(m_{n-1}(l))_{l\in\N}$ and measures  $({\Bar{\muP}}^{n},{\Bar{\nuV}}^{n})\in\CCC\EEE^{\Lambda_n}({\dP_0}_{|\Lambda_n},{\dP_1}_{|\Lambda_n})$ s.t.\ \begin{align*}
            {\Bar{\muP}_{t|\Lambda_n}^{m_n(l)}}&\underset{l\to\infty}{\to} \Bar{\muP}_t^{n}\quad \quad \text{in }\mathcal{P}_1(\Gamma_{\Lambda_n}) \text{ for all } t\in[0,1],\\
            {\Bar{\nuV}_{t|\Lambda_n}^{m_n(l)}}&\underset{l\to\infty}{\to} \Bar{\nuV}_t^{n} \quad \quad \text{in }\mathcal{M}_{b,0}(\Gamma_{\Lambda_n}\times \Lambda_n ) \text{ for all } t\in[0,1].
        \end{align*}
        For all $l\in\N$ define $h(l):=m_l(l)$,  and note that 
        \begin{align*}
            {\Bar{\muP}_{t|\Lambda_n}^{h(l)}}&\underset{l\to\infty}{\to}\Bar{\muP}^{n}_t\quad \quad \text{in } \mathcal{P}_1(\Gamma_{\Lambda_n}), \\
            {\Bar{\nuV}_{t|\Lambda_n}^{h(l)}}&\underset{l\to\infty}{\to}\Bar{\nuV}^{n}_t \quad \quad \text{in }\mathcal{M}_{b,0}(\Gamma_{\Lambda_n}\times \Lambda_n )
        \end{align*}
       for any $n\in\N.$ Moreover,  for any $m<n,$ we have that ${\Bar{\muP}^{n}_{t|\Lambda_m}}= {\Bar{\muP}^{m}_t}$  \text{ and }  ${\Bar{\nuV}^{n}_{t|\Lambda_m}}= {\Bar{\nuV}^{m}_t}.$ Hence, $(\Bar{\muP}^{n}_t)_{n\in\N}$ and $(\Bar{\nuV}^{n}_t)_{n\in\N}$ are families of consistent measures. By {Kolmogorov's consistency theorem}, there exists some measure $\Bar{\muP}_t\in\PPP_1(\Gamma)$ s.t.\ ${\Bar{\muP}_{t|\Lambda_n}}={\Bar{\muP}^{n}_t}.$ For the construction of the measure $\Bar{\nuV}_t$,  we note that the sets of the form \begin{align*}
             \A_{B,k}^{a,b}:=\{\xi\in\Gamma: \xi(B)=k\}\times ([a_1,b_1)\times\ldots\times[a_d,b_d)),\end{align*} with $B\in\BBB_0(\R^d)$ and $a_i\leq b_i\in\R$ are the generator of $\BBB(\Gamma\times \R^d).$   For any such set there exists some $n\in\N$ s.t.\ $B,[a_1,b_1)\times\ldots\times[a_d,b_d)\subseteq \Lambda_n.$ We define the measure $\Bar{\nuV}_t$ on the generating sets by \begin{align*}
            \Bar{\nuV}_t(\A_{B,k}^{a,b}):=\Bar{\nuV}_t^{n}(\A_{B,k}^{a,b}) .\end{align*}  
           By definition, it holds that ${\Bar{\nuV}_{t|\Lambda_n}}={\Bar{\nuV}^{n}_t}.$    Note that  \begin{align*}
            {\Bar{\muP}_{t|\Lambda_n}^{h(l)}}&\underset{l\to\infty}{\to}{\Bar{\muP}_{t|\Lambda_n}}\quad \quad\text{ in } \mathcal{P}_1(\Gamma_{\Lambda_n}) \\
        {\Bar{\nuV}_{t|\Lambda_n}^{h(l)}}&\underset{l\to\infty}{\to}{\Bar{\nuV}_{t|\Lambda_n}}\quad \quad \text{ in } \mathcal{M}_{b,0}(\Gamma_{\Lambda_n}\times \Lambda_n) \end{align*}for any $n\in\N.$   This implies  
        ${\Bar{\muP}_t^{h(l)}}\to{\Bar{\muP}_t}$  in $\mathcal{P}_1(\Gamma)$ and ${\Bar{\nuV}_t^{h(l)}}\to{\Bar{\nuV}_t}$  in $\mathcal{M}_{b,0}(\Gamma\times \R^d)$ as $l\to\infty$. By Theorem \ref{thm:propertiesce} $(iv),$ we can deduce that $({\Bar{\muP}},{\Bar{\nuV}}) \in \CCC\EEE(\dP_0,\dP_1).$ Moreover, by construction it holds that 
        \begin{align*}
          \sup_{n\in\N}\frac{1}{\lambda(\Lambda_n)}\int_0^1\mathcal{L}({\Bar{\muP}_{t|\Lambda_n}},{\Bar{\nuV}_{t|\Lambda_n}})\mathrm{d}t\leq C, 
        \end{align*}
        finishing the proof.
        \end{proof}
       
\section{Transport Distance for Stationary Measures}\label{sec:distance}
In this section, we define a Benamou--Brenier-type distance for stationary measures and study their properties.
    \begin{dfn}\label{dfn:stationary}
         Let $\dP_0,\dP_1\in\PPP_s(\Gamma)$. We define
         \begin{align*}
             \mathcal{W}^2_s (\dP_0,\dP_1)&:= \inf\{\AAA({\Bar{\muP}},{\Bar{\nuV}}): ({\Bar{\muP}},{\Bar{\nuV}})\in\CCC\EEE^s(\dP_0,\dP_1)\}\\
             &=\inf\left\{\sup_{n\in\N}\frac{1}{\lambda(\Lambda_n)}\int_0^1\mathcal{L}({\Bar{\muP}_{t|\Lambda_n}},{\Bar{\nuV}_{t|\Lambda_n}})\mathrm{d}t: ({\Bar{\muP}},{\Bar{\nuV}})\in\CCC\EEE^s( \dP_0,\dP_1)\right\}.
         \end{align*}
     \end{dfn}
 By a   standard reparametrization argument, see e.g.\ \cite[Theorem 5.4]{Dolbeault_2008}, \cite[Lemma 1.1.4]{Ambrosio2005GradientFI} we have
      \begin{lem}
        Let $\dP_0,\dP_1\in\PPP_s(\Gamma).$ It holds that 
        \begin{align}\label{eq:31}
            \mathcal{W}_s(\dP_0,\dP_1)=\inf\left\{\sup_{n\in\N}\frac{1}{\lambda(\Lambda_n)^{1/2}}\int_0^T\mathcal{L}^{1/2}({\Bar{\muP}_{t|\Lambda_n}},{\Bar{\nuV}_{t|\Lambda_n}})\mathrm{d}t:({\Bar{\muP}},{\Bar{\nuV}})\in\CCC\EEE^s_T(\dP_0,\dP_1)\right\}
        \end{align}
        for all $T>0.$
        \end{lem}
        The next theorem shows that the infimum in the definition of $\mathcal{W}_s$ is in fact a minimum.
   \begin{satz} \label{thm:minimizer}
       Let $\dP_0,\dP_1\in\PPP_s(\Gamma)$. If  $\mathcal{W}_s(\dP_0,\dP_1)<\infty,$ then there exists  some $({\Bar{\muP}}^*,{\Bar{\nuV}}^*)\in\CCC\EEE^s(\dP_0,\dP_1)$ s.t.\ 
    \begin{align*}
      \mathcal{W}_s(\dP_0,\dP_1)=\mathcal{A}({\Bar{\muP}}^*,{\Bar{\nuV}}^*).
    \end{align*}
    \end{satz}
    \begin{proof}
       If $\CCC\EEE^s(\dP_0,\dP_1)$ is a closed subset of $\CCC\EEE(\dP_0,\dP_1),$  the claim follows from Corollary \ref{kor:minimizergeneral}.
         Therefore, let $({\Bar{\muP}}^n,{\Bar{\nuV}}^n)_{n\in\N}\subset \CCC\EEE^s(\dP_0,\dP_1)$ be a sequence of solutions to the continuity equation such  that $\Bar{\muP}^n{\to} \Bar{\muP}$ in $\PPP_1(\Gamma)$ as $n\to\infty$ and ${\Bar{\nuV}}^n\to{\Bar{\nuV}}$ in $\MMM_{b,0}(\Gamma\times\R^d)$ as $n\to\infty$. By Theorem \ref{thm:propertiesce} $(iv)$ we have that $({\Bar{\muP}},{\Bar{\nuV}})\in\CCC\EEE(\dP_0,\dP_1).$ 
        It remains to prove that $\Bar{\nuV}_t\in\MMM_s(\Gamma\times \R^d)$ for all $t\in[0,1].$ 
         As $\Bar{\nuV}_t^n\in\MMM_s(\Gamma\times \R^d)$ for all $n\in\N$ and $t\in[0,1]$, we obtain 
        \begin{align*}
         \int G(\xi,z)\Bar{\nuV}_t(\mathrm{d}\xi,\mathrm{d}z)&= \lim_{n\to\infty}\int G(\xi,x)\Bar{\nuV}_t^n(\mathrm{d}\xi,\mathrm{d}x)\\
         &= \lim_{n\to\infty}\int G(\Theta_z\xi,x-z)\Bar{\nuV}_t^n(\mathrm{d}\xi,\mathrm{d}x)\\
         &=\int G(\Theta_z\xi,x-z)\Bar{\nuV}_t(\mathrm{d}\xi,\mathrm{d}x)
        \end{align*}
        for any $ z\in\R^d,$ $G\in \CCC_{b,0}(\Gamma\times\R^d),$ and $t\in[0,1]$, which finishes the proof.
    \end{proof}
        \begin{satz}
         The map $\mathcal{W}_s$ is an extended distance on $\PPP_s(\Gamma).$
     \end{satz}
     \begin{proof}
          The symmetry of $\mathcal{W}_s$ follows directly from the definition. We continue with the proof of the positive definiteness of $\mathcal{W}_s.$ Let $\dP_0,\dP_1\in\PPP_s(\Gamma)$ s.t.\ $\mathcal{W}_s(\dP_0,\dP_1)=0.$ By Theorem \ref{thm:minimizer} there exists some minimizer $({\Bar{\muP}}^{*},{\Bar{\nuV}}^{*})$ s.t.\ \begin{align*}
             \mathcal{W}_s^2 (\dP_0,\dP_1)=\sup_{n\in\N}\frac{1}{\lambda(\Lambda_n)}\int_0^1\mathcal{L}({\Bar{\muP}_{t|\Lambda_n}^*},{\Bar{\nuV}_{t|\Lambda_n}^*})\mathrm{d}t.
         \end{align*}
         This implies that 
         \begin{align*}
             \frac{1}{\lambda(\Lambda_n)}\int_0^1\mathcal{L}({\Bar{\muP}_{t|\Lambda_n}^*},{\Bar{\nuV}_{t|\Lambda_n}^*})\mathrm{d}t=0
         \end{align*}
          for all $n\in\N.$ As the Lagrange functional is non-negative, we can deduce that $\mathcal{L}({\Bar{\muP}_{t|\Lambda_n}^*},{\Bar{\nuV}_{t|\Lambda_n}^*})=0$ for all $n\in\N$ and a.e.\ $t\in[0,1].$ By the definition of the Lagrange functional, this implies $\Bar{\nuV}_{t|\Lambda_n}^*=0$ for all $n\in\N$ and a.e.\ $t\in[0,1].$ Since $({\Bar{\muP}}^{*},{\Bar{\nuV}}^{*})$ solves the continuity equation and $\Bar{\muP}^*$ is $\PPP_1(\Gamma)$-continuous, it follows that ${\dP_0}_{|\Lambda_n}={\dP_1}_{|\Lambda_n}.$ 
          Analogously to \cite[Theorem 5.15]{schiavo2023wasserstein}, the triangle inequality follows by Theorem \ref{thm:minimizer} and concatenation of solutions to the continuity equation.
         \end{proof}
       
        \begin{lem}
        The distance $\mathcal{W}_s$ satisfies the following properties:
        
             \begin{enumerate}

                 \item[${(i)}$]  The topology induced by $\mathcal{W}_s$ on $\PPP_s(\Gamma)$ is stronger than the usual topology  on $\PPP_s(\Gamma).$  
                 \item[${(ii)}$] $\mathcal{W}_s$ is lower semi-continuous on $\PPP_s(\Gamma)\times \PPP_s(\Gamma)$. 
                 \item[${(iii)}$] 
                 Bounded sets w.r.t.\ $\mathcal{W}_s$ are $\PPP_s$-relatively compact.
                 \item[${(iv)}$]  For every $\dP \in \PPP_s(\Gamma)$ the accessible component $(\{\mathcal{W}_s(\dP, \cdot) < \infty\}, \mathcal{W}_s) $ is a complete geodesic space.
            \end{enumerate}
         \end{lem}
         \begin{proof}
         
        $(i)$ Let $(\dP^m)_{m\in\N}\subset\PPP_s(\Gamma)$ and $\dP\in\PPP_s(\Gamma)$ s.t.\ $\mathcal{W}_s(\dP^m,\dP)\to0$ as $m\to\infty.$ By Theorem \ref{thm:minimizer}, for any $m\in\N$ there exists some $({\Bar{\muP}}^{m,*},\Bar{\nuV}^{m,*})\in\CCC\EEE^s(\dP^m,\dP)$ s.t.\ 
             \begin{align*}
                 \mathcal{W}_s^2(\dP^m,\dP)=\sup_{n\in\N}\frac{1}{\lambda(\Lambda_n)}\int_0^1\mathcal{L}({\Bar{\muP}^{m,*}_{t|\Lambda_n}},{\Bar{\nuV}^{m,*}_{t|\Lambda_n}})\mathrm{d}t.
             \end{align*}
             
            Thus, we obtain that 
            \begin{align*}
             \lim_{m\to\infty}\mathcal{W}_s(\dP^m,\dP)=  \lim_{m\to\infty} \frac{1}{\lambda(\Lambda_n)}\int_0^1\mathcal{L}({\Bar{\muP}^{m,*}_{t|\Lambda_n}},{\Bar{\nuV}^{m,*}_{t|\Lambda_n}})\mathrm{d}t=0
            \end{align*}
            for all $n\in\N.$ This implies that 
            \begin{align*}
                \lim_{m\to\infty}\int_0^1\mathcal{L}({\Bar{\muP}^{m,*}_{t|\Lambda_n}},{\Bar{\nuV}^{m,*}_{t|\Lambda_n}})\mathrm{d}t=0.
            \end{align*}
            As the restriction of a solution to the continuity equation on $\R^d$ is again a solution to the continuity equation on the restricted domain (Lemma \ref{lem:restriction}), we  conclude that 
            \begin{align*}
                \lim_{m\to\infty}\inf\left\{\int_0^1\mathcal{L}(\Bar{\muP}_t,\Bar{\nuV}_t)\mathrm{d}t:({\Bar{\muP}},{\Bar{\nuV}})\in\CCC\EEE^{\Lambda_n}(\dP^m_{|\Lambda_n},\dP_{|\Lambda_n})\right\}=0
            \end{align*}
            for any $n\in\N.$ In other words, we have $\mathcal{W}_0(\dP^m_{|\Lambda_n},\dP_{|\Lambda_n})\to0$ as $m\to\infty$. Using \cite[Theorem 5.15]{schiavo2023wasserstein},  for  all $n\in \N$ it follows that $\dP_{|\Lambda_n}^m{\to} \dP_{|\Lambda_n}$ \text{in } $\PPP_1(\Gamma_{\Lambda_n})$ as $m\to\infty.$ This implies the claim. \\

            $(ii)$ The lower semi-continuity of $\mathcal{W}_s$ follows from the lower semi-continuity of the action functional $\mathcal{A}$ (Lemma \ref{lem:action})  and the existence of minimizers realizing $\mathcal{W}_s(\dP_0,\dP_1)$ (Theorem \ref{thm:minimizer}).\\

            $(iii)$ The claim is a consequence of Lemma \ref{lem:action}.\\
            
            $(iv)$ The completeness follows from  $(ii)$ and $(iii)$. By the structure of the distance and the existence of minimizers (Theorem \ref{thm:minimizer}), we can deduce that   $(\{\mathcal{W}_s(\dP, \cdot) < \infty\}, \mathcal{W}_s) $ is a complete geodesic space for every $\dP \in \PPP_s(\Gamma)$. The geodesics are given by the minimizing curves.
         \end{proof}
     
     We conclude this section by providing another reformulation of $\mathcal{W}_s$. 
     
     \begin{lem}\label{lem:sup=limstationär}
         Let $\dP_0,\dP_1\in\PPP_s(\Gamma).$ It holds that 
         \begin{align*}
             \mathcal{W}^2_s (\dP_0,\dP_1)&=\inf\left\{\sup_{n\in\N}\frac{1}{\lambda(\Lambda_n)}\int_0^1\mathcal{L}({\Bar{\muP}_{t|\Lambda_n}},{\Bar{\nuV}_{t|\Lambda_n}})\mathrm{d}t: ({\Bar{\muP}},{\Bar{\nuV}})\in\CCC\EEE^s( \dP_0,\dP_1)\right\}\\
             & =\inf\left\{\lim_{n\to\infty}\frac{1}{\lambda(\Lambda_n)}\int_0^1\mathcal{L}({\Bar{\muP}_{t|\Lambda_n}},{\Bar{\nuV}_{t|\Lambda_n}})\mathrm{d}t: ({\Bar{\muP}},{\Bar{\nuV}})\in\CCC\EEE^s( \dP_0,\dP_1)\right\}.
         \end{align*}
         Moreover,  the limit
         \begin{align*}
             \lim_{n\to\infty}\frac{1}{\lambda(\Lambda_n)}\int_0^1\mathcal{L}({\Bar{\muP}_{t|\Lambda_n}},\Bar{\nuV}_{t|\Lambda_n})\mathrm{d}t
         \end{align*}
         exists in $[0,\infty]$ for every $({\Bar{\muP}},{\Bar{\nuV}})\in\CCC\EEE^s(\dP_0,\dP_1).$
     \end{lem}
    The proof is based on  Fekete's lemma, see e.g.\ \cite[p.74]{fekete}

     \begin{lem}[Fekete´s Lemma]\label{lem:fekete}
         Let $ \mathcal{R}$ be the collection of rectangles in $\R^d.$ Suppose that  for $\Lambda\in\mathcal{R}$ we have $a_{\Lambda}\in[0,\infty]$ that satisfy
         \begin{enumerate}
             \item[${(i)}$] $a_{\Lambda}+a_{\Lambda^\prime}\leq a_{\Lambda\cup\Lambda^\prime}$ for $\Lambda\cap\Lambda^\prime=\emptyset$ and $\Lambda,\Lambda^\prime\in\mathcal{R};$
             \item [${(ii)}$] $a_{\Lambda}=a_{\Lambda+z}$ for all $z\in\Z^d$ and $\Lambda\in\mathcal{R}.$ 
         \end{enumerate}
         Then \begin{align*}
             \lim_{n\to\infty}\frac{a_{\Lambda_n}}{\lambda(\Lambda_n)}=\sup_{\Lambda\in\mathcal{R}}\frac{a_{\Lambda}}{\lambda(\Lambda)}.
         \end{align*}
     \end{lem}

     \begin{proof}[Proof of Lemma \ref{lem:sup=limstationär}]
     The proof  follows from the super-additivity of the Lagrange functional (Lemma \ref{lem:superadditivity}), the shift-invariance of the Lagrange functional (Lemma \ref{lem:shiftinvaraiance}), and Fekete's Lemma.
         \end{proof}
          At this point, it is important that the infimum in the definition of $\mathcal{W}_s$ is taken over stationary solutions to the continuity equation. If this were not the case, an application of Fekete's Lemma would not be possible.
          
        \section{Connection between Transport Distances}
        In the following, we discuss the relationship between the transport distances $\mathcal{W}_0$ and  $\mathcal{W}_s$ in more detail. The key result of this section is the following theorem. 
        \begin{satz}\label{thm:equalitystationary}
        Let $\dP_0,\dP_1\in\PPP_s(\Gamma)$ s.t.\ $\mathrm{Ent}({\dP_0}_{|\Lambda_n}|\Poi_{|\Lambda_n}),\mathrm{Ent}({\dP_1}_{|\Lambda_n}|\Poi_{|\Lambda_n})<\infty$ for all $n\in\N.$ Then 
        \begin{align*}
            \mathcal{W}_s^2(\dP_0,\dP_1)&= \lim_{n\to\infty}\frac{1}{\lambda(\Lambda_n)}\mathcal{W}_0^2({\dP_0}_{|\Lambda_n},{\dP_1}_{|\Lambda_n}).
        \end{align*}
    \end{satz}
   The equality of Theorem \ref{thm:equalitystationary} can be reformulated as
    \begin{align}
          &\inf\left\{\lim_{n\to\infty}\frac{1}{\lambda(\Lambda_n)}\mathrm{A}({\Bar{\muP}_{|\Lambda_n}},{\Bar{\nuV}_{|\Lambda_n}}):({\Bar{\muP}},{\Bar{\nuV}})\in\CCC\EEE^s(\dP_0,\dP_1)\right\}\nonumber\\
          &\quad=\lim_{n\to\infty}\frac{1}{\lambda(\Lambda_n)}\inf\left\{\mathrm{A}({\Bar{\muP}^n},{\Bar{\nuV}^n}):({\Bar{\muP}^n},{\Bar{\nuV}^n})\in\CCC\EEE^{\Lambda_n}(\dP_{0|\Lambda_n},\dP_{1|\Lambda_n})\right\}.\label{eq:700}
    \end{align}
     The key difference between the left-hand side and the right-hand side of \eqref{eq:700} is that the order of the limit and the infimum is interchanged. On the left-hand side, the infimum is taken over global, stationary solutions to the continuity equation, whereas on the right-hand side it is taken over local solutions on finite regions. Remarkably,  the constraint of stationarity does not appear explicitly on the right-hand side.  It follows from this result that the minimizer of the right-hand side can be chosen to be stationary.
     
        For the proof of Theorem \ref{thm:equalitystationary}, we will establish upper and lower bounds for $\mathcal{W}_s$ in terms of $\mathcal{W}_0$. We begin by discussing the lower bound for $\mathcal{W}_s$.

            \begin{lem}\label{lem:boundcands}
            Let $\dP_0,\dP_1\in\PPP_s(\Gamma).$ Then
            \begin{align*}
                \mathcal{W}_s^2(\dP_0,\dP_1)\geq \frac{1}{\lambda(\Lambda_k)}\mathcal{W}_0^2({\dP_0}_{|\Lambda_k},{\dP_1}_{|\Lambda_k})
            \end{align*}
            for all $k\in\N$.
        \end{lem}
        \begin{proof}
        Let $k\in\N.$ As the restrictions of solutions to the continuity equation are again solutions on the restricted domain (Lemma \ref{lem:restriction}), we directly get 
        \begin{align*}
            \mathcal{W}_s^2(\dP_0,\dP_1)&=\inf\left\{\sup_{n\in\N}\frac{1}{\lambda(\Lambda_n)}\int_0^1\mathcal{L}({\Bar{\muP}_{t|\Lambda_n}},{\Bar{\nuV}_{t|\Lambda_n}})\mathrm{d}t:({\Bar{\muP}},{\Bar{\nuV}})\in\CCC\EEE^s(\dP_0,\dP_1)\right\}\\
            &\geq \inf\left\{\frac{1}{\lambda(\Lambda_k)}\int_0^1\mathcal{L}({\Bar{\muP}_{t|\Lambda_k}},{\Bar{\nuV}_{t|\Lambda_k}})\mathrm{d}t:({\Bar{\muP}},{\Bar{\nuV}})\in\CCC\EEE^s(\dP_0,\dP_1)\right\}\\
            &\geq \inf\left\{\frac{1}{\lambda(\Lambda_k)}\int_0^1\mathcal{L}({\Bar{\muP}_t^k},{\Bar{\nuV}_t^k})\mathrm{d}t:({\Bar{\muP}}^k,{\Bar{\nuV}}^k)\in\CCC\EEE^{\Lambda_k}({\dP_0}_{|\Lambda_k},{\dP_1}_{|\Lambda_k})\right\}\\
            &= \frac{1}{\lambda(\Lambda_k)}\mathcal{W}_0^2({\dP_0}_{|\Lambda_k},{\dP_1}_{|\Lambda_k}).
        \end{align*}
        \end{proof}

       \subsection{Tiling and Stationarizing}
          To establish the upper bound of $\mathcal{W}_s$ in terms of $\mathcal{W}_0$, we consider both distances for special types of measures, namely tiled and stationarized measures, introduced, for example, in \cite{MartinJonasBasti}.
        \begin{dfn}
            Let $k\in\N$  and $\dP\in\PPP_1(\Gamma_{\Lambda_k}).$  We define the tiled measure $\dP^{\mathrm{til}}$ of $\dP$ by 
             \begin{align*}
                \int_{\Gamma} F(\xi)\dP^{\mathrm{til}}(\mathrm{d}\xi)= \int_{\left(\Gamma_{\Lambda_k}\right)^{\Z^d}}F\left(\sum_{z\in\Z^d}\Theta_{kz}\xi_z \right)\bigotimes_{z\in\Z^d}\dP(\mathrm{d}\xi_z)
             \end{align*}
           for all $F\in\CCC_1(\Gamma)$. The stationarized version $\dP^{\text{stat}}$ of $\dP$ is defined by 
            \begin{align*}
               \int_{\Gamma} F(\xi)\dP^{\mathrm{stat}}(\mathrm{d}\xi)  = \frac{1}{\lambda(\Lambda_k)}\int_{\Lambda_k} \int_{\Gamma}F(\Theta_u\xi)\dP^{\mathrm{til}}(\mathrm{d}\xi)\lambda(\mathrm{d}u)
            \end{align*}     
            for all $F\in\CCC_1(\Gamma).$ In other words,
            \begin{align*}
                \dP^{\mathrm{stat}}=\frac{1}{\lambda(\Lambda_k)}\int_{\Lambda_k}({\Theta_u})_\#\dP^{\mathrm{til}}\lambda(\mathrm{d}u).
            \end{align*}  
        \end{dfn}
        \begin{bem}\label{bem:limstat}
        \begin{itemize}
            \item[$(i)$]
            Let $\dP\in\PPP_1(\Gamma_{\Lambda_k})$ be the distribution of a point process. Then   $\dP^{\mathrm{til}}$ is the distribution of a point process obtained by independently generating a point process according to $\dP$ within each cube of side length $k$ and then gluing all of them together. The point process with distribution $\dP^\text{stat}$ is constructed by averaging $\dP^{\mathrm{til}}$ over an independent shift that is uniformly distributed on the cube $\Lambda_k$. This averaging step ensures that  $\dP^{\mathrm{stat}}$ is stationary.

        \item[$(ii)$] 
             Assume that $\dP\in\PPP_s(\Gamma)$ and consider $(\dP_{|\Lambda_k})^\mathrm{stat}.$ Then, using  Proposition \ref{prop:convergencep1}, we can deduce that $(\dP_{|\Lambda_k})^{\mathrm{stat}}\to\dP$ in $\mathcal P_1(\Gamma).$
        \item[$(iii)$] Note that the sets $\Lambda_k+z, z\in(k\Z)^d$ are not disjoint since we need them to be closed to be able to apply the results from \cite{schiavo2023wasserstein}. However, their intersection has at most dimension $d-1$, hence, zero Lebesgue measure. Therefore, any stationary point process will have no point in this intersection with probability one.
             \end{itemize}
         \end{bem}
         Due to the continuity equation, we have a close connection between the paths of measures in $\PPP_1(\Gamma)$ and 
         velocity fields in $\MMM_{b,0}(\Gamma \times \R^d)$.  Therefore, analogously to tiled and stationarized processes, we define tiled and stationarized velocity fields.
         \begin{dfn}
            Let $k\in\N,$ $\dP\in\PPP_1(\Gamma_{\Lambda_k})$,  and ${\nuV}\in\MMM_{b,0}(\Gamma_{\Lambda_k}\times \Lambda_k).$ We define the tiled velocity field ${\nuV}^\text{til}$ of ${\nuV}$ w.r.t.\ $\dP$ by 
             \begin{align*}
                &\int_{\Gamma\times \R^d} G(\xi,x){\nuV}^{\mathrm{til}}(\mathrm{d}\xi,\mathrm{d}x)\\
                &= \sum_{z_1\in\Z^d}\int_{\left(\Gamma_{\Lambda_k}\right)^{\Z^d\setminus \{z_1\}}}\int_{\Gamma_{\Lambda_k}\times \Lambda_k}G\left(\Theta_{kz_1}\xi_{z_1}+\sum_{z_2\in\Z^d\setminus\{z_1\}}\Theta_{kz_2}\xi_{z_2}, x-z_1\right){\nuV}(\mathrm{d}\xi_{z_1},\mathrm{d}x)\bigotimes_{z_2\in\Z^d\setminus\{z_1\}}\dP(\mathrm{d}\xi_{z_2})
             \end{align*}
           for all $G\in\CCC_{b,0}(\Gamma\times \R^d)$. The stationarized version $\dP^{\text{stat}}$ of $\dP$ is defined by 
            \begin{align*}
               \int_{\Gamma} G(\xi){\nuV}^{\mathrm{stat}}(\mathrm{d}\xi)  = \frac{1}{\lambda(\Lambda_k)}\int_{\Lambda_k} \int_{\Gamma}G(\Theta_u\xi,x-u){\nuV}^{\mathrm{til}}(\mathrm{d}\xi,\mathrm{d}x)\lambda(\mathrm{d}u)
            \end{align*}     
            for all $G\in\CCC_{b,0}(\Gamma\times \R^d).$ In other words, 
            \begin{align*}
                {\nuV}^{\mathrm{stat}}=\frac{1}{\lambda(\Lambda_k)}\int_{\Lambda_k}({\Theta_u^{\Gamma\times \R^d}})_\#{\nuV}^{\mathrm{til}}\lambda(\mathrm{d}u),
            \end{align*}
            with $\Theta_u^{\Gamma\times \R^d}$ as defined in \eqref{eq:502}.
        \end{dfn}

        To begin our discussion on tiled and stationarized measures, we construct solutions to the continuity equation, which connect these types of measures. The following lemma demonstrates that if there exists a solution to the continuity equation connecting two measures, then the tiled and stationarized versions of this solution will also connect the  tiled and stationarized versions of the measures.
         \begin{lem}\label{lem:stationarysolutions}
             Let $\dP_0, \dP_1\in\PPP_1(\Gamma_{\Lambda_k})$. If $({\Bar{\muP}},{\Bar{\nuV}})\in\CCC\EEE^{\Lambda_k}(\dP_0, \dP_1),$ then also $({\Bar{\muP}}^{\mathrm{til}},{\Bar{\nuV}}^{\mathrm{til}})\in\CCC\EEE(\dP_0^{\mathrm{til}}, \dP_1^{\mathrm{til}})$ and $({\Bar{\muP}}^{\mathrm{stat}},{\Bar{\nuV}}^{\mathrm{stat}})\in\CCC\EEE^s(\dP_0^{\mathrm{stat}}, \dP_1^{\mathrm{stat}})$.
         \end{lem}
         \begin{proof}
             The proof of $({\Bar{\muP}}^{\mathrm{til}},{\Bar{\nuV}}^{\mathrm{til}})\in\CCC\EEE(\dP_0^{\mathrm{til}}, \dP_1^{\mathrm{til}})$ follows directly from Lemma \ref{lem:superpositionofsolution}. We continue with the proof of the second claim.  Properties $(i)$ and $(iv)$ of Definition \ref{dfn:ce} follow directly from the definition of $({\Bar{\muP}}^{\mathrm{stat}},{\Bar{\nuV}}^{\mathrm{stat}}).$ Property $(ii)$ can be deduced by using the usual approach involving Proposition \ref{prop:convergencep1}. We continue with the proof of property $(iii)$ and show that $({\Bar{\muP}}^{\mathrm{stat}},{\Bar{\nuV}}^{\mathrm{stat}})$ is indeed a solution to the continuity equation. Let  $f\in\CCC_c^+(\R^d)$ and $\phi\in \mathcal{C}_c^\infty((0,1)).$ By using Fubini and that $({\Bar{\muP}}^{\mathrm{til}},{\Bar{\nuV}}^{\mathrm{til}})\in\CCC\EEE(\dP_0^{\mathrm{til}}, \dP_1^{\mathrm{til}})$, we obtain that
             \begin{align*}
                 \int_0^1 \dt \phi(t)\int e^{-\int f(x)\xi(\mathrm{d}x)}\Bar{\muP}_t^{\mathrm{stat}}(\mathrm{d}\xi)\mathrm{d}t &= \int_{\Lambda_k}\frac{1}{\lambda(\Lambda_k)} \int_0^1 \dt \phi(t)\int e^{-\int f(x-u)\xi(\mathrm{d}x)}\Bar{\muP}_t^{\mathrm{til}}(\mathrm{d}\xi)\mathrm{d}t \lambda(\mathrm{d}u)\\
                 &= -  \int_{\Lambda_k}\frac{1}{\lambda(\Lambda_k)} \int_0^1  \phi(t)\int D_z e^{-\int f(x-u)\xi(\mathrm{d}x)}\Bar{\nuV}_t^{\mathrm{til}}(\mathrm{d}\xi,\mathrm{d}z)\mathrm{d}t \lambda(\mathrm{d}u)\\
                 &=- \int_{\Lambda_k}\frac{1}{\lambda(\Lambda_k)} \int_0^1  \phi(t)\int D_{z-u} e^{-\int f(x)\Theta_u\xi(\mathrm{d}x)}\Bar{\nuV}_t^{\mathrm{til}}(\mathrm{d}\xi,\mathrm{d}z)\mathrm{d}t \lambda(\mathrm{d}u)\\
                 &=  -\int_0^1  \phi(t)\int D_{z} e^{-\int f(x)\xi(\mathrm{d}x)}\Bar{\nuV}_t^{\mathrm{stat}}(\mathrm{d}\xi,\mathrm{d}z)\mathrm{d}t \lambda(\mathrm{d}u).
             \end{align*}
            Finally, the stationarity of the solution follows by construction, which finishes the proof.
         \end{proof}
         Having these two special types of solutions, we are able to establish an upper bound for $\mathcal{W}_s$ in terms of $\mathcal{W}_0$.
\begin{lem}\label{lem:boundstat}
    Let $k\in\N$, and let $\dP_0,\dP_1\in\PPP_1(\Gamma_{\Lambda_k})$ s.t.\ $\mathrm{Ent}({\dP_0}|\Poi_{|\Lambda_k}),\mathrm{Ent}({\dP_1}|\Poi_{|\Lambda_k})<\infty$. 
     It holds that 
     \begin{align*}
        \mathcal{W}_s^2(\dP_0^\mathrm{stat},\dP_1^\mathrm{stat})\leq \frac{1}{\lambda(\Lambda_k)}\mathcal{W}_0^2(\dP_0,\dP_1).
     \end{align*}
\end{lem}
\begin{proof}
As $\mathrm{Ent}({\dP_0}|\Poi_{|\Lambda_k}),\mathrm{Ent}({\dP_1}|\Poi_{|\Lambda_k})<\infty$, we can deduce from Theorem \ref{thm:talagrandDSHS} that $\mathcal{W}_0^2(\dP_0,\dP_1)<\infty.$
By Theorem  \ref{thm:minimizercompact}   there exists some $({\Bar{\muP}}^{*,k},{\Bar{\nuV}}^{*,k})\in\CCC\EEE^{\Lambda_k}(\dP_0,\dP_1)$ s.t.\ 
    \begin{align*}
        \mathcal{W}_0^2(\dP_0,\dP_1)= \inf\left\{\int_0^1\mathcal{L}({\Bar{\muP}_t^{k}}, {\Bar{\nuV}_t^{k}})\mathrm{d}t:({\Bar{\muP}}^k,{\Bar{\nuV}}^k)\in\CCC\EEE^{\Lambda_k}(\dP_0,\dP_1)\right\}=\int_0^1\mathcal{L}({\Bar{\muP}_t^{*,k}}, {\Bar{\nuV}_t^{*,k}})\mathrm{d}t.
    \end{align*}
   Lemma \ref{lem:stationarysolutions}  entails that $(({\Bar{\muP}}^{*,k})^\mathrm{stat},({\Bar{\nuV}}^{*,k})^\mathrm{stat})\in\CCC\EEE^s(\dP^\mathrm{stat}_0,\dP^\mathrm{stat}_1)$.
   Combining these observations with Lemma \ref{lem:sup=limstationär} and using the shorthand notation $\fint_{\Lambda_k}=\int_{\Lambda_k} \frac{1}{\lambda(\Lambda_k)}$ yields
    \begin{align}
    &\mathcal{W}_s^2(\dP_0^{\mathrm{stat}},\dP_1^{\mathrm{stat}})= \inf\left\{\lim_{n\to\infty}\frac{1}{\lambda(\Lambda_{(2n-1)k})}\int_0^1\mathcal{L}({\Bar{\muP}_{t|\Lambda_{(2n-1)k}}}, {\Bar{\nuV}_{t|\Lambda_{(2n-1)k}}})\mathrm{d}t: ({\Bar{\muP}},{\Bar{\nuV}})\in\CCC\EEE^s(\dP_0^{\mathrm{stat}},\dP_1^{\mathrm{stat}})\right\}\nonumber\\
        &\leq  \lim_{n\to\infty}\frac{1}{\lambda(\Lambda_{(2n-1)k})}\int_0^1\mathcal{L}((\Bar{\muP}^{*,k}_t)^\mathrm{stat}_{|\Lambda_{(2n-1)k}}, (\Bar{\nuV}^{*,k}_t)^\mathrm{stat}_{|\Lambda_{(2n-1)k}})\mathrm{d}t\nonumber\\
        &=  \lim_{n\to\infty}\frac{1}{\lambda(\Lambda_{(2n-1)k})}\int_0^1 \mathcal{L}\left( \hspace{0.1cm} {\fint_{\Lambda_k}{\Theta_u}_\#(\Bar{\muP}^{*,k}_t)^\mathrm{til}}\lambda(\mathrm{d}u)_{|\Lambda_{(2n-1)k}}, \fint_{\Lambda_k}({\Theta_u^{\Gamma\times \R^d}})_\#(\Bar{\nuV}^{*,k}_t)^\mathrm{til} \lambda(\mathrm{d}u)_{|\Lambda_{(2n-1)k}}\right)\mathrm{d}t\nonumber\\
        &=  \lim_{n\to\infty} \frac{1}{\lambda(\Lambda_{(2n-1)k})} \int_0^1 \mathcal{L} \left( \hspace{0.1cm}{\fint_{\Lambda_k}(({\Theta_u})_\#(\Bar{\muP}^{*,k}_t)^\mathrm{til}})_{|\Lambda_{(2n-1)k}} \lambda(\mathrm{d}u),  \fint_{\Lambda_k}(({\Theta_u^{\Gamma\times \R^d}})_\#(\Bar{\nuV}^{*,k}_t)^\mathrm{til})_{|\Lambda_{(2n-1)k}}  \lambda(\mathrm{d}u) \right) \mathrm{d}t.\label{eq:204}\end{align}
        Using the convexity of the Lagrange functional, see Section \ref{sec:distanceDSHS}, and Jensen's inequality, we obtain
        \begin{align}
           &  \lim_{n\to\infty} \frac{1}{\lambda(\Lambda_{(2n-1)k})} \int_0^1 \mathcal{L} \left( \hspace{0.1cm} {\fint_{\Lambda_k}(({\Theta_u})_\#(\Bar{\muP}^{*,k}_t)^\mathrm{til}})_{|\Lambda_{(2n-1)k}} \lambda(\mathrm{d}u),  \fint_{\Lambda_k}(({\Theta_u^{\Gamma\times \R^d}})_\#(\Bar{\nuV}^{*,k}_t)^\mathrm{til})_{|\Lambda_{(2n-1)k}}  \lambda(\mathrm{d}u) \nonumber\right) \mathrm{d}t\\
        &\quad\leq \lim_{n\to\infty}\frac{1}{\lambda(\Lambda_{(2n-1)k})}\int_0^1\fint_{\Lambda_k}\mathcal{L}( {(({\Theta_u})_\#(\Bar{\muP}^{*,k}_t)^\mathrm{til}})_{|\Lambda_{(2n-1)k}},(({\Theta_u^{\Gamma\times \R^d}})_\#(\Bar{\nuV}^{*,k}_t)^\mathrm{til})_{|\Lambda_{(2n-1)k}} )\lambda(\mathrm{d}u)\mathrm{d}t. \label{eq:200}
        \end{align}
        In the following, we will make use of the specific form of $(({\Theta_u})_\#(\Bar{\muP}_t^{*,k})^\mathrm{til}, ({\Theta_u^{\Gamma\times \R^d}})_\#(\Bar{\nuV}_t^{*,k})^\mathrm{til})$.   Therefore, let $u\in\Lambda_k.$ We note that  $\Lambda_{(2n-1)k}\subset\Lambda_{(2n+1)k}-u.$
        
        In the next step, we want to apply the super-additivity of the Lagrange functional (Lemma \ref{lem:superadditivity}). Therefore, we first need to ensure that 
     \begin{align*}
     &{({(\Theta_u)}_\#(\Bar{\muP}^{*,k}_t)^\mathrm{til}})_{|\Lambda_{(2n+1)k}-u}\ll\Poi_{|\Lambda_{(2n+1)k}-u} \text{ and } \\
     &({(\Theta_u^{\Gamma\times \R^d})}_\#(\Bar{\nuV}^{*,k}_t)^\mathrm{til})_{|\Lambda_{(2n+1)k}-u}\ll\Poi_{|\Lambda_{(2n+1)k}-u}\otimes\m_{|\Lambda_{(2n+1)k}-u}.
     \end{align*}
     From the convexity of the relative entropy along $\mathcal{W}_0$-geodesics and the assumption   that $\mathrm{Ent}({\dP_0}|\Poi_{|\Lambda_k}),$ $\mathrm{Ent}({\dP_1}|\Poi_{|\Lambda_k})<\infty$,  we can deduce that $\Bar{\muP}^{*,k}_t\ll\Poi_{\Lambda_k}.$ Combining this with Lemmas \ref{app:productdensities}, \ref{lem:apprestricteddensities}, and  \ref{app:shiftdensities}, we obtain that ${({(\Theta_u)}_\#(\Bar{\muP}^{*,k}_t)^\mathrm{til}})_{|\Lambda_{(2n+1)k}-u}\ll\Poi_{|\Lambda_{(2n+1)k}-u}$. Lemmas \ref{lem:additivityLagrangefunctional}, \ref{lem:shiftinvaraiance}, and \ref{lem:absolutecontinuityvelocity} finally implies that also $({(\Theta_u^{\Gamma\times \R^d})}_\#(\Bar{\nuV}^{*,k}_t)^\mathrm{til})_{|\Lambda_{(2n+1)k}-u}\ll\Poi_{|\Lambda_{(2n+1)k}-u}\otimes\m_{(2n+1)k-u}$.
     Thus, we are allowed to use the super-additivity of the Lagrange functional (Lemma \ref{lem:superadditivity}) and get together with the shift-invariance of the Lagrange functional (Lemma \ref{lem:shiftinvaraiance}) that
        \begin{align*}
            &\mathcal{L}( {(({\Theta_u})_\# (\Bar{\muP}^{*,k}_t)^\mathrm{til}})_{|\Lambda_{(2n-1)k}},(({\Theta_u^{\Gamma\times \R^d}})_\#(\Bar{\nuV}^{*,k}_t)^\mathrm{til})_{|\Lambda_{(2n-1)k}} )\\
            &\leq \mathcal{L}( {({(\Theta_u)}_\#(\Bar{\muP}^{*,k}_t)^\mathrm{til}})_{|\Lambda_{(2n+1)k}-u},({(\Theta_u^{\Gamma\times \R^d})}_\#(\Bar{\nuV}^{*,k}_t)^\mathrm{til})_{|\Lambda_{(2n+1)k}-u} )\\
            & =\mathcal{L}( ({\Theta_{-u}})_\#({(({\Theta_u})_\#(\Bar{\muP}^{*,k}_t)^\mathrm{til}})_{|\Lambda_{(2n+1)k}-u}),{(\Theta_{-u}^{\Gamma\times \R^d})}_\#(({(\Theta_u^{\Gamma\times \R^d})}_\#(\Bar{\nuV}^{*,k}_t)^\mathrm{til})_{|\Lambda_{(2n+1)k}-u} ))\\
            &  = \mathcal{L}( {((\Bar{\muP}^{*,k}_t)^\mathrm{til}})_{|\Lambda_{(2n+1)k}},((\Bar{\nuV}^{*,k}_t)^\mathrm{til})_{|\Lambda_{(2n+1)k}} ).
        \end{align*} 
        Inserting this in equation \eqref{eq:200} gives
        \begin{align}
        &\lim_{n\to\infty}\frac{1}{\lambda(\Lambda_{(2n-1)k})}\int_0^1 \fint_{\Lambda_k}\mathcal{L}( {(({\Theta_u})_\#(\Bar{\muP}^{*,k}_t)^\mathrm{til}})_{|\Lambda_{(2n-1)k}},(({\Theta_u^{\Gamma\times \R^d}})_\#(\Bar{\nuV}^{*,k}_t)^\mathrm{til})_{|\Lambda_{(2n-1)k}} )\lambda(\mathrm{d}u)\mathrm{d}t\nonumber\\
            &\leq\lim_{n\to\infty}\frac{1}{\lambda(\Lambda_{(2n-1)k})}\int_0^1 \fint_{\Lambda_k} \mathcal{L}{((\Bar{\muP}^{*,k}_t)^\mathrm{til}})_{|\Lambda_{(2n+1)k}},((\Bar{\nuV}^{*,k}_t)^\mathrm{til})_{|\Lambda_{(2n+1)k}}) \lambda(\mathrm{d}u)\mathrm{d}t\nonumber\\
            &=\lim_{n\to\infty}\frac{1}{\lambda(\Lambda_{(2n-1)k})}\int_0^1\mathcal{L}{((\Bar{\muP}^{*,k}_t)^\mathrm{til}})_{|\Lambda_{(2n+1)k}},((\Bar{\nuV}^{*,k}_t)^\mathrm{til})_{|\Lambda_{(2n+1)k}}) \mathrm{d}t.\label{eq:201}
            \end{align}
        For the tiled processes, we can use the shift-invariance of the Lagrange functional (Lemma \ref{lem:shiftinvaraiance}) and the additivity of the Lagrange functional for independent regions (Lemma \ref{lem:additivityLagrangefunctional}) to obtain that 
        \begin{align*}
           \mathcal{L}{((\Bar{\muP}^{*,k}_t)^\mathrm{til}})_{|\Lambda_{(2n+1)k}},((\Bar{\nuV}^{*,k}_t)^\mathrm{til})_{|\Lambda_{(2n+1)k}})= (2n+1)^d \mathcal{L}( \Bar{\muP}^{*,k}_t,\Bar{\nuV}^{*,k}_t ).
        \end{align*}
        Substituting this equality in \eqref{eq:201} yields
        \begin{align}
        \lim_{n\to\infty}\frac{1}{\lambda(\Lambda_{(2n-1)k})}\int_0^1\mathcal{L}{((\Bar{\muP}^{*,k}_t)^\mathrm{til}})_{|\Lambda_{(2n+1)k}},((\Bar{\nuV}^{*,k}_t)^\mathrm{til})_{|\Lambda_{(2n+1)k}}) \mathrm{d}t
            &= \lim_{n\to\infty}\frac{(2n+1)^d}{\lambda(\Lambda_{(2n-1)k})}\int_0^1\mathcal{L}( \Bar{\muP}^{*,k}_t,\Bar{\nuV}^{*,k}_t )\mathrm{d}t\nonumber\\
            &=\frac{1}{\lambda(\Lambda_{k})}\int_0^1\mathcal{L}( \Bar{\muP}^{*,k}_t,\Bar{\nuV}^{*,k}_t )\mathrm{d}t\nonumber\\
            &=\frac{1}{\lambda(\Lambda_{k})}\mathcal{W}_0(\dP_0,\dP_1).\label{eq:202}
        \end{align}
         Combining \eqref{eq:204}, \eqref{eq:200}, \eqref{eq:201}, and \eqref{eq:202} proves the claim.
        \end{proof}
    We are now able to prove Theorem \ref{thm:equalitystationary}, which deals with the asymptotic equality between the  normalised
    Wasserstein distance $\mathcal{W}_0$ on boxes $\Lambda_n$ and  $\mathcal{W}_s$.
    
    \begin{proof}[Proof of Theorem \ref{thm:equalitystationary}]
        Using Remark \ref{bem:limstat} $(ii)$,  the lower semi-continuity of $\mathcal{W}_s$, Lemma \ref{lem:boundstat}, and Lemma \ref{lem:boundcands}, we obtain
        \begin{align*}
            \mathcal{W}_s^2(\dP_0,\dP_1)&=\mathcal{W}_s^2\left(\lim_{n\to\infty}({\dP_0}_{|\Lambda_n})^\mathrm{stat},\lim_{n\to\infty}({\dP_1}_{|\Lambda_n})^\mathrm{stat}\right)\\
            &\leq \liminf_{n\to\infty}\mathcal{W}_s^2(({\dP_0}_{|\Lambda_n})^\mathrm{stat},({\dP_1}_{|\Lambda_n})^\mathrm{stat})\\
            &\leq \liminf_{n\to\infty} \frac{1}{\lambda(\Lambda_n)}\mathcal{W}_0^2({\dP_0}_{|\Lambda_n},{\dP_1}_{|\Lambda_n})\\
            &\leq \limsup_{n\to\infty} \frac{1}{\lambda(\Lambda_n)}\mathcal{W}_0^2({\dP_0}_{|\Lambda_n},{\dP_1}_{|\Lambda_n})\\
            &\leq \mathcal{W}_s^2(\dP_0,\dP_1).
        \end{align*}
        Therefore, all inequalities have to be equalities, which finishes the proof.
    \end{proof}

    \section{Gradient Flows, Geodesic Convexity, and Functional Inequalities}\label{sec:functionals}
   
    Combining Theorem \ref{thm:equalitystationary} with the results of \cite{schiavo2023wasserstein} recalled in Section \ref{sec:distanceDSHS} we can now easily prove that 
    the Ornstein--Uhlenbeck semigroup is the gradient flow of the specific relative entropy w.r.t.\ $\mathcal{W}_s$, establish 1-geodesic convexity of the specific relative entropy along $\mathcal{W}_s$-geodesics, and derive various functional inequalities such as the specific Talagrand inequality and the specific HWI inequality.

\subsection{The Specific Relative Entropy}
      Let $\dP\in\PPP_1(\Gamma)$ s.t.\ $\dP_{|B}\ll\Poi_{|B}$ for any $B\in\BBB_0(\R^d).$ The \textit{specific relative entropy} w.r.t\ $\Poi$ is defined by
\begin{align}\label{eq:27}
    \mathcal{E}(\dP):=\lim_{n\to\infty}\frac{1}{\lambda(\Lambda_n)}\mathrm{Ent}(\dP_{|\Lambda_n}|\Poi_{|\Lambda_n}).
\end{align}
 The  proper domain of the specific relative entropy is $\mathcal{D}(\mathcal{E}):=\{\dP\in\PPP_s(\Gamma):\mathcal{E}(\dP)<\infty\}$. It satisfies the following properties.
\begin{lem}[\cite{fekete, serfaty2017microscopicdescriptionlogcoulomb, erbar2018onedimensionalloggasfreeenergy}]\label{lem:entropy}
Let $\dP\in\PPP_s(\Gamma)$.
\begin{enumerate}
      \item[${(i)}$]The limit in \eqref{eq:27} exists in $[0,\infty]$.
    \item[${(ii)}$] The map $\dP\mapsto\mathcal{E}(\dP)$ is affine and lower semi-continuous.
    \item[${(iii)}$] The specific relative entropy vanishes if and only if $\dP = \Poi$.
    \item[${(iv)}$] The limit in the definition of the specific entropy can be replaced by a supremum, i.e.
\begin{align*}
    \mathcal{E}(\dP) = \sup_{n\in\N}\frac{1}{\lambda(\Lambda_n)}\mathrm{Ent}(\dP_{|\Lambda_n}|\Poi_{|\Lambda_n}).
\end{align*}
\end{enumerate}
\end{lem}
The transport distance $\mathcal{W}_s$ and the specific relative entropy $\mathcal{E}$ are related by the following Talagrand inequality. Let us highlight, that this inequality implies that $\mathcal{W}_s$ is finite on $\mathcal{D}(\mathcal{E})$.

    \begin{satz}[Talagrand inequality\label{thm:Talagrand}]
    Let $\dP\in\PPP_s(\Gamma).$ It holds that 
    \begin{align*}
        \mathcal{W}_s^2(\dP,\Poi)\leq \mathcal{E}(\dP).
    \end{align*}
\end{satz}
    \begin{proof}
      Assume that $\mathcal{E}(\dP)<\infty$, since otherwise the claim is empty.  By Theorem \ref{thm:talagrandDSHS} $(i)$, it holds that
        \begin{align*}
    \mathcal{W}_0^2(\dP_{|\Lambda_n},\Poi_{|\Lambda_{n}})\leq\mathrm{Ent}(\dP_{|\Lambda_n}|\Poi_{|\Lambda_{n}})
        \end{align*}
     for all $n\in\N.$   Dividing both sides of the inequality by  $\lambda(\Lambda_n)$, taking the limit, and using Theorem \ref{thm:equalitystationary} yields
        \begin{align*}
            \mathcal{W}_s^2(\dP,\Poi)&= \lim_{n\to\infty} \frac{1}{\lambda(\Lambda_{n})}\mathcal{W}_0^2(\dP_{|\Lambda_n},\Poi_{|\Lambda_n})\leq \lim_{n\to\infty}\frac{1}{\lambda(\Lambda_n)}\mathrm{Ent}(\dP_{|\Lambda_n}|\Poi_{|\Lambda_n})=\mathcal{E}(\dP).
        \end{align*}
    \end{proof}

\subsection{The Evolution Variational Inequality}
In this subsection, we prove that the Ornstein--Uhlenbeck semigroup is the gradient flow of the specific relative entropy w.r.t.\ $\mathcal{W}_s$ in terms of an Evolution Variational Inequality with parameter 1 (EVI(1)). Recall that the Ornstein--Uhlenbeck semigroup $(\mathrm{S}_t)_t$  is defined by
 \begin{align*}
        \mathrm{S}_t^B\dP= \dP^{(e^{-t})}\oplus\Poi_{|B}^{(1-e^{-t})}, \quad \dP\in\PPP_1(\Gamma_B),\ t\geq0.
    \end{align*}
If $B=\Lambda_n,$ we briefly write $\mathrm{S}_t^n$ and if $B=\R^d$, we write $\mathrm{S}_t$. The following lemma shows that the Ornstein--Uhlenbeck semigroup commutes with tiling and stationarizing.
\begin{lem}\label{lem:OUcommutative}
    Let $\dP\in\PPP_1(\Gamma).$ Then 
    \begin{align*}
        (\mathrm{S}_t^n(\dP_{|\Lambda_n}))^{\mathrm{til}}=((\mathrm{S}_t\dP)_{|\Lambda_n})^{\mathrm{til}}= \mathrm{S}_t((\dP_{|\Lambda_n})^{\mathrm{til}})
    \end{align*}
    and
    \begin{align*}
      (\mathrm{S}_t^n(\dP_{|\Lambda_n}))^{\mathrm{stat}}=((\mathrm{S}_t\dP)_{|\Lambda_n})^{\mathrm{stat}}= \mathrm{S}_t((\dP_{|\Lambda_n})^{\mathrm{stat}})
    \end{align*}
    for all $n\in\N.$
\end{lem}
\begin{proof}
    The proof follows directly by the construction of the Ornstein--Uhlenbeck semigroup and the stationarity and independence property of the Poisson point process.
\end{proof}
\begin{bem}
     The stationarity of a measure is preserved under the action of the Ornstein--Uhlenbeck semigroup, i.e.\
     if $\dP\in\PPP_s(\Gamma)$, then $\mathrm{S}_t\dP \in\PPP_s(\Gamma)$. This can be deduced from the facts that the homogeneous $t$-thinning of a stationary point process and the superposition of two stationary point processes are stationary.
   
\end{bem}
\begin{satz}[1-Contractivity of the Ornstein--Uhlenbeck semigroup]
      Let $\dP_0,\dP_1\in \mathcal{D}(\mathcal{E}).$ It holds that
    \begin{align*} \mathcal{W}_s(\mathrm{S}_t\dP_0,\mathrm{S}_t\dP_1)\leq e^{-t}\mathcal{W}_s(\dP_0,\dP_1)
    \end{align*}
   for all $t\geq0$. Furthermore, 
    \begin{align*}
        \mathcal{W}_s(\mathrm{S}_t\dP_0,\Poi)\leq e^{-t}\mathcal{E}(\dP_0)
    \end{align*}
    for all $t\geq0$.
\end{satz}
\begin{proof}
 We will only prove the first inequality, as the proof of the second inequality is similar.   By Theorem \ref{thm:mainthmDSHS} $(i)$ we have that 
 \begin{align*}
     \mathcal{W}_0(\mathrm{S}_t^n({\dP_0}_{|\Lambda_n}),\mathrm{S}_t^n({\dP_1}_{|\Lambda_n}))\leq e^{-t}\mathcal{W}_0({\dP_0}_{|\Lambda_n},{\dP_1}_{|\Lambda_n}).
 \end{align*}
Dividing by the volume of $\Lambda_n,$ taking the limit on both sides, and using Lemma \ref{lem:OUcommutative} yields
\begin{align*}
\lim_{n\to\infty}\frac{1}{\lambda(\Lambda_n)}\mathcal{W}_0(({\mathrm{S}_t\dP_0})_{|\Lambda_n},{(\mathrm{S}_t\dP_1)}_{|\Lambda_n})\leq \lim_{n\to\infty}e^{-t}\frac{1}{\lambda(\Lambda_n)}\mathcal{W}_0({\dP_0}_{|\Lambda_n},{\dP_1}_{|\Lambda_n}). 
\end{align*}
Finally, using Theorem \ref{thm:equalitystationary} proves the claim.
\end{proof}
Let us continue by proving that the Ornstein--Uhlenbeck semigroup is the EVI(1)-gradient flow of the specific relative entropy $\mathcal{E}$ w.r.t.\ $\mathcal{W}_s$.
\begin{satz}[EVI]\label{thm:Evi}
    Let $\dP,\dR\in \mathcal{D}(\mathcal{E})$. Then the following Evolution Variational Inequality holds
    \begin{align*}
        \dt \mathcal{W}_s^2(\mathrm{S}_t \dP,\dR)+\mathcal{W}_s^2(\mathrm{S}_t \dP,\dR)\leq 2(\mathcal{E}(\dR)-\mathcal{E}(\mathrm{S}_t\dP)),\quad  t\geq0. 
    \end{align*}
\end{satz}
 \begin{proof}[Proof of Theorem \ref{thm:Evi}]
    Due to the semigroup property of $\mathrm{S}_t$ it suffices to prove the claim for $t=0.$
        Note that the specific relative entropy is decreasing along the Ornstein--Uhlenbeck semigroup by Jensen´s inequality. Thus, as  $\mathcal{E}(\dP)<\infty$, we note that $\mathcal{E}(\mathrm{S}_t\dP)<\infty$ for all $t\geq0$. By Theorem \ref{thm:Talagrand} we deduce that  $\mathcal{W}_s(\mathrm{S}_t\dP,\dR),\mathcal{W}_s(\dP,\dR)<\infty.$

    Let $\delta>0$ and define $a(\delta)=(1-e^{-2\delta})/2\delta.$ By the proof of \cite[Theorem 5.27]{schiavo2023wasserstein}, we have that 
        \begin{align*}
\mathcal{W}^{2}_0(\mathrm{S}_\delta^n\dP_{|\Lambda_n},\dR_{|\Lambda_n}) - \mathcal{W}^{2}_0(\dP_{|\Lambda_n},\dR_{|\Lambda_n}) &\leq (a(\delta) - 1) \mathcal{W}^{2}_0(\dP_{|\Lambda_n},\dR_{|\Lambda_n}) \\
&\quad+ 2 \delta a(\delta) ( \mathrm{Ent}(\dR_{|\Lambda_n}|\Poi_{|\Lambda_n}) - e^{2\delta} \mathrm{Ent}( \mathrm{S}_{\delta}^n\dP_{|\Lambda_n}|\Poi_{|\Lambda_n} ) ) \\
&\quad+ 4 a(\delta) \delta^{2} \int_{0}^{1} e^{2t\delta} \mathrm{Ent}( \mathrm{S}_{t\delta}^n \dP_{|\Lambda_n} |\Poi_{|\Lambda_n})\mathrm{d} t,
\end{align*}
for all $n\in\N$. Dividing by the volume of $\Lambda_n$ and taking the limit gives 
\begin{align*}
    &\lim_{n\to\infty}\frac{1}{\lambda(\Lambda_n)}\mathcal{W}^{2}_0(\mathrm{S}_\delta^n\dP_{|\Lambda_n},\dR_{|\Lambda_n}) - \lim_{n\to\infty}\frac{1}{\lambda(\Lambda_n)}\mathcal{W}^{2}_0(\dP_{|\Lambda_n},\dR_{|\Lambda_n})\\
    &\leq (a(\delta) - 1) \lim_{n\to\infty}\frac{1}{\lambda(\Lambda_n)}\mathcal{W}^{2}_0(\dP_{|\Lambda_n},\dR_{|\Lambda_n}) \\
&\quad+ \lim_{n\to\infty}\frac{1}{\lambda(\Lambda_n)}2 \delta a(\delta) ( \mathrm{Ent}(\dR_{|\Lambda_n}|\Poi_{|\Lambda_n}) - e^{2\delta} \mathrm{Ent}( \mathrm{S}_{\delta}^n\dP_{|\Lambda_n}|\Poi_{|\Lambda_n} ) ) \\
&\quad+ 4 a(\delta) \delta^{2} \lim_{n\to\infty}\frac{1}{\lambda(\Lambda_n)}\int_{0}^{1} e^{2t\delta} \mathrm{Ent}( \mathrm{S}_{t\delta}^n \dP_{|\Lambda_n} |\Poi_{|\Lambda_n})\mathrm{d} t.
\end{align*}
By Theorem \ref{thm:equalitystationary} and Lemma \ref{lem:OUcommutative}, we note that 
\begin{align*}
    \lim_{n\to\infty}\frac{1}{\lambda(\Lambda_n)}\mathcal{W}^{2}_0(\mathrm{S}_\delta^n\dP_{|\Lambda_n},\dR_{|\Lambda_n})=\lim_{n\to\infty}\frac{1}{\lambda(\Lambda_n)}\mathcal{W}^{2}_0((\mathrm{S}_\delta\dP)_{|\Lambda_n},\dR_{|\Lambda_n})=\mathcal{W}_s^2(\mathrm{S}_\delta\dP,\dR)
    \end{align*}
    and 
    \begin{align*}
    \lim_{n\to\infty}\frac{1}{\lambda(\Lambda_n)}\mathcal{W}^{2}_0(\dP_{|\Lambda_n},\dR_{|\Lambda_n})=\mathcal{W}_s^2(\dP,\dR).
    \end{align*}
    Moreover, using Fatou's Lemma, we obtain
    \begin{align*}
        \lim_{n\to\infty}\frac{1}{\lambda(\Lambda_n)}\int_{0}^{1} e^{2t\delta} \mathrm{Ent}( \mathrm{S}_{t\delta}^n \dP_{|\Lambda_n} |\Poi_{|\Lambda_n})\mathrm{d} t&\leq \limsup_{n\to\infty}\frac{1}{\lambda(\Lambda_n)}\int_{0}^{1} e^{2t\delta} \mathrm{Ent}( \mathrm{S}_{t\delta}^n \dP_{|\Lambda_n} |\Poi_{|\Lambda_n})\mathrm{d} t\\
        &\leq  \int_{0}^{1} e^{2t\delta} \mathcal{E}( \mathrm{S}_{t\delta} \dP)\mathrm{d} t.
    \end{align*}
    Inserting this in the above equation and using the definition of the specific relative entropy yields 
        \begin{align*}
    \mathcal{W}^{2}_s(\mathrm{S}_\delta\dP,\dR) - \mathcal{W}^{2}_s(\dP,\dR) &\leq (a(\delta) - 1) \mathcal{W}^{2}_s(\dP,\dR) \\
&\quad+2 \delta a(\delta) ( \mathcal{E}(\dR) - e^{2\delta} \mathcal{E}( \mathrm{S}_{\delta}\dP  )) \\
&\quad+ 4 a(\delta) \delta^{2} \int_{0}^{1} e^{2t\delta} \mathcal{E}( \mathrm{S}_{t\delta} \dP)\mathrm{d} t.
\end{align*} 
After dividing both sides by $ \delta $ and taking the $ \limsup_{\delta \to 0} $, the lower semi-continuity of $ \mathcal{E} $, combined with its monotonicity along the Ornstein–Uhlenbeck semigroup, yields
    \begin{align*}
        \frac{\mathrm{d}^{+}}{\mathrm{d} t}_{ |{t=0}} \mathcal{W}^{2}_s(\mathrm{S}_t\dP, \dR) \leq -\mathcal{W}^{2}_s(\dP, \dR) + 2 (\mathcal{E}(\dR) - \mathcal{E}(\dP)),
    \end{align*}
    which proves the claim.\end{proof}

The EVI-gradient flow formulation is 
 one of the strongest gradient flow formulations, as it entails a number of consequences for the semigroup and the geometry of the underlying space, see e.g.\ \cite{Daneri_2008}. A direct consequence of this formulation is the $1$-geodesic convexity of the specific relative entropy w.r.t.\ $\mathcal{W}_s$.
\begin{kor}[$1$-Geodesic Convexity of the Specific Relative Entropy]
    Let $\dP_0,\dP_1\in \mathcal{D}(\mathcal{E}).$ For all $\mathcal{W}_s$-geodesics $(\dP_t)_{t\in[0,1]}$ connecting $\dP_0$ and $\dP_1$ it holds that 
    \begin{align*}
        \mathcal{E}(\dP_t)\leq (1-t) \mathcal{E}(\dP_0)+t\mathcal{E}(\dP_1)-\frac{t(1-t)}{2}\mathcal{W}_s^2(\dP_0,\dP_1), \quad 0\leq t\leq 1.
    \end{align*}
    \end{kor}
\begin{proof}
The proof follows directly from \cite[Theorem 3.1]{Daneri_2008} and Theorem \ref{thm:Evi}.\end{proof}
    \subsection{The Specific Modified Fisher Information and the {HWI} Inequality}
        Recall that the modified Fisher information of $\dP\in\PPP_1(\Gamma_B)$ is defined as 
         \begin{align*}
             {\mathrm{I}}(\dP|\Poi_{|B}):=\begin{cases}
        \int D_x\rho(\xi) D_x\log(\rho(\xi))\Poi_{|B}\otimes \m_{|B}(\mathrm{d}\xi,\mathrm{d}x), \quad &\text{if } \dP\ll\Poi_{|B} \text{ with } \dP=\rho \Poi_{|B}\\
        \infty, & \text{otherwise}
    \end{cases}. 
             \end{align*}
            For stationary measures $\dP\in\PPP_s(\Gamma)$,  we define  the \textit{specific modified Fisher information}  by 
             \begin{align}\label{eq:specificFisher}
                 \mathcal{I}(\dP):=\lim_{n\to\infty}\frac{1}{\lambda(\Lambda_n)}{\mathrm{I}}(\dP_{|\Lambda_n}|\Poi_{|\Lambda_n}).
             \end{align}
             It satisfies the following properties.
            \begin{lem}
          Let $\dP\in\PPP_s(\Gamma).$ 
          \begin{itemize}
              \item [$(i)$] The limit in \eqref{eq:specificFisher} exists in $[0,\infty]$ and can be replaced by a supremum, i.e.
\begin{align*}
    \mathcal{I}(\dP) = \sup_{n\in\N}\frac{1}{\lambda(\Lambda_n)}\mathrm{I}(\dP_{|\Lambda_n}|\Poi_{|\Lambda_n}).
    \end{align*}
        \item[$(ii)$] The specific modified Fisher information $\mathcal{I}$ is lower semi-continuous.
          \end{itemize}
            \end{lem}
            \begin{proof}
               $(i)$ The claim follows from Fekete's Lemma (Lemma \ref{lem:fekete}) once we show that its assumptions $(i)$ and $(ii)$  are satisfied.  For the proof of property $(i)$, we note that the function $(x,y)\mapsto x\log(x)+y\log(y)-x\log(y)-y\log(x)$ is convex. Analogously to the proof of Lemma \ref{lem:superadditivity}, we obtain for two disjoint sets $B,D\in\BBB(\R^d)$ that 
                \begin{align*}
                    \mathrm{I}(\dP_{|B}|\Poi_{|B})+\mathrm{I}(\dP_{|D}|\Poi_{|D})\leq \mathrm{I}(\dP_{|B\cup D}|\Poi_{|B\cup D}). 
                \end{align*}
                Property $(ii)$ follows from the stationarity of $\dP,$ which finishes the proof.\\

        $(ii)$ The lower semi-continuity of $\mathcal{I}$ follows directly from the lower semi-continuity of $\mathrm{I}$, see \cite[Theorem 3.2]{schiavo2023wasserstein}.
            \end{proof}
          The modified Fisher information is related to the relative entropy  by  
         Wu's modified logarithmic Sobolev inequality \cite[Corollary 2.2]{Wu}, which directly implies its specific counterpart
         \begin{align} \label{eq:modifiedlogSobolev}
             \mathcal{E}(\dP)\leq \mathcal{I}(\dP), \quad \dP\in\PPP_s(\Gamma).
         \end{align}
         Further, the specific modified Fisher information controls the specific relative entropy along the Ornstein--Uhlenbeck semigroup and the following de Bruijn's identity is satisfied. 
         \begin{satz}[De Bruijn's Identity]
             Let $\dP_0\in\PPP_s(\Gamma)$ s.t.\ $\mathcal{I}(\dP_0)<\infty.$ It holds that 
             \begin{align*}
                \mathcal{E}(\mathrm{S}_t\dP_0)-\mathcal{E}(\dP_0)=-\int_0^t\mathcal{I}(\mathrm{S}_r\dP_0)\mathrm{d}r.
             \end{align*}
         \end{satz}
        \begin{proof}
            We can deduce from Theorem \ref{thm:mainthmDSHS} $(iv)$ that 
            \begin{align*}
            \lim_{n\to\infty}\frac{1}{\lambda(\Lambda_n)}    \mathrm{Ent}(\mathrm{S}_t^n{\dP}_{0|\Lambda_n}|\Poi_{|\Lambda_n})-  \lim_{n\to\infty}\frac{1}{\lambda(\Lambda_n)} \mathrm{Ent}({\dP}_{0|\Lambda_n}|\Poi_{|\Lambda_n})=-  \lim_{n\to\infty}\frac{1}{\lambda(\Lambda_n)} \int_0^t\mathrm{I}(\mathrm{S}_r^n{\dP}_{0|\Lambda_n}|\Poi_{|\Lambda_n})\mathrm{d}r.
            \end{align*}
            The modified Fisher information is decreasing along the Ornstein--Uhlenbeck semigroup by Jensen's inequality. This implies 
            \begin{align*}
               \frac{1}{\lambda(\Lambda_n)} \mathrm{I}(\mathrm{S}_r^n{\dP}_{0|\Lambda_n}|\Poi_{|\Lambda_n}) \leq \sup_{n\in\N} \frac{1}{\lambda(\Lambda_n)} \mathrm{I}({\dP}_{0|\Lambda_n}|\Poi_{|\Lambda_n})=\mathcal{I}(\dP_0)<\infty
            \end{align*}
            Using the dominated convergence theorem and   Lemma \ref{lem:OUcommutative} finishes the proof.
        \end{proof}
        In analogy to the well-known $\mathrm{HWI}$ inequality of Otto and Villani \cite{OtVi00}, which links the relative entropy, the Wasserstein distance, and the Fisher information, we now establish an inequality that connects the specific relative entropy, $\mathcal{W}_s$, and the specific modified Fisher information $\mathcal{I}$.
    \begin{satz}[HWI inequality]
             Let $\dP\in\PPP_s(\Gamma)$ s.t.\ $\mathcal{I}(\dP)<\infty$. Then the following $\mathrm{HWI}$ inequality holds 
             \begin{align*}
                 \mathcal{E}(\dP)\leq \mathcal{W}_s(\dP,\Poi)\sqrt{\mathcal{I}(\dP)}-\frac{1}{2}\mathcal{W}_s^2(\dP,\Poi).
             \end{align*}
         \end{satz}
         \begin{proof}
          Using the specific modified logarithmic Sobolev inequality \eqref{eq:modifiedlogSobolev} and the specific Talagrand inequality (Theorem \ref{thm:Talagrand}), we can deduce that  $\mathcal{E}(\dP)<\infty$ and $\mathcal{W}_s(\dP,\Poi)<\infty.$   By Theorem \ref{thm:mainthmDSHS} $(v)$ we have that 
             \begin{align*}
                 \mathrm{Ent}(\dP_{|\Lambda_n}|\Poi_{|\Lambda_n})\leq \mathcal{W}_0(\dP_{|\Lambda_n}|\Poi_{|\Lambda_n})\sqrt{\mathrm{I}(\dP_{|\Lambda_n}|\Poi_{|\Lambda_n})}-\frac{1}{2}\mathcal{W}_0^2(\dP_{|\Lambda_n}|\Poi_{|\Lambda_n})
             \end{align*}
             for any $n\in\N$.
             This implies 
             \begin{align}
              \nonumber  &\lim_{n\to\infty}\frac{1}{\lambda(\Lambda_n)} \mathrm{Ent}(\dP_{|\Lambda_n}|\Poi_{|\Lambda_n})\leq\\ &\quad\lim_{n\to\infty}\frac{1}{\lambda(\Lambda_n)}\mathcal{W}_0(\dP_{|\Lambda_n}|\Poi_{|\Lambda_n})\sqrt{\mathrm{I}(\dP_{|\Lambda_n}|\Poi_{|\Lambda_n})}-\lim_{n\to\infty}\frac{1}{\lambda(\Lambda_n)}\frac{1}{2}\mathcal{W}_0^2(\dP_{|\Lambda_n}|\Poi_{|\Lambda_n}). \label{eq:100}
             \end{align}
             We can deduce by Theorem \ref{thm:equalitystationary} that 
             \begin{align*}
                 \frac{1}{2}\lim_{n\to\infty} \frac{1}{\lambda(\Lambda_n)}\mathcal{W}_0^2(\dP_{|\Lambda_n},\Poi_{|\Lambda_n})
                &=\frac{1}{2}\mathcal{W}_s^2(\dP,\Poi).
             \end{align*}
             Moreover, by the same theorem, it holds that 
             \begin{align*}
                 \lim_{n\to\infty}\frac{1}{\sqrt{\lambda(\Lambda_n)}}\mathcal{W}_0(\dP_{|\Lambda_n}|\Poi_{|\Lambda_n})= \mathcal{W}_s(\dP,\Poi).
             \end{align*}
            We obtain the claim by inserting these observations in the right-hand side of equation \eqref{eq:100} and using the definition of the specific relative entropy and the specific modified Fisher information.
             \end{proof}
      \appendix
\section{Densities of Point Processes}
\begin{lem}\label{app:productdensities}
   Let $B,D\in\BBB(\R^d)$ be disjoint sets and $\dP\in\PPP_1(\Gamma_{B\cup D})$ s.t.\ $\dP=\dP_{|B}\otimes\dP_{|D}$. Assume that $\dP_{|B}\ll\Poi_{|B}$ and $\dP_{|D}\ll\Poi_{|D}$ with $\dP_{|B}=\rho^B\Poi_{|B}$ and $\dP_{|D}=\rho^D\Poi_{|D}.$ Then $\dP\ll\Poi_{|B\cup D}$ and \begin{align*}
        \dP= \rho^B(\cdot_{|B})\rho^D(\cdot_{|D})\Poi_{|B\cup D}.\end{align*}
    Further, let ${\nuV}\in\MMM(\Gamma_{B\cup D}\times (B\cup D)).$ 
    Assume that  ${\nuV}_{|B}\ll\Poi_{|B}\otimes \m_{|B}$ and ${\nuV}_{|D}\ll\Poi_{|D}\otimes \m_{|D}$ with ${\nuV}_{|B}= w^B(\Poi_{|B}\otimes \m_{|B})$ and ${\nuV}_{|D}= w^D(\Poi_{|D}\otimes \m_{|D}).$ Define $\Tilde{\mathbf{V}}\in\MMM_{b,0}(\Gamma_{B\cup D}\times(B\cup D))$ s.t.\ 
       \begin{align*}
           \int_{\Gamma_{B\cup D}\times(B\cup D)}G(\xi,x)\Tilde{\mathbf{V}}(\mathrm{d}\xi,\mathrm{d}x)&= \int_{\Gamma_D}\int_{\Gamma_{ B}\times B}G(\xi_1+\xi_2,x_1){\nuV}_{|B}(\mathrm{d}\xi_1, \mathrm{d}x_1){\muP}_{|D}(\mathrm{d}\xi_2)\\
           &\quad+ \int_{\Gamma_B}\int_{\Gamma_{ D}\times D}G(\xi_1+\xi_2,x_2){\nuV}_{|D}(\mathrm{d}\xi_1, \mathrm{d}x_2){\muP}_{|B}(\mathrm{d}\xi_2).
        \end{align*}
        Then 
    \begin{align*}
        \Tilde{\mathbf{V}}=(\ind_B(\cdot) w^B(\cdot_{|B},\cdot)\rho^D(\xi_{|D})+\ind_D(\cdot) w^D(\cdot_{|D},\cdot)\rho^B(\cdot_{|B}, \cdot))\Poi_{|B\cup D}\otimes \m.
    \end{align*}
    \end{lem}
\begin{proof}
    Let $\A\in\BBB(\Gamma_{B\cup D})$. 
    Using the superposition theorem for Poisson point processes, we get 
    \begin{align*}
        \dP(\A)&=\int\int\ind_{\A}(\xi_1+\xi_2)\dP_{|B}(\mathrm{d}\xi_1)\dP_{|D}(\mathrm{d}\xi_2)\\
        &=\int\int \ind_{\A}(\xi_1+\xi_2) \rho^B(\xi_1)\rho^D(\xi_2)\Poi_{|B}(\mathrm{d}\xi_1)\Poi_{|D}(\mathrm{d}\xi_2)\\
        &= \int\int \ind_{\A}(\xi) \rho^B(\xi_{|B})\rho^D(\xi_{|D})\Poi_{|B\cup D}(\mathrm{d}\xi).
    \end{align*}
    
    This proves the first claim. Analogously, let $K\times E\in\BBB(\Gamma_{B\cup D}\times (B\cup D)).$ Using again the superposition theorem for Poisson point processes, we get that 
    \begin{align*}
        &\Tilde{\mathbf{V}}(K\times E)=\int \ind_{K\times E}(\xi,z)\Tilde{\mathbf{V}}(\mathrm{d}\xi,\mathrm{d}z)\\
        &= \int\int\int \ind_{K\times E}(\xi_1+\xi_2,z) (\ind_B(z) w^B(\xi_1,z)\rho^D(\xi_2)+\ind_D(z) w^D(\xi_2,z)\rho^B(\xi_1, z) )\Poi_{|B}(\mathrm{d}\xi_1)\Poi_{|D}(\mathrm{d}\xi_2)\m(\mathrm{d}z)\\
        &=\int \ind_{K\times E}(\xi,z) (\ind_B(z) w^B(\xi_{|B},z)\rho^D(\xi_{|D})+\ind_D(z) w^D(\xi_{|D},z)\rho^B(\xi_{|B}, z)) \Poi_{|B\cup D}(\mathrm{d}\xi)\m(\mathrm{d}z),
    \end{align*}
    which finishes the proof.
\end{proof}
\begin{lem}\label{lem:apprestricteddensities}
    Let $B,D\in\BBB(\R^d)$, $B\subseteq D$, and $\dP\in\PPP_1(\Gamma_D).$ Assume that $\dP\ll\Poi_{|D}$ with $\dP=\rho\Poi_{|D}.$ Then $\dP_{|B}\ll\Poi_{|B}$ with \begin{align*}
        \dP_{|B}=\int \rho(\xi_1+\xi_2)\Poi_{|D\setminus B}(\mathrm{d}\xi_2)\Poi_{|B}.
    \end{align*}
    Further, let  ${\nuV}\in\MMM_{b,0}(\Gamma_D\times D)$ s.t.\ ${\nuV}\ll\Poi_{|D}\otimes \m_{|D}$ with ${\nuV}=w(\Poi_{|D}\otimes \m_{|D}).$ 
    Then  ${\nuV}_{|B}\ll \Poi_{|B}\otimes \m_{|B}$ with $${\nuV}_{|B}=\int w(\xi_1+\xi_2,x)\Poi_{|D\setminus B}(\mathrm{d}\xi_2)(\Poi_{|B}\otimes \m_{|B}).$$
\end{lem}
\begin{proof}
    Let  $\A\in\BBB(\Gamma_{A\cup D})$.
    Using the superposition theorem for Poisson point processes, we get
    \begin{align*}
        \dP_{|B}(\A)=\int\ind_{\A}(\xi)\dP_{|B}(\mathrm{d}\xi) &= \int\ind_{\A}(\xi_{|B})\dP(\mathrm{d}\xi)\\
        &= \int\ind_{\A}(\xi_{|B})\rho(\xi)\Poi_{|D}(\mathrm{d}\xi)\\
        &=\int\int\ind_{\A}((\xi_1+\xi_2)_{|B})\rho(\xi_1+\xi_2)\Poi_{|D\setminus B}(\mathrm{d}\xi_1)\Poi_{|B}(\mathrm{d}\xi_2)\\
        &= \int\ind_{\A}({\xi_2})\int\rho(\xi_1+\xi_2)\Poi_{|D\setminus B}(\mathrm{d}\xi_1)\Poi_{|B}(\mathrm{d}\xi_2).
    \end{align*}
    This proves the first claim.

    For the proof of the second claim let $\A\times E\in\BBB(\Gamma_B\times B).$ It holds that 
    \begin{align*}
        {\nuV}_{|B}(\A\times E)&=\int \ind_{\A\times E} (\xi,x){\nuV}_{|B}(\mathrm{d}\xi,\mathrm{d}x)\\&=\int \ind_{\A\times E}(\xi_{|B},x)\ind_{B}(x){\nuV}(\mathrm{d}\xi,\mathrm{d}x)\\
        &= \int \ind_{\A\times E}(\xi_{|B},x)\ind_{B}(x)w(\xi,x)\Poi_{|D}(\mathrm{d}\xi)\m_{|D}(\mathrm{d}x)\\
        &= \int\int \ind_{\A\times E}((\xi_1+\xi_2)_{|B},x)\ind_{B}(x)w(\xi_1+\xi_2,x)\Poi_{|D\setminus B}(\mathrm{d}\xi_1)\Poi_{|B}(\mathrm{d}\xi_2)\m_{|D}(\mathrm{d}x)\\
        &= \int \ind_{\A\times E}((\xi_2)_{|B},x)\ind_{B}(x)\int w(\xi_1+\xi_2,x)\Poi_{|D\setminus B}(\mathrm{d}\xi_1)\Poi_{|B}(\mathrm{d}\xi_2)\m_{|D}(\mathrm{d}x).
    \end{align*}
    This proves the second claim, since sets of the form $\A\times E$ generate $\BBB(\Gamma_B\times B).$
\end{proof}
\begin{lem}\label{app:shiftdensities}
   Let $B\in\BBB(\R^d)$ be Polish, $z\in\R^d$, $\dP\in\PPP_1(\Gamma_{B})$, and ${\nuV}\in\MMM(\Gamma_{B}\times B)$. If  $\dP\ll\Poi_{|B}$, then $(\theta_z)_\#\dP\ll\Poi_{|B-z}$.  Moreover, if  ${\nuV}\ll\Poi_{|B}\otimes \m_{|B}$, then $(\Theta_z^{\Gamma\times \R^d})_\#{\nuV}\ll\Poi_{|B-z}\otimes \m_{|B-z}$
    \end{lem}
 \begin{proof}
     The proof directly follows from the shift-invariance of the Poisson distribution and the Lebesgue measure.
 \end{proof}

\end{document}